\documentclass[a4paper,oneside,reqno]{amsart}

\usepackage[english]{babel}
\usepackage[utf8]{inputenc}		
\usepackage[scaled]{helvet}		
\usepackage{courier}
\usepackage{mathpazo}			
\usepackage{eulervm}			
\usepackage[bb=boondox,bbscaled=1.05,scr=dutchcal]{mathalfa}	%
\usepackage{dsfont}
\normalfont
\usepackage{amsmath}
\usepackage{bm} 

\usepackage{indentfirst}				
\usepackage{tabularx,booktabs}			
\usepackage{caption,subcaption}			
\usepackage{fullpage}				    
\usepackage[english]{varioref}			
\usepackage[dvipsnames]{xcolor}			
\usepackage{hyperref}	
\usepackage{cleveref}				
\usepackage{bookmark}					
\usepackage{textcomp}
\usepackage{rotating,afterpage,pdflscape}
\usepackage{enumerate}

\definecolor{webbrown}{rgb}{0.65, 0.16, 0.16}

\hypersetup{
colorlinks=true, linktocpage=true, pdfstartpage=1, pdfstartview=FitV,breaklinks=true, pdfpagemode=UseNone, pageanchor=true, pdfpagemode=UseOutlines,plainpages=false, bookmarksnumbered, bookmarksopen=true, bookmarksopenlevel=1,hypertexnames=true, pdfhighlight=/O,urlcolor=webbrown, linkcolor=RoyalBlue, citecolor=ForestGreen}

\usepackage{graphicx}		
\usepackage{float}
\usepackage{tikz}			
\usepackage{tikz-cd}		
\usetikzlibrary{
  arrows,decorations.pathreplacing,decorations.markings,shapes.geometric,positioning,calc,fadings,angles
	}
\tikzfading[name=inner fade, inner color=transparent!0, outer color=transparent!100]
\graphicspath{{Images/}}

\usepackage[textsize=footnotesize,textwidth=0.85in]{todonotes}
\presetkeys%
    {todonotes}%
    {color=Dandelion}{}%
\usepackage{amsmath,amssymb,amsthm,mathtools}
\usepackage{bm,braket,faktor,stmaryrd}

\numberwithin{equation}{section}

\newcommand{\hhh}{\text{{\LARGE {\<h>}}}}

\newcommand{\MMM}[2]{\overline{\mathcal{M}}_{#1, #2} \!}
\newcommand{\oM}{\overline{\mathcal{M}}}
\newcommand{\ii}{\mathsf{i}}

\newcommand{\ZZ}{\mathbb{Z}}
\newcommand{\QQ}{\mathbb{Q}}

\newcommand{\Ch}{{ \Omega}}

\newcommand{\DR}{\mathrm{DR}}

\newcommand{\suma}{\mathbf{a}}

\DeclareMathOperator{\newe}{new}
\DeclareMathOperator{\Coeff}{Coeff}
\DeclareMathOperator{\nod}{nod}
\DeclareMathOperator{\Stab}{Stab}
\DeclareMathOperator{\rel}{rel}

\theoremstyle{definition}
\newtheorem{defn}{Definition}[section]
\newtheorem{rem}[defn]{Remark}
\newtheorem{ex}[defn]{Example}
\newtheorem{notation}[defn]{Notation}
\theoremstyle{plain}
\newtheorem{thm}{Theorem}

\newtheorem{prop}[defn]{Proposition}
\newtheorem{conj}{Conjecture}
\newtheorem{lem}[defn]{Lemma}
\newtheorem{cor}[defn]{Corollary}

\usepackage{cjhebrew}
\usepackage{pdftexcmds}

\DeclareFontFamily{U}{cbgreek}{}
\DeclareFontShape{U}{cbgreek}{m}{n}{
        <-6>    grmn0500
        <6-7>   grmn0600
        <7-8>   grmn0700
        <8-9>   grmn0800
        <9-10>  grmn0900
        <10-12> grmn1000
        <12-17> grmn1200
        <17->   grmn1728
      }{}
\DeclareFontShape{U}{cbgreek}{bx}{n}{
        <-6>    grxn0500
        <6-7>   grxn0600
        <7-8>   grxn0700
        <8-9>   grxn0800
        <9-10>  grxn0900
        <10-12> grxn1000
        <12-17> grxn1200
        <17->   grxn1728
      }{}

\makeatletter
\newcommand{\normalorbold}{%
  \ifnum\pdf@strcmp{\math@version}{bold}=\z@ bx\else m\fi
}
\makeatother

\usepackage{array}
\usepackage{arydshln}
\usepackage{array,booktabs}
\newcolumntype{C}[1]{>{\centering\let\newline\\\arraybackslash\hspace{0pt}}m{#1}}

\allowdisplaybreaks

\title{Cohomological representations of quantum tau functions}
\vspace{.3cm}

\author{Xavier Blot}
\address{X.~B.: Korteweg-de Vriesinstituut voor Wiskunde, Universiteit van Amsterdam, Postbus 94248, 1090 GE Amsterdam, Nederland}
\email{x.j.c.v.blot@uva.nl}

\author{Danilo Lewa\'{n}ski}
\address{D.~L.: Universit\`a di Trieste, Dipartimento MIGe, via Valerio 12/1, 34133 Trieste, Italia}
\email{danilo.lewanski@units.it}

\author{Sergey Shadrin}
\address{S.~S.: Korteweg-de Vriesinstituut voor Wiskunde, Universiteit van Amsterdam, Postbus 94248, 1090 GE Amsterdam, Nederland}
\email{s.shadrin@uva.nl}
	
\date{}

\begin{document}



\begin{abstract}
In 2016, Buryak and Rossi introduced the quantum Double Ramification (DR) hierarchies which associate a quantum integrable hierarchy to any Cohomological Field Theory (CohFT). Shortly after, they introduced, in collaboration with Dubrovin and Gu\'er\'e, the quantum tau functions of these hierarchies. In this work, we study  quantum tau functions associated to a specific solution called the topological solution. We provide two cohomological representations for the correlators of these tau functions. The first representation involves an analog in the quantum setting of the $A$-class of the DR-DZ equivalence.  The second representation, valid for CohFT of low degree, involves the so-called $\Omega$-classes. Furthermore, we establish the string and dilaton equations for these tau functions, and present certain vanishing of their correlators.


\end{abstract}

\maketitle

\vspace{1cm}
\tableofcontents

\section{Introduction}\label{sec:intro}

Over the past 30 years, moduli spaces of curves and integrable systems have proven a rich and fruitful interaction, in which the results of each of the two fields complete the understanding of the other. The foundational result establishing this link can be traced back to Witten's conjecture in 1991 \cite{wit-1}, shortly after proved by Kontsevich \cite{kon}. It states that the coefficients of the tau function (of topological type) of the Korteweg--de Vries hierarchy are represented by intersection numbers of the so-called $\psi$-classes in the moduli space of curves. There are now several different proofs that use principally different geometric ideas; the most recent proof that uses the methods close to the ones of the present paper as well as a short overview of other approaches is available in~\cite{AHIS}.

\medskip

Witten-Kontsevich result turned out to be just the tip of an iceberg which generalises to the interaction between two fundamental concepts \textemdash\; on the side of the moduli spaces, the concept of cohomological field theories~\cite{KM94} (in short, CohFTs), which are internally coherent collections of cohomological classes simultaneously defined on all moduli spaces of stable curves $\oM_{g,n}$; on the side of integrable systems, the concept of integrable hierarchies of evolutionary Hamiltonian equations, along with a key object defined within this framework: tau functions. 

\medskip

A first construction of Hamiltonian integrable hierarchy of PDEs associated to semi-simple CohFT was introduced in~\cite{BPS1} as an interpretation and a generalization of the 2001 construction of Dubrovin and Zhang \cite{DZ}. The former hierarchy is such that its tau function (of topological type)  is the potential of the CohFT, that is the correlators are given by intersection numbers of the CohFT with $\psi$-classes over $\oM_{g,n}$.


\medskip

A second construction was given by Buryak in 2014 \cite{Bur}. It associates to a CohFT a Hamiltonian integrable hierarchy of PDEs my means of one 
specific cohomology class: the double ramification cycle. This hierarchy is called the DR hierarchy.
The DR hierarchy and the DZ hierarchy had been conjectured to be (normal Miura) equivalent in \cite{Bur,BDGR1} in the semi-simple case and proved to be so in \cite{DRDZ}. An extension to general non-semisimple setup is also available~\cite{BDGR1,BGR19} and to F-CohFT~\cite{BS22}, and their proofs are based on the argument of \cite{DRDZ} with an additional result obtained in~\cite{BSS}.

\medskip

In the context of the DR hierarchies, the geometrical representation of tau functions (of topological type) is more complicated. In \cite{BDGR1}, the authors represent the correlators of these tau functions associated to the CohFT $c_{g,n}$ by intersection numbers over $\oM_{g,n}$ obtained by intersecting the CohFT $c_{g,n}$ with a certain class that does not depend on the CohFT, called the $A$-class. The $A$-class was initially given as a  sum over a stable trees, decorated with double ramification cycles on all of their vertices. Later, in \cite{DRDZ} the $A$-class was realised geometrically as the virtual fundamental class of the moduli space of stable maps to $\mathbb{P}^1(\mathbb{C})$ further intersected with certain common Chern classes.


\medskip

In 2016 Buryak and Rossi \cite{BR16-quantum} have quantised the DR hierarchy. In this framework, a definition of quantum tau functions was  proposed in \cite{BDGR20}. This definition imitates its classical counterpart, however the significance of this object and the interpretation of its coefficients was unknown. The first study of quantum tau functions was performed in \cite{Blot}, where the author proves that certain coefficients are given in terms of one-part double Hurwitz numbers. The proof was simplified in \cite{Blot-Buryak}.

\medskip

	
The main goal of this paper is to give two cohomological representations for the correlators of the quantum tau functions (of topological type) of the DR hierarchies. 
\begin{enumerate}
	\item A first representation, valid for all CohFTs, involves the virtual fundamental class of the moduli space of stable rubber maps to $\mathbb{P}^1(\mathbb{C})$  with $\psi$-classes pulled back from the Losev-Manin space and intersected with the full Hodge class. This representation generalises in the quantum context the one of \cite{BDGR1}, and the class involved is called the quantum $A$-class. 
	\item The second representation, valid for CohFT of low enough degree, involves a parametrisation of the so-called $\Omega$-class intersected with the full Hodge class. This second representation generalises the representation of \cite{Blot}. More explicitly, the specific numbers studied in \emph{op. cit.} are represented as intersection number of the $\Omega$-class only, which by the orbifold ELSV formula \cite{JPT} through the results of \cite{Do-Lew, LPSZ} provides the one-part double Hurwitz numbers.
\end{enumerate}

\medskip

Comparing these two representations at the classical level leads to the formulation of a conjecture between the classical $A$-class and another class given by a sum over trees decorated with the parametrisation of the $\Omega$-class involved. This conjecture is formulated in \cite{BLRS} and proved in \cite{BLS-Omega}. 

\medskip

We also prove that quantum tau functions satisfy a deformation of the string equation, and a deformation of the dilaton equation. The latter follows from the representation of quantum correlators in terms of the quantum A-class, and resolve a conjecture of \cite{Blot}. Finally, we  present some vanishing results for the correlators of the quantum tau functions.


\medskip

\subsection{Outline of the paper}

In \Cref{sec:QuantumDRhierarchy} we provide the necessary background on the DR hierarchies and their tau functions. In \Cref{sec:Omega:classes} we provide the necessary background on $\Omega$-classes and their specialisations.
In \Cref{sec:main} we state the main results, which we prove in the remaining sections.


%
%
%

\subsection{The topological tau function of the quantum DR hierarchy}
\label{sec:QuantumDRhierarchy}
We introduce the topological tau functions for the quantum DR hierarchies following \cite{BDGR20}. Note that we adopt a different choice for the constant term of the two-point function, as detailed in \Cref{def:two-point}.
\smallskip

\subsubsection{Cohomological Field Theories with unit}

The constructions and results of this paper work for an arbitrary Cohomological Field Theory (CohFT) with unit. We fix the notations for such a CohFT once and for all. %

\begin{notation}
\label{notation CohFT}We fix a triplet $\left(V,\eta,\bm{1}\right)$, where $V$ is a $\mathbb{C}$-vector space of dimension $N$, $\eta$ is a symmetric nondegenerate bilinear form and $\bm{1}$ is a non-zero distinguished element of $V$. We fix a basis $\left(v_{1},\dots,v_{N}\right)$ of $V$ such that $v_{1}=\bm{1}$. We denote $\eta_{\alpha\beta}:=\eta\left(v_{\alpha},v_{\beta}\right)$ and $\eta^{\alpha\beta}:=\eta^{-1}\left(v_{\alpha},v_{\beta}\right)$. In this text, a Greek index will be an integer between $1$ and $N$, furthermore we use the convention of implicit sum over repeated Greek indices (but not over Latin indices). We fix a cohomological field theory with unit $\left(c_{g,n}\right)_{2g-2+n>0}$ associated to $\left(V,\eta,\bm{1}\right)$ and denote by 
\begin{align}
	c_{g,n;k}
\end{align}
	its cohomological degree $2k$. We introduce a formal variable $\mu$ parametrising its cohomological degree and denote  
\begin{align}
	c_{g,n}^{[\mu]} \coloneqq \sum_{k \geq 0} \mu^k c_{g,n;k} \in H^{*}(\overline{\mathcal{M}}_{g,n},\mathbb{Q})[\mu]. 
\end{align}
Notice that $c_{g,n}^{[\mu]}$ only involves even cohomological degrees. We made this choice because in this paper the CohFT only play a role as intersection numbers after being  intersected with classes of even degree. Therefore its odd degree contribution is killed.
\end{notation}

\smallskip

\subsubsection{Differential polynomials and local functionals}

We introduce two rings: the ring of differential polynomials and the ring of local densities. The Hamiltonian densities of the DR hierarchy belong to the former, while the Hamiltonians belong to the latter.
\begin{defn}
The \emph{differential polynomial ring} $\mathcal{A}_{N}$ is $\mathbb{C}\left[\left[u_{0}^{*}\right]\right]\left[u_{>0}^{*}\right]\left[\left[\epsilon,\hbar,\mu\right]\right]$, i.e. the ring of power series in $\epsilon$, $\hbar$ and $\mu$ whose coefficients whose coefficients are polynomials in $u_{k}^{\alpha}$ , for $k>0$ and $1\leq\alpha\leq N$, and power series in $u_{0}^{\alpha}$, for $1\leq\alpha\leq N$. Recall that $N$ is the dimension of $V$ in \Cref{notation CohFT}. We make $\mathcal{A}_{N}$ into a graded algebra by assigning $\deg u_{k}^{\alpha}=k$, $\deg\epsilon=-1$, $\deg\hbar=-2$ and $\deg \mu =0$. 
\end{defn}

In order to quantize the hierarchy, is it natural to express the differential polynomial using another set of variables. Let $\mathcal{B}_{N}:=[p_{>0}^{*}][[p_{\leq0}^{*},\epsilon,\hbar,\mu]][[e^{\ii x},e^{-\ii x}]]$ be the algebra of power series in $e^{\ii x},e^{-\ii x}\epsilon,\hbar,\mu$ and $p_{j}^{\alpha}$ for $j\leq0$ and $1\leq\alpha\leq n$ with coefficients in the complex polynomials in the variables $p_{j}^{\alpha}$ for $j>0$ and $1\leq\alpha\leq n$. Define the morphism $\phi:\mathcal{A}_{N}\hookrightarrow\mathcal{B}_{N}$ defined on the generators as 
\begin{align}
	u_{s}^{\alpha}\rightarrow\sum_{a\in\mathbb{Z}}\left(\ii a\right)^{s}p_{a}^{\alpha}e^{\ii ax}.
\end{align}
Notice that this map is injective. Therefore there are two ways to write a differential polynomial: as an element of $\mathcal{A}_{N}$ using the $u$-variables, or an element of $\phi\left(\mathcal{A}_{N}\right)$ using the $p$-variables.

We then define the operator $\partial_{x}:\mathcal{B}_{N}\rightarrow\mathcal{B}_{N}$ by the usual derivative with respect to $x$. The restriction of this operator to $\phi\left(\mathcal{A}_{N}\right)$ in the $u$-variables is given by $\partial_{x}:\mathcal{A}_{N}\rightarrow\mathcal{A}_{N}$ as $\sum_{i\geq0}u_{i+1}^{\alpha}\frac{\partial}{\partial u_{i}^{\alpha}},$ recalling that the sum over Greek indices is implicit. 

Finally, let $\overline{\mathcal{B}}_{N}:=[p_{>0}^{*}][[p_{\leq0}^{*},\epsilon,\hbar]]$ and define the integration map
\begin{align}
\int\cdot \, dx:\mathcal{B}_{N} & \rightarrow\overline{\mathcal{B}}_{N}\\
f & \rightarrow{\rm \Coeff}_{\left(e^{\ii x}\right)^{0}}\left(f-f\vert_{p_{*}^{*}=0}\right),\nonumber
\end{align}
where ${\rm \Coeff}_{\left(e^{\ii x}\right)^{0}}$ stands for extracting the coefficient of $\left(e^{\ii x}\right)^{0}$ in the element of $\mathcal{B}_{N}$, and $p_{*}^{*}=0$ stands for evaluating $p_{i}^{\alpha}=0$ for $i\geq0$ and $1\leq\alpha\leq N$. We will also use the notation $\overline{f}:=\int fdx$. Notice that the kernel of the restricted map $\int:\mathcal{A}_{N}\rightarrow\overline{\mathcal{B}}_{N}$ is ${\rm Im}\left(\partial_{x}:\mathcal{A}_{N}\rightarrow\mathcal{A}_{N}\right)\oplus\mathbb{C}[[\epsilon,\hbar,\mu]]$. Thus we identify the image $\int\mathcal{A}_{N}dx$ with the quotient $\mathcal{A}_{N}/{\rm Im}\left(\partial_{x}:\mathcal{A}_{N}\rightarrow\mathcal{A}_{N}\right)\oplus\mathbb{C}[[\epsilon,\hbar,\mu]]$. This quotient space is called the space of \emph{local functionals}, we denote it $\overline{\mathcal{A}}_N.$

\smallskip

\subsubsection{Hamiltonian of the DR hierarchies\label{subsec:Hamiltonian-densities}}

The Hamiltonian densities of the DR hierarchy are represented by differential polynomials, whose coefficients are intersection numbers on the moduli space of curves.
 
Fix $g,n\geq0$ such that $2g-2+n>0$. We denote by $\overline{\mathcal{M}}_{g,n}$ the the moduli space of stable curves of genus $g$ with $n$ marked points. Let $\pi:\overline{\mathcal{C}}_{g,n}\rightarrow\overline{\mathcal{M}}_{g,n}$ be the universal curve. Denote by $\omega_{rel}$ the relative cotangent line bundle over $\overline{\mathcal{C}}_{g,n}$. We introduce two cohomology classes and a homology cycle on $\overline{\mathcal{M}}_{g,n}$:
\begin{itemize}
\item the class $\psi_{i}$ is the first Chern class of the line bundle
$\sigma_{1}^{*}\left(\omega_{\rel}\right)$, where $\sigma_{i}:\overline{\mathcal{M}}_{g,n}\rightarrow\overline{\mathcal{C}}_{g,n}$
is the $i$-th section of the universal curve; %
\item the class $\lambda_{j}$ is the $j$-th Chern class of the Hodge bundle
$\pi_{*}\omega_{\rel}$, where $0\leq j\leq g$; %
\item fix a list of integers $\left(a_{1},\dots,a_{n}\right)$ such that
$\sum_{i=1}^{n}a_{i}=0$. We denote by ${\rm DR}_{g}\left(a_{1},\dots,a_{n}\right)\in H_{2\left(2g-3+n\right)}\left(\overline{\mathcal{M}}_{g,n}, \QQ \right)$
the double ramification cycle and refer, for example, to \cite{BSSZ}
for a definition.
\end{itemize}
It is proved in \cite{PZ,Pim} that the double ramification cycle is a degree $2g$ polynomial in the $a_{i}$\textquoteright s. 

\begin{defn}
Fix $d\geq-1$ and $1\leq\alpha\leq N$. The \emph{Hamiltonian density
$H_{d,\alpha}$} \emph{of the quantum DR hierarchy associated to the
CohFT }$\left(c_{g,n}^{[\mu]}\right)_{2g-2+n>0}$ is 
\begin{align}
\label{eq:def-hamilto-density}
H_{d,\alpha} & =\sum_{\substack{g,n\geq0\\
2g+n>0
}
}\frac{\left(\ii \hbar\right)^{g}}{n!}\sum_{a_{1},\dots,a_{n}\in\mathbb{Z}}\\
 & \times\left(\int_{{\rm DR}_{g}\left(0,a_{1},\dots,a_{n},-\sum a_{i}\right)}\psi_{1}^{d+1}\Lambda\left(\frac{-\epsilon^{2}}{i\hbar}\right)c_{g,n+2}^{[\mu]}\left(v_{\alpha},v_{\alpha_{1}},\dots,v_{\alpha_{n}},v_{1}\right)\right)p_{a_{1}}^{\alpha_{1}}\cdots p_{a_{n}}^{\alpha_{n}}e^{\ii x\sum a_{i}}\nonumber,
\end{align}
where $\Lambda(x)=1+x\lambda_1+\cdots+x^{g}\lambda_g$.
\end{defn}

It follows from the polynomiality of the DR cycle that the hamiltonian density $H_{d,\alpha}$ is a differential polynomial written with the $p$-variables. The Hamiltonian associated to this density is $\overline{H}_{\alpha,d}=\int H_{\alpha,d} dx$. 

\smallskip

\subsubsection{The star product\label{subsec:The-star-product}}
The quantum DR hierarchies are obtained as a deformation quantization of their classical counterparts using the following star product.

\begin{defn}
Let $\overline{f},\overline{g}\in \overline{\mathcal{B}}_{N}$. The star product of $\overline{f}$ and $\overline{g}$ is
\begin{align}
	\label{eq:def-star-product}
	\overline{f}\star \overline{g}=\overline{f}\exp\left(\sum_{k>0}\ii \hbar k\eta^{\alpha\beta}\overleftarrow{\frac{\partial}{\partial p_{k}^{\alpha}}}\overrightarrow{\frac{\partial}{\partial p_{-k}^{\beta}}}\right)\overline{g}\in  \overline{\mathcal{B}}_{N},
\end{align}
where the notations $\overleftarrow{\frac{\partial}{\partial p_{k}^\alpha}}$ and $\overrightarrow{\frac{\partial}{\partial p_{-k}^\beta}}$ mean that the derivative acts on the left or on the right, i.e. $\overline{f}\star \overline{g}=\overline{f}\overline{g}+ \sum_{k>0}\ii \hbar k \eta^{\alpha\beta} \frac{\partial \overline{f}}{\partial p_{k}^{\alpha}}\frac{\partial \overline{g}}{\partial p_{-k}^{\beta}}+\sum_{k_{1},k_{2}>0}\frac{\left(\ii \hbar\right)^{2}}{2}k_{1}k_{2} \eta^{\alpha_1\beta_1}\eta^{\alpha_2\beta_2}\frac{\partial^{2}\overline{f}}{\partial p_{k_{1}}^{\alpha_1} \partial p_{k_{2}}^{\alpha_2} }\frac{\partial^{2}\overline{g}}{\partial p_{-k_{1}}^{\beta_1}\partial p_{-k_{2}}^{\beta_2}}+\cdots$.
\end{defn}

The commutator of the star product is
\begin{align}
	[\overline{f},\overline{g}]:=\overline{f}\star \overline{g}-\overline{g}\star \overline{f}. 
\end{align}
One can define by the same expression the star product, and therefore the commutator, between a differential polynomial $f \in \mathcal{A}_N$ and a local functional $\overline{g} \in \overline{\mathcal{A}}_N$. In \cite{BR16-quantum}, they obtain an explicit expression for this commutator in terms of the $u$-variable. As a consequence of this expression, we get the following property.

\begin{prop}[\cite{BR16-quantum}]
Fix a differential polynomial $f \in\mathcal{A}_{N}$ and a local functional $\overline{g}\in\overline{\mathcal{A}}_{N}$. Then, the commutator of the star product $\left[f,\overline{g}\right]$
is a differential polynomial, that is an element of $\mathcal{A}_{N}$.
\end{prop}

\smallskip

\subsubsection{Integrability and tau symmetry of the quantum DR hierarchies}

The quantum DR hierarchies satisfy two fundamental properties.
\begin{prop}
[Integrability, \cite{BR16}]Let $d_{1},d_{2}\geq-1$ dans $1\leq\alpha_{1},\alpha_{2}\leq N$.
We have
\begin{align}
\left[\overline{H}_{d_{1},\alpha_{1}},\overline{H}_{d_{2},\alpha_{2}}\right]=0.
\end{align}
\end{prop}
\begin{prop}
[Tau symmetry, \cite{BDGR20}]Let $d_{1},d_{2}\geq0$ dans $1\leq\alpha_{1},\alpha_{2}\leq N$.
We have 
\begin{align}
\left[H_{d_{1}-1,\alpha_{1}},\overline{H}_{d_{2},\alpha_2}\right]=\left[H_{d_{2}-1,\alpha_2},\overline{H}_{d_{1},\alpha_1}\right].
\end{align}
\end{prop}
Since the quantum DR hierarchies are integrable and tau-symmetric,
we can construct their tau functions.

\smallskip

\subsubsection{Topological tau function of the quantum DR hierarchies\label{subsec:Topological-tau-function}}
\begin{defn}
[Two point function]\label{def:two-point} Let $d_{1},d_{2}\geq0$ and $1\leq\alpha_{1},\alpha_{2}\leq N$. The two point function $\Omega_{\alpha_{1},d_{1};\alpha_{2},d_{2}}$ is the element of $\mathcal{A}_{N}$ defined by
		\begin{align}
		\label{eq:def-two-point}
		\partial_{x}\Omega_{d_{1},\alpha_{1};d_{2},\alpha_{2}}:=\frac{1}{\hbar}\left[H_{d_{1}-1,\alpha_{1}},\overline{H}_{d_{2},\alpha_{2}}\right],
		\end{align}
	and we fix the constant using the formula
		\begin{equation}
		\left.\frac{\partial\Omega_{d_{1},\alpha_{1};\alpha_{2},d_{2}}}{\partial u_{0}^{1}}\right|_{u_{*}^{*}=0} = \left.\Omega_{d_{1}-1,\alpha_{1};d_{2},\alpha_{2}}\right|_{u_{*}^{*}=0} + \left.\Omega_{d_{1},\alpha_{1};d_{2}-1,\alpha_{2}}\right|_{u_{*}^{*}=0},\quad (d_1,d_2) \neq (0,0) \label{eq:constant-two point function}
		\end{equation}
with the convention that $\Omega$ vanishes if one of its index is negative.  
\end{defn}


\begin{lem}
	\label{lem:compatibility}
	The system of Equations~(\ref{eq:constant-two point function}) admits a unique solution. 
\end{lem}

This lemma is proven in \Cref{app: compatibility constant term two point functions}.

\begin{rem}
The integrability condition ensures that $\left[H_{d_{1}-1,\alpha_{1}},\overline{H}_{d_{2},\alpha_{2}}\right]$
is in the image of $\partial_{x}$ so that the two point functions
exist.
\end{rem}
\begin{rem}
The choice of the constant \Cref{eq:constant-two point function} differs from the convention of \cite{BDGR20}. We made this choice so that the topological tau function satisfies the string and dilaton equations.
\end{rem}
\begin{defn}
We call the \emph{quantum double ramification tau function of topological type} the exponential of the generating series
\begin{equation}\label{eq:def:tautoptype}
F\left(t_{*}^{*};\epsilon,\hbar,\mu\right)= \sum_{\substack{g,n,l,k\geq0\\
2g-2+n>0\\
0\leq l\leq g
}
}\sum_{\underset{1\leq\alpha_{1},\dots,\alpha_{n}\leq N}{d_{1},\dots,d_{n}\geq0}}\frac{(-\ii)^{\sum d_i-3g+3-n+k}}{n!}\left\langle \tau_{d_{1},\alpha_{1}}\cdots\tau_{d_{n},\alpha_{n}}\right\rangle _{l,g-l;k}\epsilon^{2l}\hbar^{g-l}\mu^kt_{d_{1}}^{\alpha_{1}}\cdots t_{d_{n}}^{\alpha_{n}},
\end{equation}
where the quantum correlators satisfying the stability condition $2g-2+n> 0 $ are defined by
\end{defn}
\begin{align}
\left\langle \tau_{d_{1},\alpha_{1}}\cdots\tau_{d_{n}, \alpha_{n} }\right\rangle_{l,g-l;k} & \coloneqq \ii^{\sum d_i-3g+3-n+k} \times \nonumber \\
  \times \Coeff_{\epsilon^{2l}\hbar^{g-l}\mu^{k}}  &\left(\frac{1}{\hbar^{n-2}}\left[\cdots\left[\Omega_{d_{1},\alpha_{1};d_{2},\alpha_{2}},\overline{H}_{d_{3},\alpha_{3}}\right],\cdots,\overline{H}_{d_{n},\alpha_{n}}\right]\right)\Big|_{u_{i}^{\alpha}=\delta_{i,1}\delta_{\alpha,1}}, & \text{if }n\geq2;
\\
\left \langle \tau_{d,\alpha}\right\rangle _{l,g-l;k} & \coloneqq \left\langle \tau_{0,1}\tau_{d+1,\alpha}\right\rangle _{l,g-l;k}, & \text{if }g\geq1;
\\
\left\langle \right\rangle _{l,g-l;k} & \coloneqq \frac{1}{2g-2}\left\langle \tau_{1,1}\right\rangle _{l,g-l;k}, & \text{if }g\geq2;
\end{align}
and by zero otherwise. Note that the first definition for $n=2$ defines the $2$-point correlators in terms of the $\Omega$, which in turn allows to define the $1$-point correlators, and hence the $0$-point correlators. The evaluation $u_{i}^{\alpha}=\delta_{i,1}\delta_{\alpha,1}$ means
substituing $u_{1}^{1}=1$ and $u_{j}^{\alpha}=0$ if $\alpha\neq1$
and $j\neq1$ in the differential polynomial ${\rm Coef}_{\epsilon^{2l}\hbar^{g-l}\mu^k}\left(\frac{1}{\hbar^{n-2}}\left[\cdots\left[\Omega_{d_{1},\alpha_{1};d_{2},\alpha_{2}},\overline{H}_{d_{3},\alpha_{3}}\right],\cdots,\overline{H}_{d_{n},\alpha_{n}}\right]\right)$.
\begin{rem}
The integrability condition and the tau symmetry implies that the
correlators are symmetric if we exchange $\tau_{d_{i},\alpha_{i}}$ and $\tau_{d_{j},\alpha_{j}}$
for $1\leq i,j\leq n$. 

\end{rem}
\begin{rem}
The terms $\left\langle \tau_{d,\alpha}\right\rangle _{l,g-l}$ and
$\left\langle \right\rangle _{l,g-l}$ were not defined in \cite{BDGR20}. They are defined so that $F$ satisfies the string and dilaton equations.
\end{rem}
\begin{rem}
	Our choice of the power of $\ii$ in the normalisation of the correlators differs from \cite{Blot} but aligns with \cite{Blot-Buryak} in the case of the trivial CohFT. We prefer this convention because it cancels any power of $\ii$ in the representations of the correlators given in \Cref{thm: quantum A class} and in \Cref{thm:bounded-degree}.
\end{rem}
\begin{rem}
The classical limit of the logarithm of the quantum tau function of topological type, that is the power series
\begin{align}
F\vert_{\hbar=0}
\end{align}
is the logarithm of the topological tau function of the classical
DR hierarchy. For example, if the CohFT is trivial, this is the Witten-Kontsevich series \cite{wit-1,kon}. More generally, for any CohFT, the series $F\vert_{\hbar=0}$ is related to the partition function of the CohFT by the generalized strong DR/DZ conjecture, see \cite{BGR19} and~\cite{BS22}. The conjecture is now proved for semi-simple CohFT in \cite{DRDZ} with methods developed in \cite{BLRS, BLS-Omega}, and these results are extended for all CohFT using an extra result in \cite{BSS}. 
\end{rem}

\medskip

\subsection{Background on \texorpdfstring{$\Omega$}{Omega}-classes}
\label{sec:Omega:classes}

Let $g,n$ be two nonnegative integers such that $2g-2+n>0$. Let $r$ and $s$ be integers such that $r$ is positive, and let $a_1, \ldots, a_n$ be integers satisfying the modular constraint
\begin{equation}
	a_1 + a_2 + \cdots + a_n \equiv (2g-2+n)s \pmod{r}.
\end{equation}
We denote by  $\overline{\mathcal{M}}_{g}^{r,s}(a_1, \ldots, a_n)$ the moduli stack of $r$-spin structures parametrizing $r$-th roots of the line bundle
\begin{equation}
	\omega_{\log}^{\otimes s}\biggl(-\sum_{i=1}^n a_i p_i \biggr),
\end{equation}
where $\omega_{\log} = \omega(\sum_i p_i)$ is the log-canonical bundle. The modular condition guarantees the existence of a $r$th root. Let $\pi \colon \overline{\mathcal{C}}_{g}^{r,s}(a_1, \ldots, a_n) \to \overline{\mathcal{M}}_{g}^{r,s}(a_1, \ldots, a_n)$ be the universal curve, and $\mathcal{L} \to \overline{\mathcal C}_{g}^{r,s}(a_1, \ldots, a_n)$ the universal $r$-th root. 
 In complete analogy with the case of moduli spaces of stable curves, one can define $\psi$-classes. There is moreover a natural forgetful morphism
\begin{equation}
	\epsilon \colon
	\overline{\mathcal{M}}^{r,s}_{g}(a_1, \ldots, a_n)
	\longrightarrow
	\overline{\mathcal{M}}_{g,n}
\end{equation}
which forgets the line bundle. We denote
\begin{equation}\label{eqn:Omega}
	\Ch_{g,n}^{[x]}(r,s;a_1,\dots,a_n)
	=
	\epsilon_{\ast}
	\left(\sum_{i\geq 0} x^i c_i\left( -R^*\pi_* \mathcal L \right) 
	\right)
	\in
	H^{\textup{even}}(\overline{\mathcal{M}}_{g,n}).
\end{equation}



\medskip


\begin{defn} Let $g,n \geq 0$ such that $2g - 2 + n > 0$. Let $a_1, \dots, a_n$ be non-negative integers and let $\suma = a_1 + \dots + a_n$. We define the class:
\begin{align}
\text{\hhh}_{g,n}(a_1, \dots, a_n) &\coloneqq \suma^{1-g} \cdot \Omega^{[\suma]}_{g,n}(\suma, 0; -a_1, \dots, -a_n).
\end{align}
\end{defn}

We recall from \cite{BLS-Omega, GLN, LPSZ} that the $\Omega$-classes satisfy

\begin{enumerate}
	\item $\Ch^{[x]}_{g,n}(r, s; a_1, \dots, a_i + r, \dots, a_n) = \Ch^{[x]}_{g,n}(r,s;a_1, \dots, a_n) \cdot\left( 1 + x\frac{a_i}{r}\psi_i\right)$. In particular we have that 
		\begin{equation}
			\text{\hhh}_{g,n}(a_1, \dots, a_n) = \frac{ \suma^{1-g} \cdot \Omega^{[\suma]}_{g,n}(\suma, 0; \suma-a_1, \dots, \suma-a_n)}{\prod_{i=1}^n(1 - a_i \psi_i)}.
		\end{equation} 
	\item Assume in addition that $0\leq a_1,\dots,a_n\leq r$. Then 
	\begin{equation}
		\Ch^{[x]}_{g,n}(r,s;a_1, \dots, a_n, s) = \pi^{\ast}\Ch^{[x]}_{g,n}(r,s;a_1, \dots, a_n)
	\end{equation}
	for any integer $s$, where $\pi\colon \oM_{g,n+1}\to \oM_{g,n}$ forgets the last marked point. 
\end{enumerate}

\medskip

\subsection{Acknowledgements}

This work is partly a result of the ERC-SyG project, Recursive and Exact New Quantum Theory (ReNewQuantum) which received funding from the European Research Council (ERC) under the European Union's Horizon 2020 research and innovation programme under grant agreement No 810573. 

\smallskip

X.~B.~and S.~S.~are supported by the Dutch Research Council (grant OCENW.M.21.233 ``Quantum tau functions and related topics''). D.~L.~is supported by the University of Trieste, by the INdAM group GNSAGA, and by the INFN within the project MMNLP (APINE).

\smallskip

The authors thank V.~Delecroix, P.~Rossi, A.~Sauvaget, J.~Schmitt, and D.~Zvonkine for useful discussions.

\bigskip

\section{Main statements}\label{sec:main}

\subsection{The quantum \texorpdfstring{$A$}{A}-class and quantum correlators}
Let $g$ be a nonnegative integers, and let $n,a_1, \dots, a_n$ be positive integers. Consider the moduli space 
\begin{align}
	\overline{\mathcal{M}}_{g}^{\sim}(\mathbb{P}^1,a_1,\dots,a_n,-\suma), \qquad \suma \coloneqq \sum_{i=1}^n a_i
\end{align}
of rubber stable maps to $(\mathbb{P}^1,0,\infty)$, where the profile over $0$ consists of the positive entries of the list $(a_1, \dots, a_n, -\suma)$, and the profile over $\infty$ consists of the opposite of the negative entries. Its projection onto the source curve is denoted by 
\begin{align}
	s: \overline{\mathcal{M}}_{g}^{\sim}(\mathbb{P}^1,a_1,\dots,a_n,-\suma) \to \oM_{g,n+1}.
\end{align}
Let $\lambda_i$ for $i=0,\dots,g$, the lifts of the $\lambda$-classes with respect to the projection $s$ and let 
\begin{align}
\Lambda(t) \coloneqq 1 + t\lambda_1 + t^2\lambda_2 + \dots + t^g\lambda_g
\end{align}
Consider the projection on the target curve
\begin{equation}
t\colon \oM_{g}^{\sim}(\mathbb{P}^1,a_1,\dots,a_n,-\suma)\to LM_{2g-1+n},
\end{equation}
where $LM_{m}$ denotes the Losev-Manin space with $m$ marked points. Let $\tilde \psi_0$ be the pull-back by $t$ of the $\psi$-class at the point $0$ in the cohomology of the Losev-Manin space. Let $\pi \colon \oM_{g,n+1} \to \oM_{g,n}$ be the map that forgets the last marked point. 
\begin{defn} \label{def:QuantumA}
We define the quantum $A$-class as
\begin{align} \label{eq:NewA1}
	A_{g,n}(a_1,\dots,a_n)\coloneqq \frac{1}{\suma} \pi_* s_* \left(\frac{\hbar^g\Lambda\Big(\frac{1}{\hbar}\Big)}{1-\tilde\psi_0} \left[\oM_{g}^{\sim}\left(\mathbb{P}^1,a_1,\dots,a_n,-\suma \right)\right]^{\mathrm{vir}}\right) \in R^*(\oM_{g,n}).
\end{align}
\end{defn}

\begin{rem}
	The class $A_{g,n}(a_1,\dots,a_n)\vert_{\hbar=0} $ is the $A$-class of the recently established DR-DZ equivalence, see \cite{BDGR20,BGR19,DRDZ,BSS}.
\end{rem}

The first result concerns polynomial properites of the quantum $A$-class. 

\begin{prop}
\label{thm:polynomiality-A}
Let $d\geq0$. Suppose that $a_1,\dots,a_n$ are positive integers, then the class
\begin{equation}
	s_{*}\left(\tilde{\psi}_{0}^{d}\left[\overline{\mathcal{M}}_{g}^{\sim}\left(\mathbb{P}^{1},a_{1},\dots,a_{n},-\suma\right)\right]\right)
\end{equation}
is represented by a polynomial in the variables $a_{1},\dots,a_{n}$  with coefficients in the ring $H^{*}\left(\overline{\mathcal{M}}_{g,n+1}\right)$. This polynomial is of degree $2g+d$, and it is odd or even depending on its degree. Moreover, its push-forward by $\pi:\overline{\mathcal{M}}_{g,n+1}\rightarrow\overline{\mathcal{M}}_{g,n}$, forgetting the last marked point, is divisible by $\suma$. 
\end{prop}

This proposition is proven in \Cref{sec:polyA}. We have the following formula for the correlators of the quantum double ramification hierarchy:

\begin{thm}\label{thm: quantum A class} Let $g,n,k$ be nonnegative integers such that $2g - 2 + n > 0$. Let $(c_{g,n}^{[\mu]})_{2g-2+n>0}$ be a CohFT with unit as in \Cref{notation CohFT}. Let $d_1,\dots,d_n \geq0$ and $0\leq \alpha_1,\dots,\alpha_n \leq N$. We have:
	\begin{align}
 \left\langle \tau_{d_{1},\alpha_{1}}\cdots\tau_{d_{n},\alpha_{n}}\right\rangle _{l,g-l;k}  = \Coeff_{a_{1}^{d_{1}}\cdots a_{n}^{d_{n}}\hbar^{g-l}}\int_{\overline{\mathcal{M}}_{g,n}} A_{g,n}(a_1,\dots,a_n)   c_{g,n;k}\left(v_{\alpha_{1}},\dots,v_{\alpha_{n}}\right)
\end{align}
\end{thm}

This theorem is proven in \Cref{sec:proof-A-class}.

\medskip

\subsection{Quantum correlators and the \texorpdfstring{\hhh}{Hej}-class} 
We present an expression for the quantum correlators in terms of the $\hhh$-class for the class of cohomological field theories with unit $(c_{g,n}^{[\mu]})_{2g-2+n>0}$ satisfying the degree condition
\begin{align}\label{eq:DegreeAssumption}
	\deg c_{g,n;k}(v_{\alpha_1} \otimes \dots \otimes v_{\alpha_n})  < g-1+n,
\end{align} 
for all tuples $\alpha_1,\dots,\alpha_n$ and all $k\geq 0$.


\begin{thm} \label{thm:bounded-degree}
	Fix a CohFT with unit as in \Cref{notation CohFT} satisfying the degree condition~\eqref{eq:DegreeAssumption}. Let $g,n,k$ be nonnegative integers such that $2g - 2 + n > 0$.  Let $d_1,\dots,d_n \geq0$ and $0\leq \alpha_1,\dots,\alpha_n \leq N$. We have:
	\begin{itemize}
		\item The integrals $\int_{\overline{\mathcal{M}}_{g,n}} \lambda_l\text{\hhh}_{g,n}(a_1, \dots, a_n)  c_{g,n;k}(v_{\alpha_1} \otimes \dots \otimes v_{\alpha_n})$ are polynomials in $a_1,\dots,a_n$.
		\item The following formula holds:
		\begin{align}
			\langle \tau_{d_{1},\alpha_1}&\dots\tau_{d_{n},\alpha_n}\rangle_{l,g-l;k}
			=
			\mathrm{\Coeff}_{a_{1}^{d_{1}}\cdots a_{n}^{d_{n}}}\int_{\overline{\mathcal{M}}_{g,n}} \lambda_l\text{\hhh}_{g,n}(a_1, \dots, a_n)  c_{g,n;k}(v_{\alpha_1} \otimes \dots \otimes v_{\alpha_n}). \label{eq:main-theorem}
		\end{align}
	\end{itemize}
\end{thm}

This theorem is proven in \Cref{sec:ProofOf-HEJ-Expression}. 

\begin{rem}
	The argument of \Cref{sec:ProofOf-HEJ-Expression} shows that \Cref{eq:main-theorem} is actually satisfied by CohFT satisfying milder condition than the degree assumption (\ref{eq:DegreeAssumption}), see \Cref{cor:VanishingExtraTermsUnderAssumption}.
\end{rem}

\begin{rem} Note that we don't claim the polynomiality of the class $\text{\hhh}_{g,n}(a_1, \dots, a_n)$ in $a_i$'s, though, taking into account the first  statement of \Cref{thm:bounded-degree}, it is a very plausible conjecture.
\end{rem}

An immediate corollary of \Cref{thm:bounded-degree} is the following expression for the quantum Witten-Kontsevich correlators, that is  the quantum correlators for the trivial CohFT, previously studied in~\cite{Blot}. 

\begin{cor} \label{cor:qWK} For the trivial CohFT $c_{g,n} \equiv 1 \in R^0(\oM_{g,n}, \QQ)$, we have 	
\begin{align}
	\label{eq:qWK}
		\langle \tau_{d_{1}}&\dots\tau_{d_{n}}\rangle_{l,g-l}
		=
		\Coeff_{a_{1}^{d_{1}}\cdots a_{n}^{d_{n}}}\int_{\overline{\mathcal{M}}_{g,n}} \lambda_{l} \text{\hhh}_{g,n}(a_1, \dots, a_n)
	\end{align}
for any $g,n\geq 0$, $l=0,\dots,g$ and $d_1,\dots,d_n\geq 0$. In particular, the classical correlators (for $l=g$) give a new expression for the intersection of monomial of $\psi$-classes: 
\begin{equation}
\int_{\oM_{g,n}} \frac{1}{\prod_{i=1}^n (1 - a_i \psi_i)} =  \int_{\oM_{g,n}} \lambda_g \hhh_{g,n}(a_1, \dots, a_n). 
\end{equation}
\end{cor}

\begin{proof}
	The first statement is straightforward. As for the second, we have that $F\vert_{\hbar=0}$ is the logarithm of the tau function associated to the topological solution of the classical DR hierarchy for the trivial CohFT. In this case, the DR hierarchy is KdV, therefore  $F\vert_{\hbar=0}$ is the Witten-Kontsevich series. 
\end{proof}

\begin{rem}
	When $l=0$, the expression of the quantum Witten-Kontsevich correlators given in \Cref{eq:qWK} only involve intersection with the $\hhh$-class. In particular, we recover the main theorem of \cite{Blot} using the relation between these integrals and one-part Hurwitz numbers provided by \cite[Theorem~3.1]{Do-Lew}, which is a specialization of \cite[Theorem~1]{JPT}. The original motivation for the present paper was to generalise this statement for $l \geq 1$. 
\end{rem}

\medskip

\subsection{Level structure for the correlators}
Fix $g,l,k \geq 0$ such that $l\leq g$. The representation of the correlators given by \Cref{thm: quantum A class} combined with the polynomial behaviours of the quantum $A$-class given in \Cref{thm:polynomiality-A} implies that the correlator $\left\langle \tau_{d_{1},\alpha_{1}}\cdots\tau_{d_{n},\alpha_{n}}\right\rangle _{l,g-l;k}$ can be nonzero only if  
\begin{equation}
	\sum d_i \leq 4g-3+n-l-k,\qquad and \qquad \sum d_i\equiv 4g-3+n-l-k \mod 2.
\end{equation}	
Each allowed value for $\sum d_i$ is called a level. We conjecture that there exists a minimal level 
and give an explicit formula for this minimal level. 
\begin{conj}
	\label{conj:minimal-level}
	Fix $g,l,k \geq 0$. The correlator $\left\langle \tau_{d_{1},\alpha_{1}}\cdots\tau_{d_{n},\alpha_{n}}\right\rangle _{l,g-l;k}$ vanishes if $\sum d_i < 2g-3+n-l-k$. Therefore, $\sum d_i$ can be non zero if it is equal to one of the $g+1$ levels
	\begin{equation}
	2g-3+n-l-k+2i,\qquad 0\leq i \leq g.
	\end{equation}
	  Moreover, in the minimal case $\sum d_i = 2g-3+n-l-k$, we have
	\begin{equation}
\left\langle \tau_{d_{1},\alpha_{1}}\cdots\tau_{d_{n},\alpha_{n}}\right\rangle _{l,g-l;k} = \int_{\overline{\mathcal{M}}_{g,n}}  \lambda_g \lambda_l \psi_1^{d_1}\cdots\psi_n^{d_n}  c_{g,n;k}\left(v_{\alpha_{1}},\dots,v_{\alpha_{n}}\right).
\end{equation}
\end{conj}

\begin{rem}
	In the classical setting, the cohomological degree $d$ of the class $A_{g,n}(a_1,\dots,a_n)\vert_{\hbar=0} $ is a homogenous polynomial of degree $d$. Therefore, in the classical setting, there is only one level which is the top level. 
\end{rem}

\begin{rem}
	\Cref{conj:minimal-level} holds for the trivial CohFT when $l=0$ see \cite[Remark~1.39]{Blot}.
\end{rem}

\medskip

\subsection{String and dilaton} We have:

\begin{thm} \label{thm:StringANDDilaton}
Fix a CohFT with unit as in \Cref{notation CohFT}.The quantum correlators satisfy the string equation
\begin{align}
	\label{eq:STRING} \frac{\partial F}{\partial t_{0}^{1}}=\sum_{d\geq 0}t_{d+1}^{\alpha}\frac{\partial F}{\partial t_{d}^{\alpha}}+\frac{t_{0}^{\alpha}t_{0}^{\beta}}{2}\eta_{\alpha\beta}-\ii\hbar\frac{N}{24},
\end{align}
 and the dilaton equation
 \begin{align}
 	\label{eq:DILATON}  
 	\frac{\partial F}{\partial t_{1}^{1}}=\left( \sum_{d\geq0}t_{d}^{\alpha}\frac{\partial }{\partial t_{d}^{\alpha}}+\epsilon\frac{\partial }{\partial\epsilon}+2\hbar\frac{\partial }{\partial\hbar}-2\right)F+\epsilon^{2}\frac{N}{24}-\ii\hbar\int_{\overline{\mathcal{M}}_{1,1}}c_{1,1}^{[\mu]}\left(1\right). 
 \end{align}

\end{thm}

\begin{rem} With respect to \Cref{thm: quantum A class}, the string equation \eqref{eq:STRING} and the dilaton equation \eqref{eq:DILATON} have different status:
\begin{itemize}
	\item We prove the string equation right from the definition and it is a part of the proof of \Cref{thm: quantum A class}. See \Cref{sec:proof-string} below.
	\item The dilaton equation is a corollary of  \Cref{thm: quantum A class}. We prove it in \Cref{sec:ProofDilaton} below. 
\end{itemize}
\end{rem}

\begin{rem} The dilaton equation for the trivial CohFT was conjectured in~\cite[Conjecture 3]{Blot}. Specialising \Cref{thm:StringANDDilaton} to the trivial CohFT proves this conjecture. 
\end{rem}

\medskip

\subsection{Push-forward of the double ramification cycles}
We now present the last statement, which, although not being directly related to the quantum tau function, is a highly useful property of DR cycles. This property extends \cite[Lemma~5.1]{BDGR1} from the moduli space of compact type curves to the moduli space of stable curves.

\begin{prop} \label{prop:DR-pushforward-divisible-bb} We have 
\begin{equation}
\pi_*DR_{g} \left(a_1,\dots,a_n,-\suma-b,b\right) \in b^2 R^g(\overline{\mathcal{M}}_{g,n+1}, \mathbb{Q})[a_1,\dots,a_n,b], \qquad \qquad \suma \coloneqq \sum_i a_i,
\end{equation}
where $\pi:\overline{\mathcal{M}}_{g,n+2}\rightarrow\overline{\mathcal{M}}_{g,n+1}$ forgets the last marked point. In other words, each coefficient of the class above expressed as tautological class, after possible simplifications, is divisible by $b^2$.
\end{prop}
This proposition is proven in \Cref{sec:push-forward-DR}. It has the following direct corollary.

\begin{cor}
\label{cor: push forward DR numbers} Fix a CohFT with unit as in \Cref{notation CohFT}. Fix $g,n,l,d\geq0$ such that $\left(g,n\right)\neq\left(0,1\right)$. We have 
\begin{align}
 & \int_{{\rm DR}_{g}\left(a_{1},\dots,a_{n},-\sum a_{i}-b,b\right)}\psi_{1}^{d}\lambda_{l}c_{g,n+2}\left(\otimes_{i=1}^{n+1}v_{\alpha_{i}}\otimes v_{1}\right)=\\
 &\qquad \qquad \qquad \qquad =\begin{cases}
\int_{{\rm DR}_{g}\left(a_{1}+b,\dots,a_{n},-\sum a_{i}-b\right)}\psi_{1}^{d-1}\lambda_{l}c_{g,n+2}\left(\otimes_{i=1}^{n+1}v_{\alpha_{i}}\right)+O\left(b^{2}\right) & {\rm if}\;d\geq1,\\
O\left(b^{2}\right) & {\rm if}\;d=0. \nonumber
\end{cases}
\end{align}

\end{cor}


\bigskip
\section{Proof of the push-forward property of the DR cycles}
\label{sec:push-forward-DR}

In this section we prove \Cref{prop:DR-pushforward-divisible-bb}. We use the formula of Janda et al.~\cite{JPPZ} in combination with the Zagier formula through the spanning trees~\cite{PZ}. Let us first recall the formula for $DR_{g}(a_1,\dots,a_n)$, $\sum_{i=1}^n a_i = 0$, and then we adapt it to our choice of parameters. 
	\begin{align} \label{eq:DR-Pixton-Zagier}
		& DR_{g}(a_1,\dots,a_n) = \\ \notag
		& \frac 1{2^g} \sum_{\Gamma\in \mathrm{SG}_{g,n}} \frac{(-2)^{|E(\Gamma)|} }{|\mathrm{Aut}(\Gamma)|}\sum_{ \substack {d\colon H(\Gamma)  \to\ZZ_{\geq 0}} }
		\Bigg( (\mathsf{b}_\Gamma)_* \bigg(\prod_{i=1}^n \psi_{\ell_i}^{d({\ell_i})} \prod_{e\in E(\Gamma)} \psi_{h(e)}^{d({h(e)})} \psi_{t(e)}^{d({t(e)})} \bigg) \Bigg)_g
		\\ \notag & \qquad \qquad 
		\times \prod_{i=1}^n \frac{a_i^{ 2d({\ell_i})}}{d({\ell_i})!} \prod_{e\in E(\Gamma)} \frac{(2d({h(e)})+2d({t(e)})+1)!}{d({h(e)})!d({t(e)})!}
		\\ \notag & \qquad \qquad 
		\times \mathrm{\Coeff}_{\left(\prod_{e\in E(\Gamma)} x_e^{2d({h(e)})+2d({t(e)})+2}\right)} 
		\sum_{T \in \mathrm{SpTr}(\Gamma)} \prod_{f\in E(T)} e^{a_{f,T} x_f} \prod_{f\not\in E(T)} \frac{x_f}{e^{x_{f,T}}-1}.
	\end{align}
	Let us list the notation that we use here (for the convenience of the reader we also partially recall the standard conventions).
	\begin{itemize}
		\item The set $\mathrm{SG}_{g,n}$ is the set of stable graphs of genus $g$ with $n$ leaves.
		\item The first sum is over the stable graphs $\Gamma$, with the set of leaves $L(\Gamma)$ and edges $E(\Gamma)$. The leaves are labeled by the bijection $\ell: \{1,\dots,n\} \rightarrow L(\Gamma)$, $i\mapsto \ell_i$. 
		\item $|\mathrm{Aut}(\Gamma)|$ denotes the order of the automorphism group of $\Gamma$. 
		\item For each stable graph $\Gamma$ we choose and fix an orientation on its edges, in an arbitrary way. The ingredients of the formula depend on this choice, but the resulting expression does not (it is part of the statement of~\cite{PZ}). 
		\item Let $H(\Gamma)$ be the set of half-edges of $\Gamma$. Using orientations on the edges we define $h\colon E(\Gamma)\hookrightarrow H(\Gamma)$ and $t\colon E(\Gamma)\hookrightarrow H(\Gamma)$ (heads and tails), and we have $H(\Gamma) = L(\Gamma)  \sqcup h(E(\Gamma)) \sqcup t(E(\Gamma))$. 
		\item The second sum is over all maps $d\colon H(\Gamma)\to\ZZ_{\geq 0}$.
		\item Let $V(\Gamma)$ be the set of vertices of $\Gamma$. 
		\item We associate to each vertex $v\in V(\Gamma)$ the moduli space of curves $\MMM {g(v)} {n(v)}$, where $n(v)$ is the number of the incident half-edges. By $\psi_h$, where $h \in H(\Gamma)$ and is attached to $v$, we mean the $\psi$-class on  $\MMM {g(v)} {n(v)}$ at the marked point corresponding to $h$. We put in correspondence with each vertex $v\in V(\Gamma)$ a monomial of $\psi$-classes in the tautological ring of $\MMM {g(v)} {n(v)}$.
		\item By $\mathsf{b}_\Gamma$ we mean the boundary morphism map to the tautological ring $\MMM{g}{n}$ associated to $\Gamma$. 
		\item For a class $C\in R^*(\MMM gn)$ let $(C)_g$ denote its homogeneous component of degree $g$. 
		\item The operation $\mathrm{\Coeff}_{\prod_{i=1}^k x_i^{p_i}}$ extracts the coefficient of $\prod_{i=1}^k x_i^{p_i}$ in the Laurent series to its immediate right. 
		\item The sum over $T \in \mathrm{SpTr}(\Gamma)$ is the sum over all spanning trees $T$. 
		\item $x_e$'s are the formal variables associated to the edges of $\Gamma$. 
		\item $a_{f,T}$ is defined as the sum of $a_i$'s over all leaves attached to the vertices that are ahead of $f$ in $T$ with respect to the orientation of $f$. 
		\item $x_{f,T}$ is the sum of $\pm x_e$ over all edges $e$ that enter the unique cycle formed by $f$ and edges in $T$. The sign is `$+$' if the orientation of $e$ agrees with the orientation of $f$ in the cycle, and `$-$' otherwise. 
	\end{itemize}
	Note that the summands in the sum over $T \in \mathrm{SpTr}(\Gamma)$ are not formal power series in $x_e$, $e\in E(\Gamma)$; however, the whole sum is. In order to work with this formula term-wise one has to consider the formal Laurent expansions of the factors in some sector (fixed to be the same for all terms) and then select the coefficients of the resulting Laurent series. 
	
	Now we apply Equation~\eqref{eq:DR-Pixton-Zagier} to $DR_{g}(a_1,\dots,a_n,-A-b,b)$ and take the push-forward with respect to $\pi\colon \MMM{g}{n+2}\to \MMM g{n+1}$ that forgets the last marked point. We manifestly get a polynomial in $a_1,\dots,a_n,b$. 
	Our next goal is to analyze this polynomial up to $O(b^2)$ terms (that is, the terms lying in the ideal $(b^2)
	\subset \QQ[a_1,\dots,a_n,b]$). 
	
	Note that once we have a non-trivial degree of $\psi_{n+2}$ in the formula $DR_{g}(a_1,\dots,a_n,-A-b,b)$, the coefficient is proportional to $b^2$. Hence, 
	\begin{align} \label{eq:DR-Pixton-Zagier-pushforward}
		& \pi_*DR_{g}(a_1,\dots,a_n,-A-b,b) = O(b^2) + \\ \notag
		& \frac 1{2^g} \sum_{\Gamma\in \mathrm{SG}_{g,n+2}} \frac{(-2)^{|E(\Gamma)|} }{|\mathrm{Aut}(\Gamma)|}\sum_{ \substack {d\colon H(\Gamma)  \to\ZZ_{\geq 0}} }
		\Bigg( \pi_* (\mathsf{b}_\Gamma)_* \bigg(\prod_{i=1}^{n+1} \psi_{\ell_i}^{d({\ell_i})} \prod_{e\in E(\Gamma)} \psi_{h(e)}^{d({h(e)})} \psi_{t(e)}^{d({t(e)})} \bigg) \Bigg)_g
		\\ \notag & \qquad \qquad 
		\times \prod_{i=1}^n \frac{a_i^{ 2d({\ell_i})}}{d({\ell_i})!} \cdot \frac{(-A-b)^{2d(\ell_{n+1})} }{d({\ell_{n+1}})!}\cdot 
		\prod_{e\in E(\Gamma)} \frac{(2d({h(e)})+2d({t(e)})+1)!}{d({h(e)})!d({t(e)})!}
		\\ \notag & \qquad \qquad 
		\times \mathrm{\Coeff}_{\left(\prod_{e\in E(\Gamma)} x_e^{2d({h(e)})+2d({t(e)})+2}\right)} 
		\sum_{T \in \mathrm{SpTr}(\Gamma)} \prod_{f\in E(T)} e^{a_{f,T} x_f} \prod_{f\not\in E(T)} \frac{x_f}{e^{x_{f,T}}-1}.
	\end{align}
	Since $d(\ell_{n+2})=0$, this push-forward gives a similar expression in terms of stable graphs decorated only by $\psi$-classes, with the coefficients $C_\Gamma (a_1,\dots,a_n,b)$ polynomial in $a_1,\dots,a_n,b$:
	\begin{align} \label{eq:DR-Pixton-Zagier-pushforward-result}
		& \pi_*DR_{g}(a_1,\dots,a_n,-A-b,b) = O(b^2) + \\ \notag
		& \frac 1{2^g} \sum_{\Gamma\in \mathrm{SG}_{g,n+1}} \frac{(-2)^{|E(\Gamma)|} }{|\mathrm{Aut}(\Gamma)|}\sum_{ \substack {d\colon H(\Gamma)  \to\ZZ_{\geq 0}} }
		\Bigg( (\mathsf{b}_\Gamma)_* \bigg(\prod_{i=1}^{n+1} \psi_{\ell_i}^{d({\ell_i})} \prod_{e\in E(\Gamma)} \psi_{h(e)}^{d({h(e)})} \psi_{t(e)}^{d({t(e)})} \bigg) \Bigg)_{g-1}
		\\ \notag & \qquad \qquad 
		\times C_\Gamma (a_1,\dots,a_n,b).
	\end{align}
	
	There are $n+1+|E(\Gamma)|+|V(\Gamma)|$ graphs $\Gamma'$ in $\mathrm{SG}_{g,n+2}$ that under the push-forward give $\Gamma$:
	\begin{enumerate}
		\item either we replace the leaf $\ell_i$ by a new edge $e_{\newe}$ that connects the vertex where $\ell_i$ was attached with a new vertex of genus $0$, where we further attach $\ell_i$ and $\ell_{n+2}$, $i=1,\dots,n+1$;
		\item or we replace an edge $e\in E(\Gamma)$ by two new edges $e_{\newe,1}$ and $e_{\newe,2}$ that are further attached to a vertex of genus $0$ with the leaf $\ell_{n+2}$;
		\item or we just attach the leaf $\ell_{n+2}$ to one of the vertices of $\Gamma$. 
	\end{enumerate}
	In all cases, a choice of orientation on the edges of $\Gamma$ can be canonically lifted to a choice of orientation of the edges of $\Gamma'$:
	\begin{enumerate}
		\item In the first case, the new edge $e_{\newe}$ is oriented toward the new vertex, that is, $t(e_{\newe})$ is naturally identified with $\ell_i$ in $\Gamma$ and $h(e_{\newe})$ is attached to the new vertex of genus $0$. All other edges of $\Gamma'$ are naturally identified with the edges in $\Gamma$ and inherit their orientation.
		\item In the second case, the two new edges $e_{\newe,1}$ and $e_{\newe,2}$ in $E(\Gamma')$ that replace the edge $e\in E(\Gamma)$ are oriented in the same direction as $e$, that is, $h(e_{\newe,1})$ is naturally identified with $h(e)$, $t(e_{\newe,2})$ is naturally identified with $t(e)$ and $t(e_{\newe,1})$ and $h(e_{\newe,2})$ are attached to the new vertex of genus $0$. All other edges of $\Gamma'$ are naturally identified with the edges in $\Gamma$ and inherit their orientation.
		\item In the third case, all edges of $\Gamma'$ are naturally identified with the edges in $\Gamma$ and inherit their orientation.
	\end{enumerate}
	For each $d\colon H(\Gamma) \to \ZZ_{\geq 0}$ there is a number of choices of possible functions $d'$ on $H(\Gamma')$:
	\begin{enumerate}
		\item In the first case, the only possible choice is $d'(t(e_{\newe})) = d(\ell_i)$ and $d'(h(e_{\newe})) = d'(\ell_i) = d'(\ell_{n+2}) = 0$. All other half-edges of $\Gamma'$ are naturally identified with the half-edges of $\Gamma$ and the value of $d'$ coincides with the corresponding value of $d$. 
		\item In the second case, the only possible choice is $d'(h(e_{\newe,1})) = d(h(e))$, $d'(h(e_{\newe,1})) = d(h(e))$, and $d'(t(e_{\newe,1})) = d'(h(e_{\newe,2})) = d'(\ell_{n+2})=0$. All other half-edges of $\Gamma'$ are naturally identified with the half-edges of $\Gamma$ and the value of $d'$ coincides with the corresponding value of $d$. 
		\item In the third case, there is a variety of choices indexed by the half-edges attached to the same vertex as $\ell_{n+2}$. We set $d'(h) = d(h)+1$ for one selected half-edge attached to the same vertex as $\ell_{n+2}$ and $d'(\ell_{n+2}) = 0$. 
		All other half-edges of $\Gamma'$ are naturally identified with the half-edges of $\Gamma$ and the value of $d'$ coincides with the corresponding value of $d$. 
	\end{enumerate}
	Finally, for each $T\in \mathrm{SpTr}(\Gamma)$ we can put in correspondence one or two choices of $T'\in \mathrm{SpTr}(\Gamma)$:
	\begin{enumerate}
		\item  In the first case, the only possible choice of $T'$ is obtained from $T$ by adding $e_{\newe}$ and the new vertex of genus $0$. 
		\item In the second case, if $e\in T$, the there is the only possible choice of $T'$ that is obtained by replacing $e$ with $e_{\newe,1}$, $e_{\newe,2}$, and the new vertex. 
		If $e\not\in T$, then there are two possible choices. We obtain $T'$ by adding to $T$ either $e_{\newe,1}$ or $e_{\newe, 2}$, and, in both cases, the new vertex. 
		\item In the third case $T'$ is naturally identified with $T$. 
	\end{enumerate}
	Thus, for every tuple $(\Gamma,d,T)$, where $\Gamma\in \mathrm{SG}_{g,n+1}$, $d\colon H(\Gamma)  \to\ZZ_{\geq 0}$, and $T \in \mathrm{SpTr}(\Gamma)$, we put in correspondence a finite number of choices $(\Gamma',d',T')$ with $\Gamma'\in \mathrm{SG}_{g,n+2}$, $d'\colon H(\Gamma')  \to\ZZ_{\geq 0}$, and $T' \in \mathrm{SpTr}(\Gamma')$. By this correspondence we obviously exhaust all choices of tuples $(\Gamma',d',T')$. This can be encoded as a surjective map $F\colon (\Gamma',d',T')\mapsto (\Gamma,d,T)$ from the set of all tuples $(\Gamma',d',T')$, where $\Gamma'\in \mathrm{SG}_{g,n+2}$, $d'\colon H(\Gamma')  \to\ZZ_{\geq 0}$, and $T' \in \mathrm{SpTr}(\Gamma')$ to the set of all tuples $(\Gamma,d,T)$, where $\Gamma\in \mathrm{SG}_{g,n+1}$, $d\colon H(\Gamma)  \to\ZZ_{\geq 0}$, and $T \in \mathrm{SpTr}(\Gamma)$.

	Therefore, we can  rewrite Equation~\eqref{eq:DR-Pixton-Zagier-pushforward-result} as
	\begin{align} \label{eq:DR-Pixton-Zagier-pushforward-result-with-Trees}
		& \pi_*DR_{g}(a_1,\dots,a_n,-A-b,b) = O(b^2) +  \\ \notag
		& \frac 1{2^g} \sum_{\substack{ \Gamma\in \mathrm{SG}_{g,n+1} \\ d\colon H(\Gamma)  \to\ZZ_{\geq 0} \\ T \in \mathrm{SpTr}(\Gamma) }} \frac{(-2)^{|E(\Gamma)|} }{|\mathrm{Aut}(\Gamma)|} C_{\Gamma,d,T} (a_1,\dots,a_n,b)
		\Bigg( (\mathsf{b}_\Gamma)_* \bigg(\prod_{i=1}^{n+1} \psi_{\ell_i}^{d({\ell_i})} \prod_{e\in E(\Gamma)} \psi_{h(e)}^{d({h(e)})} \psi_{t(e)}^{d({t(e)})} \bigg) \Bigg)_{g-1},
	\end{align}
	where the coefficient $C_{\Gamma,d,T} (a_1,\dots,a_n,b)$ is obtained as  
	\begin{align} \label{eq:Coeff-Gamma-d-T}
		& C_{\Gamma,d,T} (a_1,\dots,a_n,b) \coloneqq 
		\\ \notag 
		& \sum_{(\Gamma',d',T')\in F^{-1}(\Gamma,d,T)} \prod_{i=1}^n \frac{a_i^{ 2d'({\ell_i})}}{d'({\ell_i})!} \cdot \frac{(-A-b)^{2d'(\ell_{n+1})} }{d'({\ell_{n+1}})!}\cdot 
		\prod_{e\in E(\Gamma)} \frac{(2d'({h(e)})+2d'({t(e)})+1)!}{d'({h(e)})!d'({t(e)})!}
		\\ \notag & \qquad \qquad 
		\times (-2)^{|E(\Gamma')|-|E(\Gamma)|}\mathrm{\Coeff}_{\left(\prod_{e\in E(\Gamma')} x_e^{2d'({h(e)})+2d'({t(e)})+2}\right)} 
		\prod_{f\in E(T')} e^{a_{f,T'} x_f} \prod_{f\not\in E(T')} \frac{x_f}{e^{x_{f,T'}}-1}.
	\end{align} 
	By construction $C_{\Gamma,d,T} (a_1,\dots,a_n,b)$ is obviously a polynomial in $a_1,\dots,a_n,b$. Let us prove that 
	\begin{align}
			\sum_{T\in \mathrm{SpTr}(\Gamma)} C_{\Gamma,d,T} (a_1,\dots,a_n,b) = O(b^2).
	\end{align}
	In order to do this, we split the set $\cup_{T\in \mathrm{SpTr}(\Gamma)} F^{-1}(\Gamma,d,T)$ into subsets such that the sum 
	\begin{align}
	\sum_{T\in \mathrm{SpTr}(\Gamma)} C_{\Gamma,d,T} (a_1,\dots,a_n,b)
	\end{align}
	 (given as the sum of all $T\in \mathrm{SpTr}(\Gamma)$ of the sums over $(\Gamma',d',T')\in F^{-1}(\Gamma,d,T)$
	on the right hand side of Equation~\eqref{eq:Coeff-Gamma-d-T}) splits into the polynomials in $O(b^2)$.
	
	The first type of sets $X_{T,\ell}\subseteq F^{-1}(\Gamma,d,T)$ is indexed by spanning trees $\sum_{T\in \mathrm{SpTr}(\Gamma)}$. The set $X_{T,\ell}$ consists of triples $(\Gamma',d',T')$ where either one of $d_i$ is increased by $1$ and the leave $\ell_{n+2}$ is attached to the same vertex as the leaf $\ell_i$ (these are graphs of type (3), but not all of them) or the graphs of type (1), that is, with an extra vertex of genus $0$ with the leaves $\ell_i$ and $\ell_{n+2}$ attached to this extra vertex, $i=1,\dots,n+1$. In this case let 
\begin{align} 
	& C_{X_{T,\ell}} (a_1,\dots,a_n,b) \coloneqq \\ \notag 
	& \sum_{(\Gamma',d',T')\in X_1} \prod_{i=1}^n \frac{a_i^{ 2d'({\ell_i})}}{d'({\ell_i})!} \cdot \frac{(-A-b)^{2d'(\ell_{n+1})} }{d'({\ell_{n+1}})!}\cdot 
	\prod_{e\in E(\Gamma)} \frac{(2d'({h(e)})+2d'({t(e)})+1)!}{d'({h(e)})!d'({t(e)})!}
	\\ \notag & \qquad \qquad 
	\times \mathrm{\Coeff}_{\left(\prod_{e\in E(\Gamma')} x_e^{2d'({h(e)})+2d'({t(e)})+2}\right)} 
	\prod_{f\in E(T')} e^{a_{f,T'} x_f} \prod_{f\not\in E(T')} \frac{x_f}{e^{x_{f,T'}}-1}
	\\ \notag 
	& = \sum_{i=1}^n \frac{a_i^2}{d(\ell_i)+1} \prod_{j=1}^n \frac{a_j^{ 2d({\ell_j})}}{d({\ell_j})!} \cdot \frac{(-A-b)^{2d(\ell_{n+1})} }{d({\ell_{n+1}})!}\cdot 
	\prod_{e\in E(\Gamma)} \frac{(2d({h(e)})+2d({t(e)})+1)!}{d({h(e)})!d({t(e)})!}
	\\ \notag & \qquad \qquad 
	\times \mathrm{\Coeff}_{\left(\prod_{e\in E(\Gamma)} x_e^{2d({h(e)})+2d({t(e)})+2}\right)} 
	\prod_{f\in E(T)} e^{(a_{f,T} + \delta_{f\uparrow \ell_i} b)x_f} \prod_{f\not\in E(T)} \frac{x_f}{e^{x_{f,T}}-1}
	\\ \notag 
	& \quad -2 \sum_{i=1}^n \prod_{j\ne i} \frac{a_j^{ 2d({\ell_j})}}{d({\ell_j})!} \cdot \frac{(-A-b)^{2d(\ell_{n+1})} }{d({\ell_{n+1}})!}\cdot 
	\prod_{e\in E(\Gamma)} \frac{(2d({h(e)})+2d({t(e)})+1)!}{d({h(e)})!d({t(e)})!} \cdot \frac{(2 d(\ell_i) +1)!}{d(\ell_i)!}
	\\ \notag & \qquad \qquad 
	\times \mathrm{\Coeff}_{ x_{e_{\newe}}^{2d(\ell_i)+2}
	\left(\prod_{e\in E(\Gamma)} x_e^{2d({h(e)})+2d({t(e)})+2}\right)}  e^{(a_i+b) x_{e_{\newe}} }
	\prod_{f\in E(T)} e^{(a_{f,T} + \delta_{f\uparrow \ell_i} b) x_f} \prod_{f\not\in E(T)} \frac{x_f}{e^{x_{f,T}}-1}
 	\\ \notag & \quad 
 +  \frac{(A+b)^2}{d(\ell_{n+1})+1} \prod_{j=1}^n \frac{a_j^{ 2d({\ell_j})}}{d({\ell_j})!} \cdot \frac{(-A-b)^{2d(\ell_{n+1})} }{d({\ell_{n+1}})!}\cdot 
\prod_{e\in E(\Gamma)} \frac{(2d({h(e)})+2d({t(e)})+1)!}{d({h(e)})!d({t(e)})!}
\\ \notag & \qquad \qquad 
\times \mathrm{\Coeff}_{\left(\prod_{e\in E(\Gamma)} x_e^{2d({h(e)})+2d({t(e)})+2}\right)} 
\prod_{f\in E(T)} e^{(a_{f,T}+ \delta_{f\uparrow \ell_{n+1}} b) x_f} \prod_{f\not\in E(T)} \frac{x_f}{e^{x_{f,T}}-1}
	\\ \notag & \quad -2 \prod_{j=1}^n \frac{a_j^{ 2d({\ell_j})}}{d({\ell_j})!}  \cdot 
	\prod_{e\in E(\Gamma)} \frac{(2d({h(e)})+2d({t(e)})+1)!}{d({h(e)})!d({t(e)})!} \cdot \frac{(2 d(\ell_i) +1)!}{d(\ell_i)!}
	\\ \notag & \qquad \qquad 
	\times \mathrm{\Coeff}_{ x_{e_{\newe}}^{2d(\ell_i)+2}
		\left(\prod_{e\in E(\Gamma)} x_e^{2d({h(e)})+2d({t(e)})+2}\right)}  e^{(-A) x_{e_{\newe}} }
	\prod_{f\in E(T)} e^{(a_{f,T}+ \delta_{f\uparrow \ell_{n+1}} b) x_f}  \prod_{f\not\in E(T)} \frac{x_f}{e^{x_{f,T}}-1},
\end{align} 
where $\delta_{f\uparrow \ell_i}$ is equal to $1$ if the vertex to which $\ell_i$ is attached is ahead of $f$ in $T$ and $0$ otherwise. The terms in the latter expression can be combined in pairs to give 
\begin{align} 
	& \quad \sum_{i=1}^n \frac{a_i^{2d(\ell_i)+2}-(a_i+b)^{2d(\ell_i)+2}}{(d(\ell_i)+1)!} \prod_{j\ne i} \frac{a_j^{ 2d({\ell_j})}}{d({\ell_j})!} \cdot \frac{(-A-b)^{2d(\ell_{n+1})} }{d({\ell_{n+1}})!}\cdot 
	\prod_{e\in E(\Gamma)} \frac{(2d({h(e)})+2d({t(e)})+1)!}{d({h(e)})!d({t(e)})!}
	\\ \notag & \qquad \qquad 
	\times \mathrm{\Coeff}_{\left(\prod_{e\in E(\Gamma)} x_e^{2d({h(e)})+2d({t(e)})+2}\right)} 
	\prod_{f\in E(T)} e^{(a_{f,T} + \delta_{f\uparrow \ell_i} b)x_f} \prod_{f\not\in E(T)} \frac{x_f}{e^{x_{f,T}}-1}
	\\ \notag & \quad 
	+  \frac{(A+b)^{2d(\ell_{n+1})+2}-A^{2d(\ell_{n+1})+2}}{(d(\ell_{n+1})+1)!} \prod_{j=1}^n \frac{a_j^{ 2d({\ell_j})}}{d({\ell_j})!} \cdot 
	\prod_{e\in E(\Gamma)} \frac{(2d({h(e)})+2d({t(e)})+1)!}{d({h(e)})!d({t(e)})!}
	\\ \notag & \qquad \qquad 
	\times \mathrm{\Coeff}_{\left(\prod_{e\in E(\Gamma)} x_e^{2d({h(e)})+2d({t(e)})+2}\right)} 
	\prod_{f\in E(T)} e^{(a_{f,T}+ \delta_{f\uparrow \ell_{n+1}} b) x_f} \prod_{f\not\in E(T)} \frac{x_f}{e^{x_{f,T}}-1}.
\end{align} 
The latter expression is equal mod $O(b^2)$ terms to
\begin{align} 
	& \quad \sum_{i=1}^n (-2ba_i) \prod_{j=1}^n \frac{a_j^{ 2d({\ell_j})}}{d({\ell_j})!} \cdot \frac{(-A-b)^{2d(\ell_{n+1})} }{d({\ell_{n+1}})!}\cdot 
	\prod_{e\in E(\Gamma)} \frac{(2d({h(e)})+2d({t(e)})+1)!}{d({h(e)})!d({t(e)})!}
	\\ \notag & \qquad \qquad 
	\times \mathrm{\Coeff}_{\left(\prod_{e\in E(\Gamma)} x_e^{2d({h(e)})+2d({t(e)})+2}\right)} 
	\prod_{f\in E(T)} e^{(a_{f,T} + \delta_{f\uparrow \ell_i} b)x_f} \prod_{f\not\in E(T)} \frac{x_f}{e^{x_{f,T}}-1}
	\\ \notag & \quad 
	+  2bA \prod_{j=1}^n \frac{a_j^{ 2d({\ell_j})}}{d({\ell_j})!} \cdot \frac{(-A-b)^{2d(\ell_{n+1})} }{d({\ell_{n+1}})!} \cdot 
	\prod_{e\in E(\Gamma)} \frac{(2d({h(e)})+2d({t(e)})+1)!}{d({h(e)})!d({t(e)})!}
	\\ \notag & \qquad \qquad 
	\times \mathrm{\Coeff}_{\left(\prod_{e\in E(\Gamma)} x_e^{2d({h(e)})+2d({t(e)})+2}\right)} 
	\prod_{f\in E(T)} e^{(a_{f,T}+ \delta_{f\uparrow \ell_{n+1}} b) x_f} \prod_{f\not\in E(T)} \frac{x_f}{e^{x_{f,T}}-1}.
\end{align} 
Now, since all terms are proportional to $b$, and we want to consider this expression mod $O(b^2)$ terms, we can drop the summands $\delta_{f\uparrow \ell_{i}} b$, $i=1,\dots,n+1$, in the exponents. Thus, we obtain that the coefficient $C_{X_{T,\ell}} (a_1,\dots,a_n,b)$ is equal to 
\begin{align}
& C_{X_{T,\ell}} (a_1,\dots,a_n,b)  \\ \notag & 
= \left(\sum_{i=1}^n (-2ba_i) +2bA \right)\prod_{j=1}^n \frac{a_j^{ 2d({\ell_j})}}{d({\ell_j})!} \cdot \frac{(-A-b)^{2d(\ell_{n+1})} }{d({\ell_{n+1}})!}\cdot 
\prod_{e\in E(\Gamma)} \frac{(2d({h(e)})+2d({t(e)})+1)!}{d({h(e)})!d({t(e)})!}
\\ \notag & \qquad \qquad 
\times \mathrm{\Coeff}_{\left(\prod_{e\in E(\Gamma)} x_e^{2d({h(e)})+2d({t(e)})+2}\right)} 
\prod_{f\in E(T)} e^{a_{f,T} x_f} \prod_{f\not\in E(T)} \frac{x_f}{e^{x_{f,T}}-1} + O(b^2)
\\ \notag & = O(b^2). 
\end{align}

The second type of sets is indexed by pairs $(T\in \mathrm{SpTr}(\Gamma), e\in T)$. We denote these sets by $X_{T,e\in T}$. They consist of one graph of type (2), where the edge $e$ is replaced by $e_{\newe,1}$ and $e_{\newe,2}$ (and this fixes all the data of $(\Gamma',d',T')$ uniquely) and two graphs of type (3), where either the leaf $\ell_{n+2}$ is attached to the same vertex as $h(e)$, and then $d'(h(e))\coloneqq d(h(e))+1$ or the $\ell_{n+2}$ is attached to the same vertex as $t(e)$, and then $d'(t(e))\coloneqq d(t(e))+1$. In this case,
\begin{align}  \label{eq:Coeff-T-e-in-T-details}
	& C_{X_{T,e\in T}} (a_1,\dots,a_n,b) \coloneqq \\ \notag 
	& \sum_{(\Gamma',d',T')\in X_{T,e\in T}} \prod_{i=1}^n \frac{a_i^{ 2d'({\ell_i})}}{d'({\ell_i})!} \cdot \frac{(-A-b)^{2d'(\ell_{n+1})} }{d'({\ell_{n+1}})!}\cdot 
	\prod_{s\in E(\Gamma)} \frac{(2d'({h(s)})+2d'({t(s)})+1)!}{d'({h(s)})!d'({t(s)})!}
	\\ \notag & \qquad \qquad 
	\times \mathrm{\Coeff}_{\left(\prod_{s\in E(\Gamma')} x_s^{2d'({h(s)})+2d'({t(s)})+2}\right)} 
	\prod_{f\in E(T')} e^{a_{f,T'} x_f} \prod_{f\not\in E(T')} \frac{x_f}{e^{x_{f,T'}}-1}
	\\ \notag & 
	= \prod_{i=1}^n \frac{a_i^{ 2d({\ell_i})}}{d({\ell_i})!} \cdot \frac{(-A-b)^{2d(\ell_{n+1})} }{d({\ell_{n+1}})!}\cdot 
	\prod_{s\in E(\Gamma)\setminus \{e\}} \frac{(2d({h(s)})+2d({t(s)})+1)!}{d({h(s)})!d({t(s)})!} \times
	\\ \notag & \qquad
	\Bigg( -2\cdot \frac{(2d({h(e)})+1)!(2d({t(e)})+1)!}{d({h(e)})!d({t(s)})!}
	\cdot\mathrm{\Coeff}_{\left( x_{e_{\newe,1}}^{2d(h(e))+2} x_{e_{\newe,2}}^{2d(t(e))+2} \prod_{s\in E(\Gamma)\setminus \{e\}} x_s^{2d({h(s)})+2d({t(s)})+2} \right)} 
	\\ \notag & \qquad \qquad e^{a_{e,T} x_{e_{\newe,1}} +(a_{e,T}+b) x_{e_{\newe,2}}} 
	\prod_{f\in E(T)\setminus \{e\}} e^{(a_{f,T} + \delta_{f\uparrow e} b) x_f} \prod_{f\not\in E(T)} \frac{x_f}{e^{x_{f,T}}-1}\big\vert_{x_e\to x_{e_{\newe,1}}+  x_{e_{\newe,2}} }
	\\ \notag & \qquad 
	+ \frac{(2d({h(e)})+2d({t(e)})+3)!}{(d({h(e)})+1)!d({t(s)})!}
	\cdot\mathrm{\Coeff}_{\left( x_{e}^{2d(h(e))+2d(t(e))+4} \prod_{s\in E(\Gamma)\setminus \{e\}} x_s^{2d({h(s)})+2d({t(s)})+2} \right)} 
	\\ \notag & \qquad \qquad e^{(a_{e,T}+b) x_{e}} 
	\prod_{f\in E(T)\setminus \{e\}} e^{(a_{f,T} + \delta_{f\uparrow e} b) x_f} \prod_{f\not\in E(T)} \frac{x_f}{e^{x_{f,T}}-1}
	\\ \notag & \qquad 
	+ \frac{(2d({h(e)})+2d({t(e)})+3)!}{d({h(e)})!(d({t(s)})+1)!}
	\cdot\mathrm{\Coeff}_{\left( x_{e}^{2d(h(e))+2d(t(e))+4} \prod_{s\in E(\Gamma)\setminus \{e\}} x_s^{2d({h(s)})+2d({t(s)})+2} \right)} 
	\\ \notag & \qquad \qquad e^{a_{e,T} x_{e}} 
	\prod_{f\in E(T)\setminus \{e\}} e^{(a_{f,T} + \delta_{f\uparrow e} b) x_f} \prod_{f\not\in E(T)} \frac{x_f}{e^{x_{f,T}}-1} \Bigg).
\end{align}
We are going to analyze only the expression in the parentheses, up to the $O(b^2)$ terms. Note that 
\begin{align}  \label{eq:Obs1}
\mathrm{\Coeff}_{x_1^{d_1}x_2^{d_2}} f(x_1+x_2) = \frac{(d_1+d_2)!}{d_1!d_2!} \mathrm{\Coeff}_{x^{d_1 + d_2}} f(x).
\end{align} 
Note also that 
\begin{align} \label{eq:Obs2}
	\Coeff_{x^d} e^{(a+b)x} f(x) = \Coeff_{x^d} e^{ax} f(x) + b \Coeff_{x^{d-1}} e^{ax} f(x) + O(b^2).
\end{align}
Combining these two observations, the expression in parentheses on the right hand side of~\eqref{eq:Coeff-T-e-in-T-details} is equal to
\begin{align}
& 
\bigg( -2\cdot \frac{(2d({h(e)})+1)!(2d({t(e)})+1)!}{d({h(e)})!d({t(s)})!} \frac{(2d(h(e))+2d(t(e))+4)!}{(2d({h(e)})+2)!(2d({t(e)})+2)!}
\\ \notag & \qquad \qquad +  \frac{(2d({h(e)})+2d({t(e)})+3)!}{(d({h(e)})+1)!d({t(s)})!} + \frac{(2d({h(e)})+2d({t(e)})+3)!}{d({h(e)})!(d({t(s)})+1)!} \bigg)
\\ \notag & \qquad  \times 
\mathrm{\Coeff}_{\left( x_{e}^{2} \prod_{s\in E(\Gamma)} x_s^{2d({h(s)})+2d({t(s)})+2} \right)} 
e^{a_{e,T} x_{e}} \prod_{f\in E(T)\setminus \{e\}} e^{(a_{f,T} + \delta_{f\uparrow e} b) x_f} \prod_{f\not\in E(T)} \frac{x_f}{e^{x_{f,T}}-1} \Bigg)
\\ \notag &
+ b\cdot \bigg( -2\cdot \frac{(2d({h(e)})+1)!(2d({t(e)})+1)!}{d({h(e)})!d({t(s)})!} \frac{(2d(h(e))+2d(t(e))+3)!}{(2d({h(e)})+2)!(2d({t(e)})+1)!}  
\\ \notag & \qquad \qquad 
+ \frac{(2d({h(e)})+2d({t(e)})+3)!}{(d({h(e)})+1)!d({t(s)})!} \bigg)
\\ \notag & \qquad  \times 
\mathrm{\Coeff}_{\left( x_{e} \prod_{s\in E(\Gamma)} x_s^{2d({h(s)})+2d({t(s)})+2} \right)} 
e^{a_{e,T} x_{e}} \prod_{f\in E(T)\setminus \{e\}} e^{(a_{f,T} + \delta_{f\uparrow e} b) x_f} \prod_{f\not\in E(T)} \frac{x_f}{e^{x_{f,T}}-1} \Bigg) + O(b^2)
\\ \notag & = O(b^2).
\end{align} 
Thus $C_{X_{T,e\in T}} (a_1,\dots,a_n,b) = O(b^2)$.

Finally, the third type of sets is indexed by pairs $(T\in \mathrm{SpTr}(\Gamma), e\not\in T)$. We denote these sets by $X_{T,e\not\in T}$. They consist of two graphs of type (2), where the edge $e$ is replaces by $e_{\newe,1}$ and $e_{\newe,2}$ --- this fixes $d'$ uniquely, but there $T'$ can contain either $e_{\newe,1}$ or  $e_{\newe,2}$, and this gives two possible choices ---  and two graphs of type (3), where either the leaf $\ell_{n+2}$ is attached to the same vertex as $h(e)$, and then $d'(h(e))\coloneqq d(h(e))+1$ or the $\ell_{n+2}$ is attached to the same vertex as $t(e)$, and then $d'(t(e))\coloneqq d(t(e))+1$. Thus we have four graphs in total, and 
\begin{align}  \label{eq:Coeff-T-e-in-T}
	& C_{X_{T,e\not\in T}} (a_1,\dots,a_n,b) \coloneqq \\ \notag 
	& \sum_{(\Gamma',d',T')\in X_{T,e\not\in T}} \prod_{i=1}^n \frac{a_i^{ 2d'({\ell_i})}}{d'({\ell_i})!} \cdot \frac{(-A-b)^{2d'(\ell_{n+1})} }{d'({\ell_{n+1}})!}\cdot 
	\prod_{s\in E(\Gamma)} \frac{(2d'({h(s)})+2d'({t(s)})+1)!}{d'({h(s)})!d'({t(s)})!}
	\\ \notag & \qquad \qquad 
	\times \mathrm{\Coeff}_{\left(\prod_{s\in E(\Gamma')} x_s^{2d'({h(s)})+2d'({t(s)})+2}\right)} 
	\prod_{f\in E(T')} e^{a_{f,T'} x_f} \prod_{f\not\in E(T')} \frac{x_f}{e^{x_{f,T'}}-1}
	\\ \notag & 
	= \prod_{i=1}^n \frac{a_i^{ 2d({\ell_i})}}{d({\ell_i})!} \cdot \frac{(-A-b)^{2d(\ell_{n+1})} }{d({\ell_{n+1}})!}\cdot 
	\prod_{s\in E(\Gamma)\setminus \{e\}} \frac{(2d({h(s)})+2d({t(s)})+1)!}{d({h(s)})!d({t(s)})!} \times
	\\ \notag & \qquad 
	\Bigg( -2\cdot \frac{(2d({h(e)})+1)!(2d({t(e)})+1)!}{d({h(e)})!d({t(s)})!}
	\cdot\mathrm{\Coeff}_{\left( x_{e_{\newe,1}}^{2d(h(e))+2} x_{e_{\newe,2}}^{2d(t(e))+2} \prod_{s\in E(\Gamma)\setminus \{e\}} x_s^{2d({h(s)})+2d({t(s)})+2} \right)} 
	\\ \notag & \qquad \qquad e^{-b x_{e_{\newe,1}} } 
	\prod_{f\in E(T)} e^{(a_{f,T} + \delta_{f\uparrow h(e)} b) x_f}  \frac{x_{e_{\newe,2}}}{\big(e^{x_{e,T}}|_{x_e\to x_{e_{\newe,1}}+  x_{e_{\newe,2}} }\big) -1 } \prod_{f\not\in E(T)\cup \{e\}} \frac{x_f}{e^{x_{f,T}}-1}
	\\ \notag & \qquad 
	-2\cdot \frac{(2d({h(e)})+1)!(2d({t(e)})+1)!}{d({h(e)})!d({t(s)})!}
	\cdot\mathrm{\Coeff}_{\left( x_{e_{\newe,1}}^{2d(h(e))+2} x_{e_{\newe,2}}^{2d(t(e))+2} \prod_{s\in E(\Gamma)\setminus \{e\}} x_s^{2d({h(s)})+2d({t(s)})+2} \right)} 
	\\ \notag & \qquad \qquad e^{b x_{e_{\newe,2}} } 
	\prod_{f\in E(T)} e^{(a_{f,T} + \delta_{f\uparrow t(e)} b) x_f}  \frac{x_{e_{\newe,1}}}{\big(e^{x_{e,T}}|_{x_e\to x_{e_{\newe,1}}+  x_{e_{\newe,2}} }\big) -1 } \prod_{f\not\in E(T)\cup \{e\}} \frac{x_f}{e^{x_{f,T}}-1}
	\\ \notag & \qquad 
	+ \frac{(2d({h(e)})+2d({t(e)})+3)!}{(d({h(e)})+1)!d({t(s)})!}
	\cdot\mathrm{\Coeff}_{\left( x_{e}^{2} \prod_{s\in E(\Gamma)} x_s^{2d({h(s)})+2d({t(s)})+2} \right)} 
	\\ \notag & \qquad \qquad 
	\prod_{\substack {f\in E(T)
	}} e^{(a_{f,T} + \delta_{f\uparrow h(e)} b) x_f} \prod_{\substack {f\not\in E(T)}} \frac{x_f}{e^{x_{f,T}}-1}
	\\ \notag & \qquad 
	+ \frac{(2d({h(e)})+2d({t(e)})+3)!}{d({h(e)})!(d({t(s)})+1)!}
	\cdot\mathrm{\Coeff}_{\left( x_{e}^{2} \prod_{s\in E(\Gamma)} x_s^{2d({h(s)})+2d({t(s)})+2} \right)} 
	\\ \notag & \qquad \qquad 
	\prod_{\substack {f\in E(T)
	}} e^{(a_{f,T} + \delta_{f\uparrow t(e)} b) x_f} \prod_{\substack {f\not\in E(T)}} \frac{x_f}{e^{x_{f,T}}-1} \Bigg).
\end{align}
Using the observations~\eqref{eq:Obs1},~\eqref{eq:Obs2}, we see that the expression in parentheses on the right hand side of this equation is equal to
\begin{align}
	&
	b^0\cdot \Bigg( -2\cdot \frac{(2d({h(e)})+1)!(2d({t(e)})+1)!}{d({h(e)})!d({t(s)})!}
	\cdot\mathrm{\Coeff}_{\left( x_{e_{\newe,1}}^{2d(h(e))+2} x_{e_{\newe,2}}^{2d(t(e))+1} \prod_{s\in E(\Gamma)\setminus \{e\}} x_s^{2d({h(s)})+2d({t(s)})+2} \right)} 
	\\ \notag & \qquad \qquad 
	\prod_{f\in E(T)} e^{(a_{f,T} + \delta_{f\uparrow h(e)} b) x_f}  \frac{1}{\big(e^{x_{e,T}}|_{x_e\to x_{e_{\newe,1}}+  x_{e_{\newe,2}} }\big) -1 } \prod_{f\not\in E(T)\cup \{e\}} \frac{x_f}{e^{x_{f,T}}-1}
	\\ \notag & \qquad 
	-2\cdot \frac{(2d({h(e)})+1)!(2d({t(e)})+1)!}{d({h(e)})!d({t(s)})!}
	\cdot\mathrm{\Coeff}_{\left( x_{e_{\newe,1}}^{2d(h(e))+1} x_{e_{\newe,2}}^{2d(t(e))+2} \prod_{s\in E(\Gamma)\setminus \{e\}} x_s^{2d({h(s)})+2d({t(s)})+2} \right)} 
	\\ \notag & \qquad \qquad 
	\prod_{f\in E(T)} e^{(a_{f,T} + \delta_{f\uparrow h(e)} b) x_f}  \frac{1}{\big(e^{x_{e,T}}|_{x_e\to x_{e_{\newe,1}}+  x_{e_{\newe,2}} }\big) -1 } \prod_{f\not\in E(T)\cup \{e\}} \frac{x_f}{e^{x_{f,T}}-1}
	\\ \notag & \qquad 
	+ \frac{(2d({h(e)})+2d({t(e)})+3)!}{(d({h(e)})+1)!d({t(s)})!}
	\cdot\mathrm{\Coeff}_{\left( x_{e}^{1} \prod_{s\in E(\Gamma)} x_s^{2d({h(s)})+2d({t(s)})+2} \right)} 
	\\ \notag & \qquad \qquad 
	\prod_{\substack {f\in E(T)
	}} e^{(a_{f,T} + \delta_{f\uparrow h(e)} b) x_f} \frac{1}{e^{x_{e,T}}  -1 } \prod_{\substack {f\not\in E(T)\cup \{e\}}} \frac{x_f}{e^{x_{f,T}}-1}
	\\ \notag & \qquad 
	+ \frac{(2d({h(e)})+2d({t(e)})+3)!}{d({h(e)})!(d({t(s)})+1)!}
	\cdot\mathrm{\Coeff}_{\left( x_{e}^{1} \prod_{s\in E(\Gamma)} x_s^{2d({h(s)})+2d({t(s)})+2} \right)} 
	\\ \notag & \qquad \qquad 
	\prod_{\substack {f\in E(T)
	}} e^{(a_{f,T} + \delta_{f\uparrow t(e)} b) x_f} \frac{1}{e^{x_{e,T}}  -1 } \prod_{\substack {f\not\in E(T)\cup \{e\}}} \frac{x_f}{e^{x_{f,T}}-1} \Bigg)
\\ \notag &
+ b^1 \cdot 
\Bigg( 2\cdot \frac{(2d({h(e)})+1)!(2d({t(e)})+1)!}{d({h(e)})!d({t(s)})!}
\cdot\mathrm{\Coeff}_{\left( x_{e_{\newe,1}}^{2d(h(e))+1} x_{e_{\newe,2}}^{2d(t(e))+1} \prod_{s\in E(\Gamma)\setminus \{e\}} x_s^{2d({h(s)})+2d({t(s)})+2} \right)} 
\\ \notag & \qquad \qquad 
\prod_{f\in E(T)} e^{(a_{f,T} + \delta_{f\uparrow h(e)} b) x_f}  \frac{1}{\big(e^{x_{e,T}}|_{x_e\to x_{e_{\newe,1}}+  x_{e_{\newe,2}} }\big) -1 } \prod_{f\not\in E(T)\cup \{e\}} \frac{x_f}{e^{x_{f,T}}-1}
\\ \notag & \qquad 
-2\cdot \frac{(2d({h(e)})+1)!(2d({t(e)})+1)!}{d({h(e)})!d({t(s)})!}
\cdot\mathrm{\Coeff}_{\left( x_{e_{\newe,1}}^{2d(h(e))+1} x_{e_{\newe,2}}^{2d(t(e))+1} \prod_{s\in E(\Gamma)\setminus \{e\}} x_s^{2d({h(s)})+2d({t(s)})+2} \right)} 
\\ \notag & \qquad \qquad 
\prod_{f\in E(T)} e^{(a_{f,T} + \delta_{f\uparrow t(e)} b) x_f}  \frac{1}{\big(e^{x_{e,T}}|_{x_e\to x_{e_{\newe,1}}+  x_{e_{\newe,2}} }\big) -1 } \prod_{f\not\in E(T)\cup \{e\}} \frac{x_f}{e^{x_{f,T}}-1}
\Bigg) + O(b^2)
\end{align}
Here the first summand is manifestly equal to zero, and the second summand can be rewritten as 
\begin{align}
& b^1 \cdot 
\Bigg( 2\cdot \frac{(2d({h(e)})+1)!(2d({t(e)})+1)!}{d({h(e)})!d({t(s)})!}
\cdot\mathrm{\Coeff}_{\left( x_{e_{\newe,1}}^{2d(h(e))+1} x_{e_{\newe,2}}^{2d(t(e))+1} \prod_{s\in E(\Gamma)\setminus \{e\}} x_s^{2d({h(s)})+2d({t(s)})+2} \right)} 
\\ \notag & \qquad \qquad 
\prod_{f\in E(T)} e^{(a_{f,T} + \delta_{f\uparrow h(e)} b) x_f}  \frac{1}{\big(e^{x_{e,T}}|_{x_e\to x_{e_{\newe,1}}+  x_{e_{\newe,2}} }\big) -1 } \prod_{f\not\in E(T)\cup \{e\}} \frac{x_f}{e^{x_{f,T}}-1}
\\ \notag & \qquad 
-2\cdot \frac{(2d({h(e)})+1)!(2d({t(e)})+1)!}{d({h(e)})!d({t(s)})!}
\cdot\mathrm{\Coeff}_{\left( x_{e_{\newe,1}}^{2d(h(e))+1} x_{e_{\newe,2}}^{2d(t(e))+1} \prod_{s\in E(\Gamma)\setminus \{e\}} x_s^{2d({h(s)})+2d({t(s)})+2} \right)} 
\\ \notag & \qquad \qquad 
\prod_{f\in E(T)} e^{(a_{f,T} + \delta_{f\uparrow t(e)} b) x_f}  \frac{1}{\big(e^{x_{e,T}}|_{x_e\to x_{e_{\newe,1}}+  x_{e_{\newe,2}} }\big) -1 } \prod_{f\not\in E(T)\cup \{e\}} \frac{x_f}{e^{x_{f,T}}-1}
\Bigg)
\\ \notag 
& =  b^1 \cdot 
\Bigg( 2\cdot \frac{(2d({h(e)})+1)!(2d({t(e)})+1)!}{d({h(e)})!d({t(s)})!}
\cdot\mathrm{\Coeff}_{\left( x_{e_{\newe,1}}^{2d(h(e))+1} x_{e_{\newe,2}}^{2d(t(e))+1} \prod_{s\in E(\Gamma)\setminus \{e\}} x_s^{2d({h(s)})+2d({t(s)})+2} \right)} 
\\ \notag & \qquad \qquad 
\prod_{f\in E(T)} e^{a_{f,T}  x_f}  \frac{1}{\big(e^{x_{e,T}}|_{x_e\to x_{e_{\newe,1}}+  x_{e_{\newe,2}} }\big) -1 } \prod_{f\not\in E(T)\cup \{e\}} \frac{x_f}{e^{x_{f,T}}-1}
\\ \notag & \qquad 
-2\cdot \frac{(2d({h(e)})+1)!(2d({t(e)})+1)!}{d({h(e)})!d({t(s)})!}
\cdot\mathrm{\Coeff}_{\left( x_{e_{\newe,1}}^{2d(h(e))+1} x_{e_{\newe,2}}^{2d(t(e))+1} \prod_{s\in E(\Gamma)\setminus \{e\}} x_s^{2d({h(s)})+2d({t(s)})+2} \right)} 
\\ \notag & \qquad \qquad 
\prod_{f\in E(T)} e^{a_{f,T} x_f}  \frac{1}{\big(e^{x_{e,T}}|_{x_e\to x_{e_{\newe,1}}+  x_{e_{\newe,2}} }\big) -1 } \prod_{f\not\in E(T)\cup \{e\}} \frac{x_f}{e^{x_{f,T}}-1}
\Bigg) + O(b^2)
\\ \notag &
= O(b^2).
\end{align}
Thus $C_{X_{T,e\not\in T}} (a_1,\dots,a_n,b) = O(b^2)$.

Now we notice that $\cup_{T\in \mathrm{SpTr}(\Gamma)}F^{-1}(\Gamma,d,T) = \cup_{T\in \mathrm{SpTr}(\Gamma)} \big(X_{T,\ell} \cup_{e\in E(T)} X_{T,e\in E(T)}\cup_{e\in E(\Gamma)\setminus E(T)} X_{T,e\not\in E(T)} \big)$. Thus Equation~\eqref{eq:DR-Pixton-Zagier-pushforward-result-with-Trees} can be rearranged as 
	\begin{align} 
	& \pi_*DR_{g}(a_1,\dots,a_n,-A-b,b) = O(b^2) +  \\ \notag
	& \frac 1{2^g} \sum_{\substack{ \Gamma\in \mathrm{SG}_{g,n+1} \\ d\colon H(\Gamma)  \to\ZZ_{\geq 0} }} \frac{(-2)^{|E(\Gamma)|} }{|\mathrm{Aut}(\Gamma)|} 
	\Bigg( (\mathsf{b}_\Gamma)_* \bigg(\prod_{i=1}^{n+1} \psi_{\ell_i}^{d({\ell_i})} \prod_{e\in E(\Gamma)} \psi_{h(e)}^{d({h(e)})} \psi_{t(e)}^{d({t(e)})} \bigg) \Bigg)_{g-1}
	\\ \notag &
	\times\sum_{T \in \mathrm{SpTr}(\Gamma)} \bigg(C_{X_{T,\ell}} (a_1,\dots,a_n,b) +\sum_{e\in E(T)} C_{X_{T,e\in T}} (a_1,\dots,a_n,b)+ \sum_{e\in E(\Gamma)\setminus E(T)} C_{X_{T,e\not\in T}} (a_1,\dots,a_n,b)\bigg),
\end{align}
where all coefficients $C_{X_{T,\bullet}}$ are polynomials that belong to the ideal $(b^2)\subset \QQ[a_1,\dots,a_n,b]$. This completes the proof of the proposition.

\bigskip

\section{Proof of the quantum A-class representation for the correlators}
\label{sec:proof-A-class}
In this section we prove \Cref{thm: quantum A class}. In the first step (\Cref{sec:proof-string}), we prove the string equation. Therefore, it is sufficient to prove \Cref{thm: quantum A class} for correlators with a $\tau_{0,1}$-insertion. In the second step (\Cref{sec:geo-fomula}), we give a first geometric formula for these correlators. In the third step (\Cref{sec:simplification-push-formula}), we simplify this formula using the push-forward property of the DR cycles. In the final step (\Cref{sec:final-relation}), we demonstrate how the quantum $A$-class representation for these correlators arises from a relation that we establish.

Our proofs generalise the argument of \cite{BDGR1,BDGR20,DRDZ} establishing the $A$-class representation for the correlators in the classical setting. Notice that, our final step even  simplifies the argument presented in \emph{op. cit.}  

\medskip

\subsection{Step 1: the string equation}
\label{sec:proof-string}


In this section we prove the string equation (\Cref{eq:STRING} of \Cref{thm:StringANDDilaton}). Equivalently we show that, whenever $2g-2+n>0$, we have the relation
\begin{align}
\left\langle \tau_{0,1}\tau_{d_{1},\alpha_{1}}\cdots\tau_{d_{n},\alpha_{n}}\right\rangle _{l,g-l;k} & =\sum_{i=1}^{n}\left\langle \tau_{d_{1},\alpha_{1}}\cdots\tau_{d_{i}-1,\alpha_{i}}\cdots\tau_{d_{n},\alpha_{n}}\right\rangle _{l,g-l;k}\label{eq: String induction}
\end{align}
with the convention that a correlator vanishes if $\tau$ has a negative index. Moreover the base cases are given by
\begin{align}
\left\langle \tau_{0,1}\tau_{d_{1},\alpha_{1}}\tau_{d_{2},\alpha_{2}}\right\rangle _{l,g-l;k} & =\delta_{g,0}\delta_{l,0}\delta_{k,0}\delta_{0,d_{1}}\delta_{0,d_{2}}\eta_{\alpha_{1}\alpha_{2}},\label{eq: cubic term}\\
\left\langle \tau_{0,1}\tau_{0,\alpha}\right\rangle _{l,g-l;k} & =0,\label{eq: quadratic term}\\
\left\langle \tau_{0,1}\right\rangle _{l,g-l;k} & =-\delta_{g,1}\delta_{l,0}\delta_{k,0}\frac{N}{24}.\label{eq: linear term}
\end{align}

\smallskip

\subsubsection{Proof of \Cref{eq: String induction} and \Cref{eq: cubic term}.}

First notice that \Cref{eq: String induction} holds by definition of the correlators when $n=0,1$. For the remaining cases, we need a preliminary lemma. 
\begin{lem}
\label{lem: small lemmas string}We have the following statements.
\begin{enumerate}
\item For a differential polynomial $f\in\mathcal{A}_{N}$, we have 
\begin{equation}
\partial_{x}f=\frac{1}{\hbar}\left[f,\overline{H}_{0,1}\right].\label{eq: H_0 x-derivative}
\end{equation}
\item For $d\geq0$ and $1\leq\alpha\leq N$, we have
\begin{equation}
\frac{\partial H_{d,\alpha}}{\partial p_{0}^{1}}=H_{d-1,\alpha}.\label{eq: string for hamiltonians}
\end{equation}
\item For $d\geq0$ and $1\leq\alpha\leq N$, we have
\begin{equation}
\Omega_{d,\alpha;0,1}=H_{d-1,\alpha}.\label{eq: Omega =00003D H}
\end{equation}
\item For $d_{1},d_{2}\geq0$ and $1\leq\alpha_{1},\alpha_{2}\leq N$, we have 
\begin{equation}
\frac{\partial\Omega_{d_{1},\alpha_{1};d_{2},\alpha_{2}}}{\partial p_{0}^{1}}=\Omega_{d_{1}-1,\alpha_{1};d_{2},\alpha_{2}}+\Omega_{d_{1},\alpha_{1};d_{2}-1,\alpha_{2}}+\delta_{d_{1},0}\delta_{d_{2},0}\eta_{\alpha_{1}\alpha_{2}},\label{eq:initial condition two point}
\end{equation}
where we use the convention that $\Omega$ vanishes if at least one of its index is negative.
\end{enumerate}
\end{lem}

\begin{proof}
$\left(1\right)$ We have 
\begin{equation}
\overline{H}_{0,1}=\int_{S^{1}}\frac{\left(u_{0}^{1}\right)^{2}}{2}dx
\end{equation}
since the coefficients of $\overline{H}_{0,1}$ are 
\begin{align}
\int_{DR_{g}\left(0,a_{1},\dots,a_{m},0\right)}\psi_{1}\lambda_{l}c_{g,m+2}^{[\mu]}\left(1,\alpha_{1},\dots,\alpha_{m},1\right)=\int_{DR_{g}\left(0,a_{1},\dots,a_{m}\right)}\lambda_{l}c_{g,m+2}^{[\mu]}\left(1,\alpha_{1},\dots,\alpha_{m}\right),
\end{align}
and they vanish whenever the map forgetting a marked point exists. Now using the expression of the commutator in $u$-variables given in \cite[Eq.~(1.2)]{BR16-quantum}, we obtain the conclusion.

$\left(2\right)$ The statement is an analogue of \cite[Lemma~2.4]{Blot}.

$\left(3\right)$ Using the definition of the two point function (\ref{eq:def-two-point}) and \Cref{eq: H_0 x-derivative}, we get 
\begin{align}
\Omega_{d,\alpha;0,1}=H_{d-1,\alpha}+C_{d,\alpha},\quad d\geq0,\:1\leq\alpha\leq N
\end{align}
for some constant $C_{d,\alpha}$. The constant is fixed using \Cref{eq:constant-two point function}) for $\left(d_{1},\alpha_{1};d_{2},\alpha_{2}\right)=\left(d,\alpha;0,1\right)$. Now using \Cref{eq: string for hamiltonians}, we get $C_{d,\alpha}=0$. This proves \Cref{eq: Omega =00003D H}.

$\left(4\right)$ First, it follows from \Cref{lem:derivative-two-point} using the push forward property of the DR cycles (\Cref{cor: push forward DR numbers}) that
\begin{equation}
\frac{\partial\Omega_{0,\alpha_{1};0,\alpha_{2}}}{\partial p_{0}^{1}}\Bigg|_{p_{*}^{*}=0}=\eta_{\alpha_{1}\alpha_{2}}.
\end{equation}
Then, \Cref{eq:initial condition two point} is an equality of differential polynomials, the equality of the degree $0$ follows from the choice of the constant for the two point functions given by \Cref{eq:constant-two point function}. To prove \Cref{eq:initial condition two point} in degree greater than zero, it is enough to prove that the $\partial_{x}$-derivative of the equation is satisfied. We have 
\begin{align}
\partial_{x}\frac{\partial\Omega_{d_{1},\alpha_{1};d_{2},\alpha_{2}}}{\partial p_{0}^{1}}=\frac{\partial}{\partial p_{0}^{1}}\partial_{x}\Omega_{p,\alpha_{1};q,\alpha_{2}} & =\frac{\partial}{\partial p_{0}^{1}}\frac{1}{\hbar}\left[H_{d_{1}-1,\alpha_{1}},\overline{H}_{d_{2},\alpha_{2}}\right]\\
 & =\frac{1}{\hbar}\left[\frac{\partial H_{d_{1}-1,\alpha_{1}}}{\partial p_{0}^{1}},\overline{H}_{d_{2},\alpha_{2}}\right]+\frac{1}{\hbar}\left[H_{d_{1}-1,\alpha_{1}},\frac{\partial\overline{H}_{d_{2},\alpha_{2}}}{\partial p_{0}^{1}}\right]\nonumber \\
 & =\partial_{x}\Omega_{d_{1}-1,\alpha_{1};d_{2},\alpha_{2}}+\partial_{x}\Omega_{d_{1},\alpha_{1};d_{2}-1,\alpha_{2}}.\nonumber 
\end{align}
\end{proof}
Fix $n\geq2$. Computing the correlator $\left\langle \tau_{0,1}\tau_{d_{1},\alpha_{1}}\cdots\tau_{d_{n},\alpha_{n}}\right\rangle _{l,g-l;k}$ boils down to compute 
\begin{align}
\frac{1}{\hbar^{n-1}}\big[\ldots\big[\Omega_{d_{1},\alpha_{1};d_{2},\alpha_{2}},\overline{H}_{d_{3},\alpha_{3}}\big]\ldots & ,\overline{H}_{d_{n},\alpha_{n}}\big],\overline{H}_{0,1}\big]\Big\vert_{u_{i}^{\gamma}=\delta_{i,1}\delta_{\gamma,1}}  = \\
& =\frac{1}{\hbar^{n-2}}\partial_{x}\left[\ldots\left[\Omega_{d_{1},\alpha_{1};d_{2},\alpha_{2}},\overline{H}_{d_{3},\alpha_{3}}\right]\ldots,\overline{H}_{d_{n},\alpha_{n}}\right]\Big\vert_{u_{i}^{\gamma}=\delta_{i,1}\delta_{\gamma,1}}\nonumber\\
 & =\frac{1}{\hbar^{n-2}}\frac{\partial}{\partial p_{0}^{1}}\left[\ldots\left[\Omega_{d_{1},\alpha_{1};d_{2},\alpha_{2}},\overline{H}_{d_{3},\alpha_{3}}\right]\ldots,\overline{H}_{d_{n},\alpha_{n}}\right]\Big\vert_{u_{i}^{\gamma}=\delta_{i,1}\delta_{\gamma,1}}\nonumber,
\end{align}
where we used first \Cref{eq: H_0 x-derivative}, then used $\partial_{x}=\sum_{i\geq0}u_{i+1}^{\alpha}\frac{\partial}{\partial u_{i}^{\alpha}}$ and then switched to the $p$-variable using the identification of the operator $\frac{\partial}{\partial u_{0}^{1}}=\frac{\partial}{\partial p_{0}^{1}}$ on differential polynomials. We then act with $\frac{\partial}{\partial p_{0}^{1}}$ on each element of the commutators and use \Cref{eq: string for hamiltonians} and \Cref{eq:initial condition two point} to find
\begin{align}
\frac{1}{\hbar^{n-2}} & \Bigg(\left[\ldots\left[\Omega_{d_{1}-1,\alpha_{1};d_{2},\alpha_{2}},\overline{H}_{d_{3},\alpha_{3}}\right]\ldots,\overline{H}_{d_{n},\alpha_{n}}\right]\Big\vert_{u_{i}^{\gamma}=\delta_{i,1}\delta_{\gamma,1}}\nonumber\\
&+ \left[\ldots\left[\Omega_{d_{1},\alpha_{1};d_{2}-2,\alpha_{2}},\overline{H}_{d_{3},\alpha_{3}}\right]\ldots,\overline{H}_{d_{n},\alpha_{n}}\right]\Big\vert_{u_{i}^{\gamma}=\delta_{i,1}\delta_{\gamma,1}}\nonumber\\
 & +\left[\ldots\left[\delta_{d_{1},0}\delta_{d_{2},0}\eta_{\alpha_{1}\alpha_{2}},\overline{H}_{d_{3},\alpha_{3}}\right]\ldots,\overline{H}_{d_{n},\alpha_{n}}\right]\Big\vert_{u_{i}^{\gamma}=\delta_{i,1}\delta_{\gamma,1}}\\
 & +\left[\ldots\left[\Omega_{d_{1},\alpha_{1};d_{2},\alpha_{2}},\overline{H}_{d_{3}-1,\alpha_{3}}\right]\ldots,\overline{H}_{d_{n},\alpha_{n}}\right]\Big\vert_{u_{i}^{\gamma}=\delta_{i,1}\delta_{\gamma,1}}\nonumber\\
 &+ \ldots \nonumber\\
 & +\left[\ldots\left[\Omega_{d_{1},\alpha_{1};d_{2},\alpha_{2}},\overline{H}_{d_{3},\alpha_{3}}\right]\ldots,\overline{H}_{d_{n}-1,\alpha_{n}}\right]\Big\vert_{u_{i}^{\gamma}=\delta_{i,1}\delta_{\gamma,1}}\Bigg).\nonumber
\end{align}

\smallskip

\subsubsection{Proof of \Cref{eq: quadratic term}.}

We have
\begin{align}
\left\langle \tau_{0,\alpha}\tau_{0,1}\right\rangle _{l,g-l;k} & =i^{-3g+1+k}{\rm Coeff}_{\epsilon^{2l}\hbar^{g-l}\mu^k}\underset{=H_{-1,\alpha}}{\underbrace{\Omega_{0,\alpha;0,1}}}\vert_{u_{i}^{\beta}=\delta_{i,1}\delta_{\beta,1}}
\end{align}
by \Cref{lem: small lemmas string}. Now using the evaluation lemma \cite[Lemma~6.2]{BDGR1}, this expression boils down to compute 
\begin{equation}
{\rm \Coeff}_{a_{1}a_{2}\cdots a_{m}}\int_{{\rm DR}_{g}\left(0,a_{1},\dots,a_{m},-\sum a_{i}\right)}\lambda_{l}c_{g,m+2}^{[\mu]}\left(v_{\alpha},v_{1},\dots,v_{1}\right)\qquad m\geq0.
\end{equation}
This expression vanishes for $m\geq1$ by the push forward property of the DR cycles (\Cref{cor: push forward DR numbers}). When $m=0$ the whole integrand is a pull-back by the map forgetting a marked point so the integral vanishes as well. Therefore, the correlator vanishes.

\smallskip 

\subsubsection{Proof of \Cref{eq: linear term}.}

We have 
\begin{equation}
\left\langle \tau_{0,1}\right\rangle _{l,g-l;k}=\left\langle \tau_{1,1}\tau_{0,1}\right\rangle _{l,g-l;k}=i^{-3g+2+k}{\rm \Coeff}_{\epsilon^{2l}\hbar^{g-l}\mu^k}\underset{=H_{0,1}}{\underbrace{\Omega_{1,1;0,1}}}\vert_{u_{i}^{\beta}=\delta_{i,1}\delta_{\beta,1}}.
\end{equation}
Now using the evaluation lemma \cite[Lemma~6.2]{BDGR1}, this expression boils down to compute 
\begin{equation}
{\rm \Coeff}_{a_{1}a_{2}\cdots a_{m}}\int_{{\rm DR}_{g}\left(0,a_{1},\dots,a_{m},-\sum a_{i}\right)}\psi_{1}\lambda_{l}c_{g,m+2}^{[\mu]}\left(v_{1},v_{1},\dots,v_{1}\right)\qquad m\geq0.
\end{equation}
This expression vanishes for $m\geq1$ using the push forward property of the DR cycles (\Cref{cor: push forward DR numbers}). For $m=0$, we have ${\rm DR}_{g}\left(0,0\right)=\left(-1\right)^{g}\lambda_{g}$ so we get
\begin{equation}
\left(-1\right)^{g}\int_{\overline{\mathcal{M}}_{g,2}}\psi_{1}\lambda_{g}\lambda_{l}c_{g,2}^{[\mu]}\left(v_{1},v_{1}\right)=\delta_{g,1}\delta_{l,0}\left(-1\right)^{g}\underset{=\delta_{k,0}\frac{N}{24}}{\underbrace{\int_{\overline{\mathcal{M}}_{1,1}}\lambda_{1}c_{1,1}^{[\mu]}\left(v_{1}\right)}}.
\end{equation}
Therefore, we obtain
$
\left\langle \tau_{0,1}\right\rangle _{l,g-l;k}=-\delta_{g,1}\delta_{l,0}\delta_{k,0}\frac{N}{24}.
$

\medskip

\subsection{Step 2: a geometric formula for the correlators with $\tau_{0,1}$-insertion} \label{sec:geo-fomula}

It follows from the string equation (\Cref{eq: String induction}) that 
\begin{align}
	\label{eq:csq-string-eq}
	\sum_{\substack{d_{1},\dots,d_{n}\geq0\\
\sum d_{i}=d
}
}\left\langle \tau_{d_{1},\alpha_{1}}\cdots\tau_{d_{n},\alpha_{n}}\right\rangle _{l,g-l;k}&\prod a_{i}^{d_{i}}  = \\ \notag
\frac{1}{\suma}\sum_{\substack{d_{1},\dots,d_{n}\geq0\\
\sum d_{i}=d+1
}
}&\left\langle \tau_{0,1}\tau_{d_{1},\alpha_1}\cdots\tau_{d_{n},\alpha_n}\right\rangle _{l,g-l;k}\prod a_{i}^{d_{i}}.
\end{align}

Thus, any correlator is determined by a linear combination of correlators with a $\tau_{0,1}$-insertion. Therefore, in what follows, we will only consider correlators with a $\tau_{0,1}$-insertion. The goal of this section is to prove \Cref{prop:geometric-formula-correlators} providing geometric formula for these correlators.

\subsubsection{Admissible modified stable graphs}
We define the set ${\rm AMSG}_{g,m+1}^{n_v}$ 
of admissible modified stable graphs (to be opposed to admissible stable graphs to be defined in the next section). It is the set of connected stable graphs of total genus $g$, with $n_v$ ordered vertices $v_1,\dots,v_{n_v}$, and with $n_v+m+1$ leaves $\ell_1,\dots,\ell_{n_v+m+1}$. Further conditions are the following:
\begin{itemize}
	\item the last leaf $\ell_{n_v+m+1}$ is attached to the first vertex $v_1$,
	\item the leaf $\ell_i$ is attached to $v_i$ for $1\leq i \leq n_v$,
	\item the remaining $m$ leaves, called the free leaves, can be attached on any vertex,
	\item each vertex is connected to at least one vertex of smaller index (except the first vertex) by an edge, this condition is called the admissibility condition,
	\item any edge connects two different vertices (that is, loops are forbidden).
\end{itemize}
Note that the last condition allows to endow each edge $e\in E(T)$ with a canonical orientation in the direction of the vertex with higher index. That is, if the head of the edge $h(e)$ is attached to $v_i$ and its tail $t(e)$ is attached to $v_j$, then we have $i>j$, for any $e\in E(T)$.

\subsubsection{Admissible decorations}
Fix $a_{1},\dots,a_{m}\in\mathbb{Z}_{>0}$. An admissible decoration of $\Gamma\in{\rm AMSG}_{g,m+1}^{n_v}$ is a function
\begin{align}
a (\cdot):H\left(\Gamma\right)\rightarrow\mathbb{Z}, 
\end{align}
recalling that $H(\Gamma)$ is the set of half edges including the leaves. The function  $a (\cdot)$ satisfies the following conditions:
\begin{itemize}
\item its value at the leaves is
	\begin{itemize}
	\item $a(\ell_i)=0$ if $1\leq i \leq n_v$,
	\item $a(\ell_i)=a_{i-n_v}$ if $n_v+1 \leq i \leq n_v+m$,
	\item $a(\ell_{n_v+m+1})=-\sum_{j=1}^{n} a_j $,
	\end{itemize}
\item for each edge $e$  of $\Gamma$ we impose
\begin{align}
a(h(e))+a(t(e))=0, \quad \rm{such\,that}\quad a(h(e))<0,\,\,a(t(e))>0,
\end{align}
\item for each vertex $v$ of $\Gamma$, denote $H(v)$ the set of half edges attached to $v$, we impose
\begin{align}
\sum_{h \in H(v))=v}a\left(h\right)=0,
\end{align}
where the sum is over half edges attached to $v$.
\end{itemize}
We denote by ${\rm AD}\left(\Gamma,a_{1},\dots,a_{n}\right)$ the set of admissible decorations of $\Gamma$.

\subsubsection{The geometric formula}
Fix the nonnegative integers $g,n_v,m$ and $l$ such that $0\leq l\leq g$, and fix $a_{1},\dots,a_{m}\in\mathbb{Z}_{>0}$. Let $\Gamma\in{\rm AMSG}_{g,m+1}^{n_v}$ and let $a\in{\rm AD}\left(\Gamma,a_{1},\dots,a_{m}\right)$ be an admissible decoration of $\Gamma$. We denote by
\begin{align}
{\rm DR}_{\Gamma}\left(a_{1},\dots,a_{m},-\sum a_{i}\right)
\end{align}
the cohomology class
\begin{align}
\left(\prod_{e \in E(\Gamma)}a(t(e))\right)\cdot(b_{\Gamma})_{*}\left(\prod_{v\in V\left(\Gamma\right)}{\rm DR}_{g\left(v\right)}\left(a\left(h\right)_{h\in H\left(v\right)}\right)\right)\in R^{\sum g_{i}+\left|E\left(\Gamma\right)\right|}\left(\overline{\mathcal{M}}_{g,n_v+m+1}\right),
\end{align}
where $(b_{\Gamma})_{*}$ is the boundary pushforward map from $\otimes_{v\in V\left(\Gamma\right)}R^{g\left(v\right)}\left(\overline{\mathcal{M}}_{g\left(v\right),\left|H\left(v\right)\right|}\right)$ to \\$R^{\sum g_{i}+\left|E\left(\Gamma\right)\right|}\left(\overline{\mathcal{M}}_{g,n_v+m+1}\right)$. 

\begin{prop}
\label{prop:geometric-formula-correlators}
Let $g,l,k$ and $n$ be nonnegative integers such that $l\leq g$. Let $d_{1},\dots,d_{n}$ be nonnegative integers and let $d=\sum d_i$. The correlator
\begin{align}
\left\langle \tau_{0,1}\tau_{d_{1},\alpha_1}\cdots\tau_{d_{n},\alpha_n}\right\rangle _{l,g-l;k}
\end{align}
is given by the coefficient of $a_{1}a_{2}\cdots a_{m}$ in 
\begin{align}
\sum_{\Gamma\in{\rm AMSG}_{g,m+1}^{n}}  \sum_{a\in AD(\Gamma,a_{1},\dots,a_{m})}&\frac{1}{|\mathrm{Aut}(\Gamma)|}\frac{1}{m!} \times \\
 & \times\int_{{\rm DR}_{\Gamma}\left(a_{1},\dots,a_{m},-\sum a_{i}\right)}\lambda_l c_{g,n+m+1;k}\left(\otimes_{i=1}^{n}v_{\alpha_{i}}\otimes v_{1}^{m+1}\right)\prod_{i=1}^{n}\psi_{i}^{d_{i}},\notag
\end{align}
where $m=d+l+k-2g+1$. 
\end{prop}

\begin{proof}

We have
\begin{align}
\left\langle \tau_{0,1}\tau_{d_{1},\alpha_{1}}\cdots\tau_{d_{n},\alpha_{n}}\right\rangle _{l,g-l}&= \label{eq:getting-geometric-formula}\\
\ii^{\sum d_i -3g+2-n+k}&{\rm \Coeff}_{\epsilon^{2l}\hbar^{g-l}\mu^k}\left(\frac{1}{\hbar^{n-1}}\left[\cdots\left[\Omega_{0,1;d_{1},\alpha_{1}},\overline{H}_{d_{2},\alpha_{2}}\right],\cdots,\overline{H}_{d_{n},\alpha_{n}}\right]\right)\Big|_{u_{i}^{\alpha}=\delta_{i,1}\delta_{\alpha,1}}.\nonumber
\end{align}
By \Cref{lem: small lemmas string}, we directly make the replacement $\Omega_{0,1;d_{1},\alpha_{1}}=H_{d_{1}-1,\alpha_{1}}$. We want to compute the right hand side using the evaluation lemma \cite[Lemma~6.2]{BDGR1}. We first introduce the non associative product $f\tilde{\star}g=f\star g-fg$, which is convenient for the following reason.
\begin{lem}
Let $m\geq0$ and $1\leq\beta_{1},\dots,\beta_{m}\leq N$. Suppose $a_{1},\dots,a_{m}>0$, then the coefficient of $p_{a_{1}}^{\beta_{1}}\cdots p_{a_{m}}^{\beta_{m}}\epsilon^{2l}\hbar^{g-l+n-1}\mu^ke^{i\sum a_{i}x}$ in
\begin{equation}
\left[\cdots\big[\left[H_{d_{1}-1,\alpha_{1}},\overline{H}_{d_{2},\alpha_{2}}\right],\overline{H}_{d_{3},\alpha_{3}}\big]\cdots,\overline{H}_{d_{n},\alpha_{n}}\right]
\end{equation}
and in
\begin{equation}
\big(\cdots\big(\big(H_{d_{1}-1,\alpha_{1}}\tilde{\star}\overline{H}_{d_{2},\alpha_{2}}\big)\tilde{\star}\overline{H}_{d_{3},\alpha_{3}}\big)\tilde{\star}\cdots \big)\tilde{\star}\overline{H}_{d_{n},\alpha_{n}}
\end{equation}
is the same.
\end{lem}

\begin{proof}
This occurs because, first, the bracket $[\cdot,\cdot]$ defined as the commutator of the $\star$-product is also the commutator of the $\tilde{\star}$-product. Second, we have
\begin{equation}
{\rm \Coeff}_{p_{a_{1}}^{\beta_{1}}\cdots p_{a_{m}}^{\beta_{m}}\epsilon^{2l}\hbar^{g-l+n-1}\mu^ke^{\ii\sum a_{i}x}}\overline{H}_{d_{n},\alpha_{n}}\tilde{\star}\left[\ldots\left[\left[H_{d_{1}-1,\alpha_{1}},\overline{H}_{d_{2},\alpha_{2}}\right],\overline{H}_{d_{3},\alpha_{3}}\right],\dots\right]=0,\quad a_{1},\dots,a_{m}>0.\label{eq:vanishing-star-product}
\end{equation}
Indeed, the term $\overline{H}_{d_{n},\alpha_{n}}\tilde{\star}\left(\cdots\right)$ involves a sum of a product of derivative $\frac{\partial}{\partial p_{k_{1}}^{\gamma_{1}}}\cdots\frac{\partial}{\partial p_{k_{s}}^{\gamma_{s}}}$, with $k_{1},\dots,k_{s}>0$ and $1\leq\gamma_{1},\dots,\gamma_{s}\leq N$, from the star product acting on $\overline{H}_{d_{n},\alpha_{n}}$, and since we use the $\tilde{\star}$-product we have $s\geq1$. From the definition of the Hamiltonian density (\Cref{eq:def-hamilto-density}), after using the integration map, we get 
\begin{align}
& \frac{\partial}{\partial p_{k_{1}}^{\gamma_{1}}}\cdots\frac{\partial}{\partial p_{k_{s}}^{\gamma_{s}}}\overline{H}_{d_{n},\alpha_{n}} 
\\ \notag 
& =\sum_{\substack{g,m\geq0\\
2g+m+s-1>0
}
}\frac{\left(i\hbar\right)^{g}}{m!}\sum_{a_{1}+\cdots+a_{m}+k_{1}+\cdots+k_{s}=0}p_{a_{1}}^{\alpha_{1}}\cdots p_{a_{m}}^{\alpha_{m}}\\ \notag 
 & \qquad \times\left(\int_{{\rm DR}_{g}\left(0,a_{1},\dots,a_{m},k_{1},\dots,k_{s}\right)}\psi_{1}^{d}\Lambda\left(\frac{-\epsilon^{2}}{i\hbar}\right)c_{g,m+s+1}^{[\mu]}\left(v_{\alpha},v_{\alpha_{1}},\dots,v_{\alpha_{m}},v_{\gamma_{1}},\dots,v_{\gamma_{s}}\right)\right),
\end{align}
where we used the pull-back property of the $\psi$-class to simplify the intersection number. Consequently, when extracting the coefficient of $p_{a_{1}}^{\beta_{1}}\cdots p_{a_{m}}^{\beta_{m}}\epsilon^{2l}\hbar^{g-l+n-1}\mu^ke^{i\sum a_{i}x}$ with the condition $a_{1},\dots,a_{m}>0$, the sum of the parts of each DR-cycle in $\frac{\partial}{\partial p_{k_{1}}^{\gamma_{1}}}\cdots\frac{\partial}{\partial p_{k_{s}}^{\gamma_{s}}}\overline{H}_{d_{n},\alpha_{n}}$ is positive, contradicting the requirement that they sum to zero. This proves the vanishing of \Cref{eq:vanishing-star-product}. Then, a direct induction justifies the statement.
\end{proof}
Therefore, computing \Cref{eq:getting-geometric-formula} using the evaluation lemma \cite[Lemma~6.2]{BDGR1} yields
\begin{align}
 & \label{eq:formula-end-geometric} \left\langle \tau_{0,1}\tau_{d_{1},\alpha_{1}}\cdots\tau_{d_{n},\alpha_{n}}\right\rangle _{l,g-l;k}=i^{\sum d_i -3g+2-n+k}\times\\
 & \times \sum_{m\geq0}\left(-i\right)^{m}{\rm \Coeff}_{a_{1}a_{2}\cdots a_{m}}\left({\rm \Coeff}_{p_{a_{1}}^{1}\cdots p_{a_{m}}^{1}\epsilon^{2l}\hbar^{g-l+n-1}\mu^ke^{i\sum a_{i}x}}\left(\cdots(H_{d_{1}-1,\alpha_{1}}\tilde{\star}\overline{H}_{d_{2},\alpha_{2}})\tilde{\star}\cdots\tilde{\star}\overline{H}_{d_{n},\alpha_{n}}\right)\right).\nonumber
\end{align}

The statement of the proposition is just a rewriting of \Cref{eq:formula-end-geometric} using graphs. More specifically, in \Cref{eq:formula-end-geometric} we first expand $H_{d_{1}-1,\alpha_{1}}$ using the definition of the Hamiltonian density (\Cref{eq:def-hamilto-density}). Then, we expand each Hamiltonian $\overline{H}_{d_{i},\alpha_{i}}$ using the same formula and the definition of the integration map, but we directly simplify the intersection numbers involved in $\overline{H}_{d_{i},\alpha_{i}}$ by the pull-back property of the $\psi$-class,  ensuring that each intersection number involves $\psi^{d_i}$ (instead of $\psi^{d_i+1}$). Finally, the $\tilde{\star}$-products are expanded using \Cref{eq:def-star-product}. Each term of this expansion is associated to a graph. 

	\begin{itemize}
		\item The vertices of these graphs are associated to each factor of the $\tilde{\star}$-products, therefore there are $n$ vertices indexed by $1,\dots,n$. The genus of the $i$-th vertex is the genus of the intersection number arising from the $i$-th term of the product.
		\item The half-edges attached to the $i$-th vertex correspond to the marked points of the $i$-th intersection number. Moreover, the value of the admissible decoration at a half-edge is the weight of the DR cycle at the corresponding marked point. 
		\item Each edge between the $i$-th vertex and the $j$-th vertex is associated to a conjugate derivative $\eta^{\alpha\beta}k\frac{\partial}{\partial p_{k}^{\alpha}}\frac{\partial}{\partial p_{-k}^{\beta}}$, which arises from a $\tilde{\star}$-product, acting on the $i$-th and the $j$-th terms in the product $(\cdots(H_{d_{1}-1,\alpha_{1}}\tilde{\star}\overline{H}_{d_{2},\alpha_{2}})\tilde{\star}\cdots\tilde{\star}\overline{H}_{d_{n},\alpha_{n}})$. More precisely, such a derivative, through its action on the $p$-variables, relate two marked points of the $i$-th and $j$-th intersection number, by imposing that they have opposite admissible decorations. 
		\item The first $m$ legs $\ell_1,\dots,\ell_m$ correspond to the $m$ marked points associated to the variables $p_{a_{1}}^{1}, $ $\dots ,p_{a_{m}}^{1}$. Notice that in \Cref{eq:formula-end-geometric}, since we extract these variables from the product $(\cdots(H_{d_{1}-1,\alpha_{1}}\tilde{\star}\overline{H}_{d_{2},\alpha_{2}})\tilde{\star}\cdots\tilde{\star}\overline{H}_{d_{n},\alpha_{n}})$, they can belong to any of the $n$ factors, thus these legs belong to any of the $n$ vertices. Moreover, the insertion of the CohFT at these marked points is $v_1$.
		\item Each of the $n$ next legs $\ell_{m+i}$ for $1\leq i \leq n$ corresponds to the marked point of the $i$-th intersection number supporting the $\psi$-class.
		\item The final leg, $\ell_{m+n+1}$, corresponds to the remaining half-edge: it represents the marked point of the first intersection number, which in \Cref{eq:def-hamilto-density} is the last marked point of the intersection number. Notice that this point is only present is the first intersection number because it arises from a Hamiltonian density (as opposed to Hamiltonian).
		\item The admissibility condition of the graphs follows from the non associativity of the $\tilde{\star}$-product. 
	\end{itemize}
	
	In addition, notice that the intersection numbers involved in the right hand side of \Cref{eq:formula-end-geometric} vanish if $m \neq d+l+k-2g+1$, thus the summation over $m$ is reduced to one element. This completes the proof.

\end{proof}

\medskip

\subsection{Step 3: simplification of the geometric formula by the push-forward lemma} \label{sec:simplification-push-formula}

We first define the set $\mathrm{ASG}_{g,n+1}^{n_v}$ of the so-called admissible stable graphs. In is the set of connected stable graphs of total genus $g$, with $n_v$ ordered vertices $v_1,\dots,v_{n_v}$, and with $n+1$ leaves $\ell_1,\dots,\ell_{n+1}$. Further conditions are the following:
\begin{itemize}
	\item The leaf $\ell_{n+1}$ is attached to the vertex $v_1$. 
	\item For any vertex $v_i$, $i=1,\dots,n_v$ there is at least one leaf $\ell_j$, $j=1,\dots,n$, attached to it. In particular, $1\leq n_v\leq n$. 
	\item Let $L(v)$ be the set of numbers of the leaves attached to the vertex $v$. In particular, the previous conditions mean that $L(v_{1})\ni n+1$ and $L(v_i)\ne\emptyset$, $i=1,\dots,n_v$. We require that if $i<j$, then $\min L(v_i) < \min L(v_j)$. In particular, $L(v_{1})\ni 1$.
	\item Each vertex is connected to at least one vertex of smaller index (except the first vertex) by an edge, this condition is called the admissibility condition.
	\item Any edge connects two different vertices (that is, loops are forbidden).
\end{itemize}
Note that the last condition allows to endow each edge $e\in E(T)$ with a canonical orientation in the direction of the vertex with higher index. That is, if $h(e)$ is attached to $v_i$ and $t(e)$ is attached to $v_j$, then we have $i>j$, for any $e\in E(T)$.

Let $\Gamma \in \mathrm{ASG}_{g,n+1}^{n_v}$, and let $a_1,\dots,a_n \in \ZZ_{>0}$. We are going to define the set $\mathrm{AD}(\Gamma,a_1,\dots,a_n)$ of admissible decorations of $\Gamma$ that also depends on $a_1,\dots,a_n$. It is a set of maps $a\colon H(\Gamma) \to \ZZ$ (recall that $H(\Gamma)$ denotes the set of half-edges of $\Gamma$ including its leaves) satisfying the following conditions:
\begin{itemize}
	\item $a(\ell_i) = a_i$, $i=1,\dots,n$.
	\item $a(\ell_{n+1}) = -(a_1+\cdots+a_n)$.
	\item $a(h(e)) + a(t(e)) = 0$, $a(h(e))<0$, $a(t(e))>0$, for any $e\in E$. 
	\item For any vertex $v_i$ let $H(v_i)$ be the set of half-edges (including leaves) attached to $v_i$. We require that $\sum_{h\in H(v_i)} a(h)=0$, $i=1,\dots,m$. 
\end{itemize}

Once again, we denote by
\begin{align}
{\rm DR}_{\Gamma}\left(a_{1},\dots,a_{n},-\sum a_{i}\right)
\end{align}
the cohomology class
\begin{align}
\left(\prod_{e \in E(\Gamma)}a(t(e))\right)\cdot(b_{\Gamma})_{*}\left(\prod_{v\in V\left(\Gamma\right)}{\rm DR}_{g\left(v\right)}\left(a\left(h\right)_{h\in H\left(v\right)}\right)\right),
\end{align}
where $(b_{\Gamma})_{*}$ is the boundary pushforward map from $\otimes_{v\in V\left(\Gamma\right)}R^{g\left(v\right)}\left(\overline{\mathcal{M}}_{g\left(v\right),\left|H\left(v\right)\right|}\right)$ to \\$R^{\sum g_{i}+\left|E\left(\Gamma\right)\right|}\left(\overline{\mathcal{M}}_{g,n+1}\right)$.

\begin{prop}
	\label{prop:simplification-geometric-formula}
	Let $g,l,k$ and $n$ be nonnegative integers such that $l\leq g$. Let $d_{1},\dots,d_{n}$ be nonnegative integers and let $d:=\sum d_i$. The correlator
$
\left\langle \tau_{0,1}\tau_{d_{1},\alpha_1}\cdots\tau_{d_{n},\alpha_n}\right\rangle _{l,g-l;k}
$
is given by extracting the coefficient of $a_1^{d_1}\cdots a_n^{d_n}$ in 
%
\begin{align}
\sum_{\substack{n_v,c_{1},\dots,c_{n_v}\\
1\leq n_v\leq n\\
c_{1},\dots,c_{n_v}\geq0\\
c_{1}+\cdots+c_{n_v}=d-m+(n-n_v)
}
}\sum_{\substack{\Gamma\in\mathrm{ASG}_{g,n+1}^{n_v}\\
a\in\mathrm{AD}(\Gamma,a_{1},\dots,a_{n})
}
}\frac{1}{|\mathrm{Aut}(\Gamma)|} & \times \label{eq:formula-before-the-end}\\
\times\int_{{\rm DR}_{\Gamma}\left(a_{1},\dots,a_{n},-\sum a_{i}\right)} & \lambda_{l}c_{g,n;k}\left(\otimes_{i=1}^{n_v}v_{\alpha_{i}}\otimes v_{1}\right)\prod_{i=1}^{n_v}\left(a_{\min L(v_{i})}\psi_{\min L\left(v_{i}\right)}\right)^{c_{i}}, \nonumber
\end{align}

where $m=d+l+k-2g+1$.
\end{prop}
\begin{proof}
	The proof of \Cref{prop:simplification-geometric-formula} follows from \Cref{prop:geometric-formula-correlators} after applying multiple times the push-forward property of DR cycles, forgetting each free leaf. The argument  is a straightforward generalisation of the argument of \cite[Section~6.5.2]{BDGR20} by replacing trees by graphs (more precisely we replace the set $AMST_{g,m+1}^{n_v}$ of admissible modified stable trees by the set $AMSG_{g,m+1}^{n_v}$ of admissible modified stable graphs, and we replace the set $AST_{g,n+1}^{n_v}$ of admissible stable trees by the set $ASG_{g,n+1}^{n_v}$ of admissible stable graphs), and by replacing the class $\lambda_{g}$ by the full Hodge class. This generalisation does not introduce new technical difficulties because the argument in \cite{BDGR20} relies solely on applying the push-forward property of the DR cycle (\Cref{cor: push forward DR numbers}) on each vertex of the graph, which, in their context, is applied to the moduli space of curves of compact type.
\end{proof}

\medskip

\subsection{Step 4: final relation}\label{sec:final-relation}

Next, we need a generalization of the relation given in~\cite[Section 6.5.3]{BDGR20}, combined with the argument in~\cite[Proof of Lemma~2.2]{DRDZ}.
We first introduce further notations.  For each vertex $v_i$ of a graph $\Gamma \in ASG^{n_v}_{n+1}$, we associate the class 
\begin{align}
	DR(v_i)\coloneqq DR_{g(v_i)}(\{a(h), h\in H(v_i)\}) \in R^g(\oM_{g(v_i),|H(v_i)|}).	
\end{align}
We have a freedom how to locally order the half-edges in $H(v_i)$; we assume the convention that the leaf $\ell_{\min L(v_i)}\in H(v_i)$ corresponds to the first marked point in $\oM_{g(v_i),|H(v_i)|}$.

\begin{prop} \label{prop:graph-dr-rubber}
	For any $\tilde{d}\geq 0$ we have:
	\begin{align} \label{eq:tilde-psi-0-dr-cycles}
		& s_* \left(\tilde\psi_0^{\tilde{d}} \left[\oM_{g}^{\sim}\left(\mathbb{P}^1,a_1,\dots,a_n,-\suma \right)\right]^{\mathrm{vir}}\right) 
		\\ \notag &
		= \sum_{\substack{n_v, c_1,\dots,c_{n_v}\\ 1\leq n_v\leq n \\ c_1,\dots,c_{n_v} \geq 0 \\ n_v-1+c_1+\cdots+c_{n_v}=\tilde{d}}} \sum_{\substack{\Gamma\in \mathrm{ASG}_{g,n+1}^{n_v} \\ a\in \mathrm{AD}(\Gamma,a_1,\dots,a_n)}} \frac{1}{|\mathrm{Aut}(\Gamma)|}\prod_{e\in E(\Gamma)} a(t(e)) (\mathsf{b}_\Gamma)_* \bigotimes_{i=1}^{n_v} DR(v_i) (a_{\min L(v_i)}\psi_1)^{c_i}
	\end{align}
\end{prop}

\begin{proof} We use the following formula from~\cite{BSSZ}:
\begin{align}\label{eq:tilde-psi-0}
	\tilde{\psi}_0 = a_i s^*\psi_i + D_i,
\end{align}
where $D_i$ is the sum of all irreducible divisors supported on the maps to the degenerate target such the irreducible component of the source curve that contains the $i$-th marked point is the unstable curve of genus $0$, cf.~\cite[first displayed formula in Section 2.2.6 and Lemma 2.6]{BSSZ}. 
This formula holds for any $i=1,\dots,n$.

To prove \Cref{eq:tilde-psi-0-dr-cycles}, we inductively compute each of the $\tilde{d}$ classes $\tilde{\psi}_{0}$ in
\begin{align}
\tilde{\psi}_{0}^{\tilde{d}}\left[\overline{\mathcal{M}}_{g}^{\sim}\left(\mathbb{P}^{1},a_{1},\dots,a_{n},-\suma\right)\right]^{vir}
\end{align}
 using \Cref{eq:tilde-psi-0}. In order to further simplify the computation we note that since we apply $s_*$ at the end, by \cite[Lemma~2.3]{BSSZ}, we can disregard all irreducible decorated strata in the space of the rubber maps, where the preimage of one irreducible component of the target curve has more than one stable components of the source curve. In practice this means that at each step of the computation, we can disregard the strata with more than one stable component in the preimage of each irreducible component of the target, except the one containing $0$.

 We start by describing the two first steps of our algorithm, before describing the general procedure. To compute the first power of $\tilde{\psi}_{0}$ we use \Cref{eq:tilde-psi-0} at the point $1$, we get 
\begin{align} \label{eq:psi-tilde-evaluated}
	\tilde{\psi}_{0}\left[\overline{\mathcal{M}}_{g}^{\sim}\left(\mathbb{P}^{1},a_{1},\dots,a_{n},-\suma\right)\right]^{vir}=a_{1}s^{*}\psi_{1}+D_{1}.
\end{align}
Next, to compute the second power $\tilde{\psi}_{0}$, we deal the with two summands separately. The first summand, $\tilde{\psi}_{0}\left(a_{1}s^{*}\psi_{1}\right)$, is computed again by using \Cref{eq:tilde-psi-0} at point $1$. Now recall that $D_{1}$ is a weighted sum over the following irreducible divisors: 
\begin{align}
\label{eq:irred-divisor}
\left[\overline{\mathcal{M}}_{g_{1}}^{\sim}\left(\mathbb{P}^{1},a_{I},-k_{1},\dots,-k_{p}\right)\right]^{vir}\boxtimes\left[\overline{\mathcal{M}}_{g_{2}}^{\sim}\left(\mathbb{P}^{1},a_{1},a_{J},k_{1},\dots,k_{p},-\suma\right)\right]^{vir}
\end{align}
for some $I\sqcup J=\left\{ 2,3,\dots,n\right\} $, $g_{1},g_{2}\geq0$ and $k_{1},\dots,k_{p}>0$ such that $g_{1}+g_{2}+p-1=g$. For each irreducible divisor, we use \Cref{eq:tilde-psi-0} at the point $i:=\min I$.  We now explain how to evaluate the $t$-th power of $\tilde{\psi}_{0}$ for $t=1,\dots,d$ in our algorithm. It will become clear that
\begin{align}
\tilde{\psi}_{0}^{t-1} \left[\overline{\mathcal{M}}_{g}^{\sim}\left(\mathbb{P}^{1},a_{1},\dots,a_{n},-\suma\right)\right]^{vir} =\sum_{X\in\mathcal{X}_{t-1}}\left[X\right]^{vir}
\end{align}
where $\mathcal{X}_{t}$ is a subset of the set of irreducible stratum in $\overline{\mathcal{M}}_{g}^{\sim}\left(\mathbb{P}^{1},a_{1},\dots,a_{n},-\suma\right)$ decorated by $\psi$-classes such that:
\begin{itemize}
\item the target curve has $t_{\nod}\geq0$ nodes,
\item each stratum is decorated with $t_{\psi}\geq0$ $\psi$-classes chosen among $s^*(\psi_1),\dots,s^*(\psi_n)$, 
\item we have $t_{\psi}+t_{\nod}=t-1$.
\end{itemize}
We choose to compute 
\begin{align}
\tilde{\psi}_{0}\cdot\left[X\right]^{vir},\qquad{\rm for}\;X\in\mathcal{X}_{t-1},
\end{align}
using \Cref{eq:tilde-psi-0} by setting 
\begin{align}
i_{t,X}:=\min\Stab_{t,X},
\end{align}
where $\Stab_{t,X}\subseteq\{1,\dots,n\}$ is the subset of indices corresponding to the preimages of $0$ that lie on stable components of the source curve (recall that the source curve is pre-stable).

To give an example, in the first step the set $\mathcal{X}_0$ contains just one decorated stratum, which is the whole space of rubber maps, and for this decorated stratum we choose $i_{1,X}=1$. For the second step, the set $\mathcal{X}_1$ contains the irreducible stratum (\ref{eq:irred-divisor}) as well as the whole space of rubber maps decorated by $a_1 s^* \psi_1$. 

\smallskip

We now use our algorithm to prove \Cref{eq:tilde-psi-0-dr-cycles}. The right hand side of~\eqref{eq:tilde-psi-0-dr-cycles} is just the bookkeeping device for all choices that we make in this procedure. To see this, we list the dictionary to match the possible choices throughout the computation and the data on the right hand side of~\eqref{eq:tilde-psi-0-dr-cycles}:
\begin{enumerate}
	\item The data $(n_v,c_1,\dots,c_{n_v})$ means that we choose the first summand in~\eqref{eq:tilde-psi-0} $c_1$ times (for the index $i=1$), then once the second summand (for $i=1$), then $c_2$ times the first summand (for a new index $i\not=1$), then once the second summand (for the same index and the preceding $c_2$ times), and so on. 
	\item The procedure above produces decorated strata of rubber maps to the $n_v$-component target curve $C_{n_v}\cup C_{n_v-1}\cup \cdots \cup C_1$, where $0\in C_{n_v}$, $\infty\in C_1$. For dimensional reasons the map $s_*$ is non-trivial only on those of the decorated strata that have exactly one stable component of the source curve over each component of the target curve. The vertex $v_i$ corresponds to the stable component of the source curve over $C_i$.
	\item The points $x_1,\dots,x_n$ correspond to the leaves $\ell_1,\dots,\ell_n$ correspond to the preimages of the multiplicity $a_1,\dots,a_n$ over $0$; the point $x_{n+1}$ corresponds to the leaf $\ell_{n+1}$ corresponds to the only preimage of $\infty$. 
	\item The $DR(v_i)$ emerge naturally from the stable components over $C_i$ under $s_*$, and they are equipped by $ (a_{\min L(v_i)}\psi_1)^{c_i}$ by our choices in the computation of $\tilde\psi_0^d$ above. 
	\item Each edge in $\Gamma$ that connects $v_j$ and $v_{j+k}$, $k\geq 1$, corresponds to a connected sequences of the unstable components of the source curves mapped to $C_{j+1}\cup C_{j+2}\cup \cdots \cup C_{j+k-1}$ (or just the preimages of the node if $k=1$).
	\item The coefficient $\frac{1}{|\mathrm{Aut}(\Gamma)|}\prod_{e\in E(\Gamma)} a(t(e))$ comes from Jun Li's degeneration formula~\cite{JunLi}, cf.~\cite[Proof of Lemma 2.6]{BSSZ}.
	\item Over the first component of the target curve $C_1$, the source curve has a unique component and the map has a total ramification of order $\suma$ over $\infty$. Therefore, the stable component over $C_i$ is connected via a node or a sequence of unstable components to a stable component over $C_j$ for some $1\leq j <i$. Thus, the associated graph is admissible. 
\end{enumerate}

With this dictionary we see that the right hand side of~\eqref{eq:tilde-psi-0-dr-cycles} is just a way to list all choices and all emerging decorated strata in the course of this algorithm to compute $s_* (\tilde\psi_0^{\tilde{d}} [\oM_{g}^{\sim}(\mathbb{P}^1,a_1,\dots,a_n,-\suma)]^{\mathrm{vir}}) $. This completes the proof of the proposition. 
\end{proof}

\begin{rem} Note that one can run the same argument for the computation of 
\begin{align}
			& s_* \left(\tilde\psi_0^{\tilde{d}} \lambda_g  \left[\oM_{g}^{\sim}\left(\mathbb{P}^1,a_1,\dots,a_n,-\suma\right)\right]^{\mathrm{vir}}\right), 
\end{align}
where $\lambda_g$ is lifted from the source curve. Then the only contributing graphs are the trees, and $\lambda_{g(v_i)}$ is present as an extra factor for each vertex $v_i$, $i=1,\dots,m$. This would give a new direct proof of~\cite[Lemma 2.2]{DRDZ}, bypassing the argument in op. cit. and in~\cite[Section 6.5.3]{BDGR20}.
\end{rem}


Assembling \Cref{eq:csq-string-eq} with the geometric formula for the correlator given in  \Cref{prop:simplification-geometric-formula}, and further simplifying by relation (\ref{eq:tilde-psi-0-dr-cycles}) for $\tilde{d}=2g-2+n-l-k$ gives 

\begin{align}
&\left\langle \tau_{d_{1},\alpha_{1}}\cdots\tau_{d_{n},\alpha_{n}}\right\rangle _{l,g-l;k}  =\\
 & \qquad \Coeff_{a_{1}^{d_{1}}\cdots a_{n}^{d_{n}}}\frac{1}{\suma}\int_{\overline{\mathcal{M}}_{g,n}}\pi_{*}\left(s_{*}\left(\lambda_{l}\tilde{\psi}_{0}^{\tilde{d}}\left[\oM_{g}^{\sim}\left(\mathbb{P}^{1},a_{1},\dots,a_{n},-\suma\right)\right]^{\mathrm{vir}}\right)\right)c_{g,n;k}\left(v_{\alpha_{1}},\dots,v_{\alpha_{n}}\right),\nonumber
\end{align}
where $\pi:\overline{\mathcal{M}}_{g,n+1}\rightarrow\overline{\mathcal{M}}_{g,n}$ forgets the last marked point. This completes the proof of \Cref{thm: quantum A class}.



\bigskip

\section{Proof of the \texorpdfstring{\hhh}{Hej}-representation for the correlators}
\label{sec:ProofOf-HEJ-Expression}

In this section we prove \Cref{thm:bounded-degree}.

\subsection{New notation for classes}

We need a new notation to treat the type of stable graphs that occur  in the course of the proof. 

Let $m,n\in\ZZ_{>0}$, $g\in\ZZ$. Let $a_1,\dots,a_n,b_1,\dots,b_m\in \ZZ_{>0}$. It is convenient to use notation $\vec{a}$, $\vec{b}$ for the sequences $(a_1,\dots,a_n)$, $(b_1,\dots,b_m)$.

Consider the space $\oM_{g-m+1-g'}^{\bullet,\sim}(\mathbb{P}^1,\vec a,-\vec b )\times \oM_{g',m+1}$ realized as the space of rubber stable maps with possibly disconnected domains, whose domains are the components of pre-stable curves of genus $g$ with $n+1$ marked points, such that the points corresponding to $-b_1,\dots,-b_m$ are the nodes separating the pre-stable domain of a rubber map and a stable curve in  $\oM_{g',m+1}$.  There is a natural extended source map 
\begin{align}
	exs\colon \oM_{g-m+1-g'}^{\bullet,\sim}(\mathbb{P}^1,\vec a,-\vec b)\times \oM_{g',m+1} \to \oM_{g,n+1},
\end{align}
where the first $n$ marked points correspond to $a_1,\dots,a_n$, and the last marked point is the last marked point in $\oM_{g',m+1} $ (the first $m$ marked points of the curves in this space correspond to $-b_1,\dots,-b_m$). Note that $g''\coloneqq g-m+1-g'$ might be negative. We, however, still demand that $2g''-2+n+m>0$, and we assume $2g'-1+m>0$. 


Consider a sequence of classes $\tilde{c} = \{\tilde{c}(p)\in R^*({LM}_{p},\QQ), p\geq 1\}$. Consider also a system of classes $C=\{C_g(b_1,\dots,b_m)\in R^*(\oM_{g,m+1},\QQ)\}$ be a system of classes that depends on $g, m, \vec b$. We introduce a new notation: Let
\begin{align}
	\tilde{c} \rhd C & \coloneqq \{(\tilde{c} \rhd C )(g,n,\vec a)\}; \\ \notag 
	(\tilde{c} \rhd C )(g,n,\vec a) & \coloneqq \sum_{\substack{m, g' \\ b_1,\dots,b_m}} \frac{\prod_{i=1}^m b_i}{|\mathrm{Aut}(\vec b)|} exs_*\Big( \tilde{c}(2g''-2+n+m)[\oM_{g''}^{\bullet,\sim}(\mathbb{P}^1,\vec a,-\vec b)]^{\mathrm{vir}} \\ \notag & \qquad \qquad \qquad \qquad \qquad \otimes C_{g'}(b_1,\dots,b_m)[\oM_{g',m+1}]\Big)
\end{align}
One variation of this notation that we also use below is the following:
\begin{align}
	\tilde{c} \rhd \!\!\! \rhd & \coloneqq \{(\tilde{c} \rhd \!\!\! \rhd )(g,n,\vec a)\}; \\ \notag 
	(\tilde{c} \rhd \!\!\! \rhd )(g,n,\vec a) & \coloneqq  s_*\Big( \tilde{c}(2g-1+n)[\oM_{g}^{\sim}(\mathbb{P}^1,\vec a,-\suma)]^{\mathrm{vir}} \Big).
\end{align}

${Comp}_{g,n,\vec a}[C]$ means that we extract the $(g,n,\vec a)$ component (which belongs to $R^*(\oM_{g,n+1},\QQ)$) of the respective system of classes $C$.  	
In order to further simplify the notation, we let $[C]_{\geq}$, resp. $[C]_{<}$ the classes $C'(g,n,\vec a)= ({Comp}_{g,n,\vec a}[C])_{\geq g}$, resp. $C''(g,n,\vec a)= ({Comp}_{g,n,\vec a}[C])_{< g}$, considered as a system of classes in $R^*(\oM_{g,n+1},\QQ)$ that depends on $g, n, \vec a$. 
In particular, let $\text{\hhh}$ (respectively, $\text{\hhh}_{\geq}$, $\text{\hhh}_{<}$) denote  $\text{\hhh}_{g,n+1}(\vec a,0)$ (respectively, $\text{\hhh}_{g,n+1}(\vec a,0)_{\geq g}$, $\text{\hhh}_{g,n+1}(\vec a,0)_{< g}$) considered as a system of classes in $R^*(\oM_{g,n+1},\QQ)$ that depends on $g, n, \vec a$.


\begin{ex} \label{ex:dr-cycle}
We have this way a new notation for the $\DR$-cycles:
\begin{align}
	{Comp}_{g,n,\vec a} \left[1 \rhd \!\!\! \rhd \right] = \DR_g(a_1,\dots,a_n,-\suma).
\end{align}
More generally, the quantum $A$-class is given in terms of 
\begin{align} \label{eq:A1-class}
		A^1_{g,n}(a_1,\dots,a_n)\coloneqq {Comp}_{g,n,\vec a} \left[\frac{\hbar^g\Lambda\Big(\frac{1}{\hbar}\Big)}{1-\tilde\psi_0} \rhd \!\!\! \rhd \right] 
\end{align}
as the push-forward $A_{g,n}(a_1,\dots,a_n) =  \frac{1}{\suma} \pi_* A^1_{g,n}(a_1,\dots,a_n)$.
\end{ex}

\begin{ex} \label{ex:localization}
	We have the following relation that comes from the localization statement proven in~\cite{BLS-Omega}:
	\begin{align}
		\left[ -\text{\hhh}+\frac{1}{1+\tilde\psi_\infty}  \rhd \!\!\! \rhd  +\frac{1}{1+\tilde\psi_\infty}  \rhd \, \text{\hhh}  \right]_{\geq} = 0.
	\end{align}
	By the way, note also that $\left[\frac{1}{1+\tilde\psi_\infty}  \rhd \!\!\! \rhd\right]_{\geq}=\left[\frac{1}{1+\tilde\psi_\infty}  \rhd \!\!\! \rhd\right]$.
\end{ex}

\medskip

\subsection{The key relation} In this section we prove the following proposition that gives us the key relation for the proof of Theorem~\ref{thm:bounded-degree}:

\begin{prop} \label{prop:key-relation} We have:
\begin{align} \label{eq:key-relation}
	\left[\frac{1}{1-\tilde\psi_0} \rhd \!\!\! \rhd	\right] - [\text{\hhh}]_{\geq} + \sum_{k=1}^\infty \underbrace{\bigg[\frac{1}{1+\tilde\psi_\infty} \rhd \cdots \bigg[\frac{1}{1+\tilde\psi_\infty} \rhd}_{k\ \text{times}} [\text{\hhh}]_{<} \underbrace{\bigg]_{\geq} \cdots  \bigg]_{\geq}}_{k\ \text{times}} = 0.
\end{align}
Note that the sum over $k$ is in fact finite for each $(g,n)$ for dimensional reason.
\end{prop}

\begin{proof}
First, recall a relation in the Losev-Manin space that reads
\begin{align}
	\frac{1}{1-\tilde \psi_0} -  \frac{1}{1+\tilde \psi_\infty} - \frac{1}{1-\tilde \psi_0} \Delta \frac{1}{1+\tilde \psi_\infty} = 0,
\end{align}
here $\Delta$ is the sum of all irreducible divisors. This relation implies the following identity:
\begin{align}
	0 & = \left[\frac{1}{1-\tilde \psi_0} \rhd \!\!\! \rhd -  \frac{1}{1+\tilde \psi_\infty} \rhd \!\!\! \rhd  - \frac{1}{1+\tilde \psi_\infty} \rhd \frac{1}{1-\tilde \psi_0}  \rhd \!\!\! \rhd \right]  
\end{align} 	
Using this identity and the one given in Example~\ref{ex:localization}, we see that 
\begin{align}
	& \left[\frac{1}{1-\tilde\psi_0} \rhd \!\!\! \rhd \right]- [\text{\hhh}]_{\geq }
	=  \left[\frac{1}{1-\tilde\psi_0} \rhd \!\!\! \rhd - \text{\hhh}\right]_{\geq } 
	= \Bigg[
	\frac{1}{1+\tilde\psi_\infty} \rhd \!\!\! \rhd  +\frac{1}{1+\tilde\psi_\infty}  \rhd \frac{1}{1-\tilde\psi_0} \rhd \!\!\! \rhd - \text{\hhh}  
	 \Bigg]_{\geq}
	\\ \notag & = \Bigg[\frac{1}{1+\tilde\psi_\infty}  \rhd \frac{1}{1-\tilde\psi_0} \rhd \!\!\! \rhd - \frac{1}{1+\tilde\psi_\infty}  \rhd \, \text{\hhh} 
	\Bigg]_{\geq} = \Bigg[\frac{1}{1+\tilde\psi_\infty}  \rhd \left( \frac{1}{1-\tilde\psi_0} \rhd \!\!\! \rhd - \text{\hhh}  \right)
	\Bigg]_{\geq}
	\\ \notag &= \Bigg[\frac{1}{1+\tilde\psi_\infty}  \rhd \left( \frac{1}{1-\tilde\psi_0} \rhd \!\!\! \rhd - [\text{\hhh} ]_{\geq}- [\text{\hhh} ]_{<} \right)
	\Bigg]_{\geq}.
\end{align}

Using this, we see that the left hand side of~\eqref{eq:key-relation} can be rewritten as
\begin{align}
		&
		\left[\frac{1}{1-\tilde\psi_0} \rhd \!\!\! \rhd	\right] - [\text{\hhh}]_{\geq} + \sum_{k=1}^\infty \underbrace{\bigg[\frac{1}{1+\tilde\psi_\infty} \rhd \cdots \bigg[\frac{1}{1+\tilde\psi_\infty} \rhd}_{k\ \text{times}} [\text{\hhh}]_{<} \underbrace{\bigg]_{\geq} \cdots  \bigg]_{\geq}}_{k\ \text{times}} 
		\\ \notag & = 
		\Bigg[\frac{1}{1+\tilde\psi_\infty}  \rhd \left( \left[\frac{1}{1-\tilde\psi_0} \rhd \!\!\! \rhd	\right] - [\text{\hhh}]_{\geq} + \sum_{k=1}^\infty \underbrace{\bigg[\frac{1}{1+\tilde\psi_\infty} \rhd \cdots \bigg[\frac{1}{1+\tilde\psi_\infty} \rhd}_{k\ \text{times}} [\text{\hhh}]_{<} \underbrace{\bigg]_{\geq} \cdots  \bigg]_{\geq}}_{k\ \text{times}} \right) \Bigg]_{\geq},
\end{align}
which proves Equation~\eqref{eq:key-relation} by induction.
\end{proof}

\medskip

\subsection{Corollaries for the intersections with the quantum \texorpdfstring{$A$}{A}-class} A direct corollary of Proposition~\ref{prop:key-relation} is the following formula for the quantum $A$-class $A_{g,n}(\vec a)$ defined in Definition~\ref{def:QuantumA}:

\begin{cor} \label{cor:difference-Aquantum-Hej} The difference
	$A_{g,n}(\vec a) - \hbar^g\Lambda\Big(\frac{1}{\hbar}\Big)(\text{\hhh}_{g,n}(\vec a))_{\geq (g-1)}$
	is equal to a tautological class obtained as a sum over rooted stable graphs where one vertex $v$ is decorated by $\hbar^{g'}\Lambda\Big(\frac{1}{\hbar}\Big) (\text{\hhh}_{g',n'}(\vec a'))_{< (g'-1)} $, where $\sum_{i=1}^{n'} a'_i = \suma$. 
\end{cor} 

\begin{proof} We just multiply Equation~\eqref{eq:key-relation} by $\frac 1\suma\hbar^g\Lambda\Big(\frac{1}{\hbar}\Big)$ and apply the push-forward. Note that these two operations commute. Note also that 
	\begin{align}
		\frac{1}{\suma} \pi_* \big(\hbar^g\Lambda\Big(\frac{1}{\hbar}\Big)(\text{\hhh}_{g,n+1}(\vec a,0))_{\geq g}\big) = \hbar^g\Lambda\Big(\frac{1}{\hbar}\Big)(\text{\hhh}_{g,n}(\vec a))_{\geq (g-1)}.
	\end{align}	
	The class
	\begin{align}
		\sum_{k=1}^\infty \underbrace{\bigg[\frac{1}{1+\tilde\psi_\infty} \rhd \cdots \bigg[\frac{1}{1+\tilde\psi_\infty} \rhd}_{k\ \text{times}} [\text{\hhh}]_{<} \underbrace{\bigg]_{\geq} \cdots  \bigg]_{\geq}}_{k\ \text{times}}
	\end{align}		 
	represented as a sum over decorated graphs has the vertex decorated by $[\text{\hhh}]_{<} = [\text{\hhh}_{g',n'+1}(\vec a',0)]_{<}$, where $\sum_{i=1}^{n'} a_i' = \suma$ and the leaf $\ell_{n+1}$ is attached to this vertex (and is labeled by $0$ in the argument of $\text{\hhh}$). Let's call it the root vertex. Note that the root vertex cannot have genus $0$. The push-forward acts only on this vertex it removing the leaf $\ell_{n+1}$ and replacing the class $(\text{\hhh}_{g',n'}(\vec a',0))_{< g'} \big)$ by $\pi_* (\text{\hhh}_{g',n'}(\vec a',0))_{< g'} = \suma (\text{\hhh}_{g',n'}(\vec a'))_{< (g'-1)}$. Multiplying the whole graph by $\frac 1\suma\hbar^g\Lambda\Big(\frac{1}{\hbar}\Big)$ we factorize $\Lambda\Big(\frac{1}{\hbar}\Big)$ over the vertices, so in particular we get such a factor for the decoration of the root vertex as well.
\end{proof}

An immediate corollary is the following one:
\begin{cor} \label{cor:VanishingExtraTermsUnderAssumption} Let $\{c_{g,n}\}$ be any CohFT such that
	\begin{align} \label{eq:assumption-CohFT-degree-hej}
		\int_{\oM_{g,n}} \hbar^g\Lambda\Big(\frac{1}{\hbar}\Big)(\text{\hhh}_{g,n}(\vec a))_{< (g-1)} c_{g,n} (v_{\alpha_1}\otimes\cdots\otimes v_{\alpha_n}) = 0 
	\end{align}
	for any $g$, $n$, $a_1,\dots,a_n$ and $\alpha_1,\dots,\alpha_n$. Then we have
	\begin{align}
		\int_{\oM_{g,n}}  A_{g,n}(\vec a) c_{g,n} (v_{\alpha_1}\otimes \cdots \otimes v_{\alpha_n}) = \int_{\oM_{g,n}}  \hbar^g\Lambda\Big(\frac{1}{\hbar}\Big) \text{\hhh}_{g,n}(\vec a) c_{g,n} (v_{\alpha_1}\otimes\cdots\otimes v_{\alpha_n}).
	\end{align}
\end{cor} 

\begin{proof} Indeed, by assumption~\eqref{eq:assumption-CohFT-degree-hej} we have
	\begin{align}
		\int_{\oM_{g,n}}  \hbar^g\Lambda\Big(\frac{1}{\hbar}\Big) \text{\hhh}_{g,n}(\vec a) c_{g,n} (v_{\alpha_1}\otimes\cdots\otimes v_{\alpha_n}) = \int_{\oM_{g,n}}  \hbar^g\Lambda\Big(\frac{1}{\hbar}\Big) (\text{\hhh}_{g,n}(\vec a))_{\geq g-1} c_{g,n} (v_{\alpha_1}\otimes\cdots\otimes v_{\alpha_n}).
	\end{align}
	Furthermore, by Corollary~\ref{cor:difference-Aquantum-Hej}  the integral of the class $A_{g,n}(\vec a) - \hbar^g\Lambda\Big(\frac{1}{\hbar}\Big) (\text{\hhh}_{g,n}(\vec a))_{\geq g-1}$ is represented as a sum of tautological classes represented by stable graphs, where at lease one vertex is decorated by $\hbar^{g'}\Lambda\Big(\frac{1}{\hbar}\Big)(\text{\hhh}_{g',n'}(\vec a'))_{< (g'-1)}$. The factorization property of cohomological field theories implies that the integral of $c_{g,n} (v_{\alpha_1}\otimes\cdots\otimes v_{\alpha_n})$ over each such summand contains a factor that vanishes by assumption~\eqref{eq:assumption-CohFT-degree-hej}.
\end{proof}

Now recall that in 
Theorem~\ref{thm:bounded-degree} we assume that the CohFT satisfies the degree condition (\ref{eq:DegreeAssumption}). Its proof is a direct application of Corollary~\ref{cor:VanishingExtraTermsUnderAssumption} using that condition.

\begin{proof}[Proof of Theorem~\ref{thm:bounded-degree}]
Fix a cohomological field theory $c=\{c_{g,n}\}$ satisfying the degree condition $\deg c_{g,n}(v_{\alpha_1}\otimes\cdots\otimes v_{\alpha_n}) < g-1+n$ for all $g$, $n$, $\alpha_1,\dots,\alpha_n$ (that is condition (\ref{eq:DegreeAssumption})). This immediately implies Equation~\eqref{eq:assumption-CohFT-degree-hej}. Indeed, note that $\deg \Lambda\Big(\frac{1}{\hbar}\Big)\leq g$, $\deg (\text{\hhh}_{g,n}(\vec a))_{< (g-1)}\leq g-2$, so if $\deg c_{g,n}(v_{\alpha_1}\otimes\cdots\otimes v_{\alpha_n}) \leq g-2+n$, then
\begin{align}
	\deg  \Lambda\Big(\frac{1}{\hbar}\Big)(\text{\hhh}_{g,n}(\vec a))_{< (g-1)} c_{g,n} (v_{\alpha_1}\otimes\cdots\otimes v_{\alpha_n}) \leq 3g-4+n < \dim \oM_{g,n},
\end{align}
which implies~\eqref{eq:assumption-CohFT-degree-hej}. Thus we can use Corollary~\ref{cor:VanishingExtraTermsUnderAssumption}.  It implies that 
\begin{align}
	\int_{\oM_{g,n}}  A_{g,n}(\vec a) c_{g,n} (v_{\alpha_1}\otimes\cdots\otimes v_{\alpha_n}) = 
	 \int_{\oM_{g,n}}  \hbar^g\Lambda\Big(\frac{1}{\hbar}\Big) \text{\hhh}_{g,n}(\vec a) c_{g,n} (v_{\alpha_1}\otimes\cdots\otimes v_{\alpha_n}),
\end{align}
which, in particular, means that the right hand side must be polynomial in $\hbar$ and $a_1,\dots,a_n$ (since the left hand side does). Thus, by Theorem~\ref{thm: quantum A class}, we have:
\begin{align}
	\langle \tau_{d_{1},\alpha_1}\dots\tau_{d_{n},\alpha_n}\rangle_{l,g-l}
	& =
	\mathrm{\Coeff}_{\hbar^{g-l} \prod_{i=1}^n a_i^{d_i}} 
	\int_{\oM_{g,n}}  A_{g,n}(\vec a) c_{g,n} (v_{\alpha_1}\otimes\cdots\otimes v_{\alpha_n})  
	\\ \notag &
	=
	\mathrm{\Coeff}_{\hbar^{g-l} \prod_{i=1}^n a_i^{d_i}} 
	\int_{\oM_{g,n}}  \hbar^g\Lambda\Big(\frac{1}{\hbar}\Big) \text{\hhh}_{g,n}(\vec a) c_{g,n} (v_{\alpha_1}\otimes\cdots\otimes v_{\alpha_n})
	\\ \notag
	& =
	\mathrm{\Coeff}_{\prod_{i=1}^n a_i^{d_i}} 
	\int_{\oM_{g,n}}  \lambda_l \text{\hhh}_{g,n}(\vec a) c_{g,n} (v_{\alpha_1}\otimes\cdots\otimes v_{\alpha_n})
\end{align}
This completes the proof of the theorem. 
\end{proof}






\bigskip

\section{Proof of the dilaton equation}
\label{sec:ProofDilaton}

In this section we prove the dilaton equation (\Cref{thm:StringANDDilaton}). Equivalently, we prove 
\begin{equation}
\langle\tau_{1,1}\tau_{d_{1},\alpha_{1}}\cdots\tau_{d_{n},\alpha_{n}}\rangle_{l,g-l;k}=(2g-2+n)\langle\tau_{d_{1},\alpha_{1}}\cdots\tau_{d_{n},\alpha_{n}}\rangle_{l,g-l;k},\label{eq:dilaton-correlators}
\end{equation}
when $2g-2+n>0$ and 
\begin{align}
\left\langle \tau_{1,1}\tau_{d_{1},\alpha_{1}}\tau_{d_{2},\alpha_{2}}\right\rangle _{0,0;k} & =0,\label{eq:genus-0-dilaton}\\
\left\langle \tau_{1,1}\right\rangle _{1,0;k} & =\frac{N}{24}\delta_{k,0},\label{eq:genus-1-classical-dilaton}\\
\left\langle \tau_{1,1}\right\rangle _{0,1;k} & =\delta_{k,1}\int_{\overline{\mathcal{M}}_{1,1}}c_{1,1;1}\left(1\right).\label{eq:genus-1-quantum-dilaton}
\end{align}

We begin by addressing the straightforward cases. \Cref{eq:genus-0-dilaton} and \Cref{eq:genus-1-classical-dilaton} are classical correlators, their expression was already obtained in \cite[Eqs.~6.14~and~6.15]{BDGR1} . We compute $\left\langle \tau_{1,1}\right\rangle _{0,1}$ by 
\begin{align}
	\left\langle \tau_{1,1}\right\rangle _{0,1;k}=\left\langle \tau_{0,1}\tau_{2,1}\right\rangle _{0,1;k}=i^{k}{\rm Coeff}_{\hbar\mu^k}\Omega_{0,1;2,1}\vert_{u_{i}^{\alpha}=\delta_{i,1}\delta_{\alpha,1}}.
\end{align}
\Cref{lem: small lemmas string} gives $\Omega_{0,1;1,1}=H_{1,1}$, and by the evaluation lemma \cite[Lemma~6.2]{BDGR1} gives 
\begin{align}
	\left\langle \tau_{1,1}\right\rangle _{0,1;k}=i^{k+1}\sum_{m\geq0}\left(-i\right)^{m}{\rm Coeff}_{a_{1}a_{2}\cdots a_{m}}\int_{{\rm DR}_{1}\left(0,a_{1},\dots,a_{m},-\sum a_{i}\right)}\psi_{1}^{2}c_{1,m+2;k}\left(v_{1},\dots,v_{1}\right).
\end{align}
The $m>2$ terms vanish by the push-forward property of DR cycles (\Cref{cor: push forward DR numbers}). The $m=2$ term is, by the same property, given by 
\begin{align}
\delta_{k,1}{\rm Coeff}_{\suma^2}\int_{{\rm DR}_{1}\left(\suma,-\suma\right)}c_{1,2;1}\left(v_{1},v_{1}\right),
\end{align}
and this last expression equals $\int_{\overline{\mathcal{M}}_{1,1}}c_{1,1}\left(v_{1}\right)$ by \cite[Lemma~5.4]{BDGR1}. The $m=1$ term vanishes because the integral is an even polynomial in $a_{1}$, thus it has no linear term. Finally the $m=0$ vanishes by a direct dimensional counting. 

Now, \Cref{eq:dilaton-correlators} holds by definition when $n=0$. Suppose, \Cref{eq:dilaton-correlators} holds for $n\geq2$, we get using the string equation 
\begin{align}
\left\langle \tau_{1,1}\tau_{d,\alpha}\right\rangle _{l,g-l;k} & =\left\langle \tau_{0,1}\tau_{1,1}\tau_{d+1,\alpha}\right\rangle _{l,g-l;k}-\left\langle \tau_{0,1}\tau_{d+1,\alpha}\right\rangle _{l,g-l;k} \nonumber\\
 & =2g\left\langle \tau_{0,1}\tau_{d+1,\alpha}\right\rangle _{l,g-l;k}-\left\langle \tau_{0,1}\tau_{d+1,\alpha}\right\rangle _{l,g-l;k}\\
 & =\left(2g-1\right)\left\langle \tau_{0,1}\tau_{d+1,\alpha}\right\rangle _{l,g-l;k}.\nonumber
\end{align}

Therefore, we are left with the proof of \Cref{eq:dilaton-correlators} for $n\geq2$ that we address in the rest of the section. 

\medskip

\subsection{Reduction of dilation using string}
\label{sec:ProofDilaton-String}

The first statement of \Cref{thm:StringANDDilaton}, the string equation, is proven in~\Cref{sec:proof-string}. The second statement of \Cref{thm:StringANDDilaton}, the dilaton equation, is a corollary of \Cref{thm: quantum A class}, though it requires some computation. To this end, we can reformulate the dilaton equation as 
\begin{align}
	& \mathrm{\Coeff}_{b^1}\int_{\overline{\mathcal{M}}_{g,n+1}}A_{g,n+1}(a_1,\dots,a_n,b) c_{g,n+1}^{[\mu]}(v_{\alpha_1} \otimes \dots \otimes v_{\alpha_n}\otimes v_1) 
	\\ \notag & = 
	(2g-2+n)\int_{\overline{\mathcal{M}}_{g,n}}A_{g,n}(a_1,\dots,a_n) c_{g,n}^{[\mu]}(v_{\alpha_1} \otimes \dots \otimes v_{\alpha_n}).
\end{align}

However, it is easier to use not the quantum $A$-class, by the class $A^{1}_{g,n}$ defined in~\eqref{eq:A1-class}. Indeed, once we have established the string equation, the dilaton one is a formal corollary of its special case given by
\begin{align}
\langle \tau_{d_{1},\alpha_1}\dots\tau_{d_{n},\alpha_n}\tau_{1,1}\tau_{0,1}\rangle_{l,g-l;k}
= (2g-1+n) \langle \tau_{d_{1},\alpha_1}\dots\tau_{d_{n},\alpha_n}\tau_{0,1}\rangle_{l,g-l;k}.
\end{align}
This special case of the dilaton equation we can reformulate as 
\begin{align} \label{eq:Dilaton-Reduced-A1}
	& \mathrm{\Coeff}_{b^1}\int_{\overline{\mathcal{M}}_{g,n+2}}A^1_{g,n+1}(a_1,\dots,a_n,b) c_{g,n+2}^{[\mu]}(v_{\alpha_1} \otimes \dots \otimes v_{\alpha_n}\otimes v_1\otimes v_1) 
	\\ \notag & = 
	(2g-1+n)\int_{\overline{\mathcal{M}}_{g,n}}A^1_{g,n}(a_1,\dots,a_n) c_{g,n+1}^{[\mu]}(v_{\alpha_1} \otimes \dots \otimes v_{\alpha_n}\otimes v_1).
\end{align}

\medskip

\subsection{Extracting \texorpdfstring{$\tilde \psi_0$}{tilde-psi-0}} 

Using~\eqref{eq:A1-class}, the left hand side of Equation~\eqref{eq:Dilaton-Reduced-A1} can be expanded as 
\begin{align} \label{eq:Prof-Dil-Step1}
	\mathrm{\Coeff}_{b^1}\int_{\overline{\mathcal{M}}_{g,n+2}} s_* \left(\frac{\hbar^g\Lambda\Big(\frac{1}{\hbar}\Big)}{1-\tilde\psi_0} \left[\oM_{g}^{\sim}\left(\mathbb{P}^1,a_1,\dots,a_n,b,-\suma-b \right)\right]^{\mathrm{vir}}\right) c_{g,n+2}^{[\mu]}(v_{\alpha_1} \otimes \dots \otimes v_{\alpha_n}\otimes v_1\otimes v_1).
\end{align}

Let $p\colon\oM_{g,n+2} \to \oM_{g,n+1}$ be the morphism forgetting the $(n+1)$-st marked point (labeled by the multiplicity $b$ and the primary field $v_1$ in \eqref{eq:Prof-Dil-Step1}). Using the unit axiom of cohomological field theories, rewrite \eqref{eq:Prof-Dil-Step1} as 
\begin{align} \label{eq:Prof-Dil-Step2}
	& \mathrm{\Coeff}_{b^1}\int_{\overline{\mathcal{M}}_{g,n+2}}
	  s_* \left(\frac{1}{1-\tilde\psi_0} \left[\oM_{g}^{\sim}\left(\mathbb{P}^1,a_1,\dots,a_n,b,-\suma-b \right)\right]^{\mathrm{vir}}\right)
	\\ \notag  & \qquad \qquad \qquad \qquad 
	\cdot p^* c_{g,n+1}^{[\mu]}(v_{\alpha_1} \otimes \dots \otimes v_{\alpha_n}\otimes v_1) \cdot p^* \Lambda\Big(\frac{1}{\hbar}\Big)\hbar^g.
\end{align}
Furthermore, expand $(1-\tilde\psi_0)^{-1}$ as $1+ (1-\tilde\psi_0)^{-1} \tilde \psi_0$ and use Equation~\eqref{eq:tilde-psi-0} to replace the last factor of $\tilde \psi_0$ by $b\cdot s^*\psi_{n+1} + D_{n+1}$. We see that~\eqref{eq:Prof-Dil-Step2} is equal to:
\begin{align} \label{eq:Prof-Dil-Step3}
	& \mathrm{\Coeff}_{b^1}\int_{\overline{\mathcal{M}}_{g,n+2}} s_* \left( \left[\oM_{g}^{\sim}\left(\mathbb{P}^1,a_1,\dots,a_n,b,-\suma-b \right)\right]^{\mathrm{vir}}\right)
	\\ \notag &
	\qquad \qquad \qquad \qquad 
	\cdot p^* c_{g,n+1}^{[\mu]}(v_{\alpha_1} \otimes \dots \otimes v_{\alpha_n}\otimes v_1) \cdot p^* \Lambda\Big(\frac{1}{\hbar}\Big)\hbar^g
	\\ \notag &
	+ \mathrm{\Coeff}_{b^1}\int_{\overline{\mathcal{M}}_{g,n+2}}  s_* \left( \frac{1}{1-\tilde\psi_0} \left(b\cdot s^*\psi_{n+1} \right) \left[\oM_{g}^{\sim}\left(\mathbb{P}^1,a_1,\dots,a_n,b,-\suma-b \right)\right]^{\mathrm{vir}}\right)
	\\ \notag & 
	\qquad \qquad \qquad \qquad 
	\cdot p^* c_{g,n+1}^{[\mu]}(v_{\alpha_1} \otimes \dots \otimes v_{\alpha_n}\otimes v_1) \cdot p^* \Lambda\Big(\frac{1}{\hbar}\Big)\hbar^g
	\\ \notag &
	+ \mathrm{\Coeff}_{b^1}\int_{\overline{\mathcal{M}}_{g,n+2}}  s_* \left( \frac{1}{1-\tilde\psi_0} \left( D_{n+1}\right) \left[\oM_{g}^{\sim}\left(\mathbb{P}^1,a_1,\dots,a_n,b,-\suma-b \right)\right]^{\mathrm{vir}}\right)
	\\ \notag & 
	\qquad \qquad \qquad \qquad 
	\cdot p^* c_{g,n+1}^{[\mu]}(v_{\alpha_1} \otimes \dots \otimes v_{\alpha_n}\otimes v_1) \cdot p^* \Lambda\Big(\frac{1}{\hbar}\Big)\hbar^g.	
\end{align}
We analyze the three summands in~\eqref{eq:Prof-Dil-Step3} separately in the next subsections. 

\medskip

\subsection{Vanishing terms}
The first summand in~\eqref{eq:Prof-Dil-Step3} can be rewritten using the definition of the double ramification cycle and the projection formula as 
\begin{align} \label{eq:Prof-Dil-Step4}
	& \mathrm{\Coeff}_{b^1}\int_{\overline{\mathcal{M}}_{g,n+1}}  p_*
	\DR_g \left(a_1,\dots,a_n,b,-\suma-b \right)
	\cdot c_{g,n+1}^{[\mu]}(v_{\alpha_1} \otimes \dots \otimes v_{\alpha_n}\otimes v_1) \cdot  \Lambda\Big(\frac{1}{\hbar}\Big)\hbar^g
\end{align}
By \Cref{prop:DR-pushforward-divisible-bb} the latter expression is equal to $\mathrm{\Coeff}_{b^1} O(b^2) =0 $.

The third summand can be represented as
\begin{align} \label{eq:Prof-Dil-Step5}
	& \mathrm{\Coeff}_{b^1} \sum_{k=1}^\infty  \sum_{\substack{g_1\in \ZZ, g_2 \in \ZZ_{\geq 0} \\ g_1+g_2+k-1 = g  
	}} \sum_{\substack{\bar a_1,\dots,\bar a_k \geq 1 \\ \bar a_1+\cdots+\bar a_k \\ = 
			\suma
	}}  \frac{\hbar^{k-1}}{k!} \prod_{i=1}^k \eta^{\beta_i\gamma_i} \prod_{i=1}^k \bar a_i
	\\ \notag &
	\times \int_{\left[\oM_{g_1}^{\bullet,\sim}\left(\mathbb{P}^1,a_1,\dots,a_n,-\bar a_1,\dots,-\bar a_k \right)\right]^{\mathrm{vir}}}
	\frac{1}{1-\tilde\psi_0} 
	\cdot s^* c^{[\mu],\bullet}_{g_1,n+k}\left(\bigotimes_{i=1}^n v_{\alpha_i} \otimes \bigotimes_{i=1}^k v_{\beta_i}\right) \cdot s^* \Lambda^\bullet(\frac{1}{\hbar})\hbar^{g_1}
	\\ \notag &
	\times \int_{\overline{\mathcal{M}}_{g_2,k+2}}  s_* \left( \left[\oM_{g_2}^{\sim}\left(\mathbb{P}^1,\bar a_1,\dots,\bar a_k,b,-\suma-b \right)\right]^{\mathrm{vir}}\right)
	\cdot p^* c_{g_2,k+1}^{[\mu]}(v_{\gamma_1} \otimes \dots \otimes v_{\gamma_k}\otimes v_1) \cdot p^* \Lambda\Big(\frac{1}{\hbar}\Big)\hbar^{g_2}.
\end{align}
Here by $c^{[\mu],\bullet}$, $\Lambda^\bullet$ we mean the corresponding classes on moduli spaces of the not necessarily connected curves that are the source curves in $\oM_{g_1}^{\bullet,\sim}\left(\mathbb{P}^1,a_1,\dots,a_n,-\bar a_1,\dots,-\bar a_k \right)$. Note that some components of the source curves might be unstable, then we formally assign to them in this formula $c_{0,2}(v_\alpha\otimes v_\beta) \coloneqq \eta_{\alpha\beta}$. 

Now we have two different cases: $(g_2,k+2)=(0,3)$ and $(g_2,k+2)\not= (0,3)$. In the first case, both factors in~\eqref{eq:Prof-Dil-Step5} are constant in $b$, thus this term doesn't contribute to $\mathrm{\Coeff}_{b^1}$. In the second case, the first factor in~\eqref{eq:Prof-Dil-Step5} is constant in $b$ and the second factor can be rewritten as 
\begin{align} \label{eq:Prof-Dil-Step6}
& \int_{\overline{\mathcal{M}}_{g_2,k+2}}  s_* \left( \left[\oM_{g_2}^{\sim}\left(\mathbb{P}^1,\bar a_1,\dots,\bar a_k,b,-\suma-b \right)\right]^{\mathrm{vir}}\right)
	\cdot p^* c_{g_2,k+1}^{[\mu]}(v_{\gamma_1} \otimes \dots \otimes v_{\gamma_k}\otimes v_1) \cdot p^* \Lambda\Big(\frac{1}{\hbar}\Big)\hbar^{g_2}
\\ \notag 
& = \int_{\overline{\mathcal{M}}_{g_2,k+1}}  p_*
\DR_{g_2} \left(\bar a_1,\dots,\bar a_k,b,-\suma-b \right)
\cdot c_{g_2,k+1}^\mu(v_{\gamma_1} \otimes \dots \otimes v_{\gamma_k}\otimes v_1) \cdot  \Lambda\Big(\frac{1}{\hbar}\Big)\hbar^{g_2}
= O(b^2),
\end{align}
where for the second equality we use \Cref{prop:DR-pushforward-divisible-bb}. This implies that \eqref{eq:Prof-Dil-Step5} vanishes.

\medskip

\subsection{Projection formula for the \texorpdfstring{$\psi$}{psi}-class} Using the vanishing statements established in the previous subsection, we see that~\eqref{eq:Prof-Dil-Step3} is equal to 
\begin{align} \label{eq:Prof-Dil-Step7}
	& \mathrm{\Coeff}_{b^0}\int_{\overline{\mathcal{M}}_{g,n+2}}  s_* \left( \frac{1}{1-\tilde\psi_0} s^*\psi_{n+1}  \left[\oM_{g}^{\sim}\left(\mathbb{P}^1,a_1,\dots,a_n,b,-\suma-b \right)\right]^{\mathrm{vir}}\right)
	\\ \notag & 
	\qquad \qquad \qquad \qquad 
	\cdot p^* c_{g,n+1}^{[\mu]}(v_{\alpha_1} \otimes \dots \otimes v_{\alpha_n}\otimes v_1) \cdot p^* \Lambda\Big(\frac{1}{\hbar}\Big)\hbar^g
	\\ \notag & 
	=\int_{\overline{\mathcal{M}}_{g,n+2}}  s_* \left( \frac{1}{1-\tilde\psi_0} s^*\psi_{n+1}  \left[\oM_{g}^{\sim}\left(\mathbb{P}^1,a_1,\dots,a_n,0,-\suma \right)\right]^{\mathrm{vir}}\right)
	\\ \notag & 
	\qquad \qquad \qquad \qquad 
	\cdot p^* c_{g,n+1}^{[\mu]}(v_{\alpha_1} \otimes \dots \otimes v_{\alpha_n}\otimes v_1) \cdot p^* \Lambda\Big(\frac{1}{\hbar}\Big)\hbar^g.
\end{align}
In this case we can use once again that $(1-\tilde\psi_0)^{-1} = 1+ (1-\tilde\psi_0)^{-1} \tilde \psi_0$ and then use the formula for the evaluation of $\tilde \psi_0$ in the presence of a point of multiplicity $0$ (cf.~\cite[Proof of Theorem 5]{BSSZ}, more precisely, the expression for $\tilde \psi_0$ in the non-symmetrized case in~\cite[Equation (1)]{BSSZ}). We obtain that~\eqref{eq:Prof-Dil-Step7} is equal to 
\begin{align} \label{eq:Prof-Dil-Step8}
	& \int_{\overline{\mathcal{M}}_{g,n+2}}  \psi_{n+1} s_*  \left[\oM_{g}^{\sim}\left(\mathbb{P}^1,a_1,\dots,a_n,0,-\suma \right)\right]^{\mathrm{vir}}
	\cdot p^* c_{g,n+1}^{[\mu]}(v_{\alpha_1} \otimes \dots \otimes v_{\alpha_n}\otimes v_1) \cdot p^*\Lambda\Big(\frac{1}{\hbar}\Big)  \hbar^g
	\\ \notag &
	+ \sum_{k=1}^\infty  \sum_{\substack{g_1\in \ZZ, g_2 \in \ZZ_{\geq 0} \\ g_1+g_2+k-1 = g  
	}} \sum_{\substack{\bar a_1,\dots,\bar a_k \geq 1 \\ \bar a_1+\cdots+\bar a_k \\ = 
			\suma
	}}  \frac{\hbar^{k-1}}{k!} \prod_{i=1}^k \eta^{\beta_i\gamma_i} \prod_{i=1}^k \bar a_i
	\\ \notag &
	\times \int_{\left[\oM_{g_1}^{\bullet,\sim}\left(\mathbb{P}^1,a_1,\dots,a_n,-\bar a_1,\dots,-\bar a_k \right)\right]^{\mathrm{vir}}}
	\frac{1}{1-\tilde\psi_0} 
	\cdot s^* c^{mu,\bullet}_{g_1,n+k}\left(\bigotimes_{i=1}^n v_{\alpha_i} \otimes \bigotimes_{i=1}^k v_{\beta_i}\right) \cdot s^* \Lambda^\bullet(\frac{1}{\hbar})\hbar^{g_1}
	\\ \notag &
	\times \int_{\overline{\mathcal{M}}_{g_2,k+2}}  \psi_{k+1} s_* \left( \left[\oM_{g_2}^{\sim}\left(\mathbb{P}^1,\bar a_1,\dots,\bar a_k,0,-\suma \right)\right]^{\mathrm{vir}}\right)
	\cdot p^* c_{g_2,k+1}^{[\mu]}(v_{\gamma_1} \otimes \dots \otimes v_{\gamma_k}\otimes v_1) \cdot p^* \Lambda\Big(\frac{1}{\hbar}\Big)\hbar^{g_2}.
\end{align}
This we can rewrite as
 \begin{align} \label{eq:Prof-Dil-Step9}
 	& \int_{\overline{\mathcal{M}}_{g,n+2}}  \psi_{n+1} p^*s_*  \left[\oM_{g}^{\sim}\left(\mathbb{P}^1,a_1,\dots,a_n,-\suma \right)\right]^{\mathrm{vir}}
 	\cdot p^* c_{g,n+1}^{[\mu]}(v_{\alpha_1} \otimes \dots \otimes v_{\alpha_n}\otimes v_1) \cdot p^* \Lambda\Big(\frac{1}{\hbar}\Big)\hbar^g
 	\\ \notag &
 	+ \sum_{k=1}^\infty  \sum_{\substack{g_1\in \ZZ, g_2 \in \ZZ_{\geq 0} \\ g_1+g_2+k-1 = g  
 	}} \sum_{\substack{\bar a_1,\dots,\bar a_k \geq 1 \\ \bar a_1+\cdots+\bar a_k \\ = 
 			\suma
 	}}  \frac{\hbar^{k-1}}{k!} \prod_{i=1}^k \eta^{\beta_i\gamma_i} \prod_{i=1}^k \bar a_i
 	\\ \notag &
 	\times \int_{\left[\oM_{g_1}^{\bullet,\sim}\left(\mathbb{P}^1,a_1,\dots,a_n,-\bar a_1,\dots,-\bar a_k \right)\right]^{\mathrm{vir}}}
 	\frac{1}{1-\tilde\psi_0} 
 	\cdot s^* c^{[\mu],\bullet}_{g_1,n+k}\left(\bigotimes_{i=1}^n v_{\alpha_i} \otimes \bigotimes_{i=1}^k v_{\beta_i}\right) \cdot s^* \Lambda^\bullet(\frac{1}{\hbar})\hbar^{g_1}
 	\\ \notag &
 	\times \int_{\overline{\mathcal{M}}_{g_2,k+2}}  \psi_{k+1} p^*s_* \left( \left[\oM_{g_2}^{\sim}\left(\mathbb{P}^1,\bar a_1,\dots,\bar a_k,-\suma \right)\right]^{\mathrm{vir}}\right)
 	\cdot p^* c_{g_2,k+1}^{[\mu]}(v_{\gamma_1} \otimes \dots \otimes v_{\gamma_k}\otimes v_1) \cdot p^* \Lambda\Big(\frac{1}{\hbar}\Big)\hbar^{g_2},
 \end{align}
and then use the projection formula for the $\psi$-class to obtain that \eqref{eq:Prof-Dil-Step7} is equal to 
 \begin{align} \label{eq:Prof-Dil-Step10}
	& (2g-1+n)\int_{\overline{\mathcal{M}}_{g,n+1}} s_*  \left[\oM_{g}^{\sim}\left(\mathbb{P}^1,a_1,\dots,a_n,-\suma \right)\right]^{\mathrm{vir}}
	\cdot c_{g,n+1}^{[\mu]}(v_{\alpha_1} \otimes \dots \otimes v_{\alpha_n}\otimes v_1) \cdot  \Lambda\Big(\frac{1}{\hbar}\Big)\hbar^g
	\\ \notag &
	+ \sum_{k=1}^\infty  \sum_{\substack{g_1\in \ZZ, g_2 \in \ZZ_{\geq 0} \\ g_1+g_2+k-1 = g  
	}} \sum_{\substack{\bar a_1,\dots,\bar a_k \geq 1 \\ \bar a_1+\cdots+\bar a_k \\ = 
			\suma
	}}  \frac{\hbar^{k-1}}{k!} \prod_{i=1}^k \eta^{\beta_i\gamma_i} \prod_{i=1}^k \bar a_i
	\\ \notag &
	\times \int_{\left[\oM_{g_1}^{\bullet,\sim}\left(\mathbb{P}^1,a_1,\dots,a_n,-\bar a_1,\dots,-\bar a_k \right)\right]^{\mathrm{vir}}}
	\frac{1}{1-\tilde\psi_0} 
	\cdot s^* c^{[\mu],\bullet}_{g_1,n+k}\left(\bigotimes_{i=1}^n v_{\alpha_i} \otimes \bigotimes_{i=1}^k v_{\beta_i}\right) \cdot s^* \Lambda^\bullet(\frac{1}{\hbar})\hbar^{g_1}
	\\ \notag &
	\times (2g_2-1+k)\int_{\overline{\mathcal{M}}_{g_2,k+2}} s_* \left( \left[\oM_{g_2}^{\sim}\left(\mathbb{P}^1,\bar a_1,\dots,\bar a_k,-\suma \right)\right]^{\mathrm{vir}}\right)
	\cdot  c_{g_2,k+1}^{[\mu]}(v_{\gamma_1} \otimes \dots \otimes v_{\gamma_k}\otimes v_1) \cdot \Lambda\Big(\frac{1}{\hbar}\Big)\hbar^{g_2}.
\end{align}
In the second summand of this formula we recognize the evaluation of $(2g-1+n)\tilde \psi_0$ that doesn't use a point of multiplicity $0$ (cf.~\cite[Proof of Theorem 4]{BSSZ}, more precisely, the expression for $\tilde \psi_0$ in the symmetrized case in~\cite[Equation (2)]{BSSZ}). Thus  \eqref{eq:Prof-Dil-Step10} can be assembled into 
 \begin{align} \label{eq:Prof-Dil-Step11}
	& (2g-1+n)\int_{\overline{\mathcal{M}}_{g,n+1}} s_* \left( \frac{1}{1-\tilde \psi_0} \left[\oM_{g}^{\sim}\left(\mathbb{P}^1,a_1,\dots,a_n,-\suma \right)\right]^{\mathrm{vir}} \right)
	\\ \notag & \qquad \qquad \qquad \qquad
	\cdot c_{g,n+1}^{[\mu]}(v_{\alpha_1} \otimes \dots \otimes v_{\alpha_n}\otimes v_1) \cdot  \Lambda\Big(\frac{1}{\hbar}\Big)\hbar^g
	\\ \notag & = (2g-1+n)\int_{\overline{\mathcal{M}}_{g,n+1}} s_* \left( \frac{ \hbar^g\Lambda\Big(\frac{1}{\hbar}\Big)}{1-\tilde \psi_0} \left[\oM_{g}^{\sim}\left(\mathbb{P}^1,a_1,\dots,a_n,-\suma \right)\right]^{\mathrm{vir}} \right)
	\cdot c_{g,n+1}^{[\mu]}(v_{\alpha_1} \otimes \dots \otimes v_{\alpha_n}\otimes v_1),
\end{align}
which is precisely the right hand side of Equation~\eqref{eq:Dilaton-Reduced-A1}. This completes the proof of \Cref{thm:StringANDDilaton}.


\bigskip

\section{Proof of technical statements}

\label{sec:TechnicalProofs}

\subsection{Polynomiality of the \texorpdfstring{$A$}{A}-class}

\label{sec:polyA}

In this section we prove \Cref{thm:polynomiality-A}. We first proof the polynomiality statements for $s_{*}\left(\tilde{\psi}_{0}^{d}\left[\overline{\mathcal{M}}_{g}^{\sim}\left(\mathbb{P}^{1},a_{1},\dots,a_{n},-\suma\right)\right]\right)$ which rely on the following lemma obtained by Buryak and Rossi. We then prove that the push-forward by the forgetful map forgetting the last marked point $\pi:\overline{\mathcal{M}}_{g,n+1}\rightarrow\overline{\mathcal{M}}_{g,n}$ is  divisible by $\suma$. 

\begin{lem}[\cite{BR16-quantum}, Appendix~A]\label{lem:multi:Faulhaber}

Let $N$ be a nonnegative integer. Let $R$ be a ring, and let $P \in R[k_1, \dots, k_q]$ be a polynomial which is odd or even depending on the parity of its degree $\deg P$. Then
\begin{equation}
	\sum_{\substack{k_1,\dots,k_q\geq0\\k_1 + \dots + k_q = N}} k_1 \cdots k_q P(k_1, \dots, k_q)
\end{equation}
is a polynomial in $N$, with coefficient in $R$, of degree $2q-1+\deg P$,  which is odd or even depending on its degree. 
\end{lem}

We start by expressing $\psi_{0}^{d}$ as a sum of strata in $LM_{2g-2+n}$. For each list of positive integers $r_{1},\dots,r_{d+1}$ such that $\sum r_{i}=2g-2+n$, we denote by $D_{r_{1},r_{2},\dots,r_{d+1}}$ the boundary strata in $LM_{2g-2+n}$ composed of $d+1$ components, where we number the components such that the first component contains $\infty$, the second component is attached to the first and so on, and such that the $i$th component contains $r_{i}$ marked points (excluding the points $0$ and $\infty$ on each component). We have
\begin{equation}
\label{eq:psi-as-strata}
\psi_{0}^{d}=\sum_{r_{1}+\cdots+r_{d+1}=2g-2+n}D_{r_{1},r_{2},\dots,r_{d+1}}\cdot\frac{r_{1}}{r_{1}+\cdots+r_{d+1}}\frac{r_{2}}{r_{2}+\cdots+r_{d+1}}\cdots\frac{r_{d+1}}{r_{d+1}}.
\end{equation}
We prove by induction on $d$ that 
\begin{equation}
	\label{eq:divisor-pulled-pushed}
	s_{*}\Big(t^{*} D_{r_{1},r_{2},\dots,r_{d+1}}\cdot  \left[\overline{\mathcal{M}}_{g}^{\sim}\left(\mathbb{P}^{1},a_{1},\dots,a_{n},-\suma\right)\right]^{virt}\Big),
\end{equation}
for $a_i > 0$, is represented by a polynomial in $a_{1},\dots,a_{n}$ of degree $2g+d$, where we recall that $t:\overline{\mathcal{M}}_{g}^{\sim}\left(\mathbb{P}^{1},a_{1},\dots,a_{n},-\suma\right) \rightarrow LM_{2g-2+n}$ and $s:\overline{\mathcal{M}}_{g}^{\sim}\left(\mathbb{P}^{1},a_{1},\dots,a_{n},-\suma\right)\rightarrow\overline{\mathcal{M}}_{g,n+1}$. 
If $d=0$, the contribution of (\ref{eq:divisor-pulled-pushed}) is $DR_{g}\left(a_{1},\dots,a_{n},-\suma\right)$ which is an even polynomial of degree $2g$ by \cite{PZ,Pim}. 
Suppose now $d\geq 0$. Since we apply $s_{*}$ in (\ref{eq:divisor-pulled-pushed}), it follows from \cite[Lemma~2.3]{BSSZ} that the only strata of $\overline{\mathcal{M}}_{g}^{\sim}\left(\mathbb{P}^{1},a_{1},\dots,a_{n},-\suma\right)$ contributing are those with a unique stable component in the preimage of each of the $d+1$ components of the target. Isolating the contribution of the last component, we get 
\begin{align}
\label{eq:induction-divisor}
& s_{*}\Big(t^{*}  D_{r_{1},r_{2},\dots,r_{d+1}}\cdot  \left[\overline{\mathcal{M}}_{g}^{\sim}\left(\mathbb{P}^{1},a_{1},\dots,a_{n},-\suma\right)\right]^{virt}\Big)  \\ \notag &  =\sum_{\substack{g_1,g_2\geq0, \,p\geq1 \\ g_{1}+g_{2}+p-1=g}}\sum_{I_{1}\sqcup I_{2}=\left\{ 1,\dots,n\right\}}\sum_{\substack{k_1,\dots,k_p\geq0 \\ \notag k_{1}+\cdots+k_{p}=\left|a_{I_{1}}\right|}}\frac{k_{1}\cdots k_{p}}{p!}\times \nonumber
\\ \notag
 & \qquad \times b_{*}\left(DR_{g_{1}}(a_{I_{1}},-k_{1},\dots,-k_{p})\otimes s_{*}\Big(t^{*}\left(D_{r_{2},\dots,r_{d+1}}\right)\cdot\Big[\overline{\mathcal{M}}_{g_{2}}^{\sim}\left(\mathbb{P}^{1},a_{I_{2}},k_{1},\dots,k_{p},-\suma\right)\Big]^{virt}\Big)\right),
\end{align}
where $p$ counts the number of nodes between the last component and other components, the order of the ramifications at these nodes is given by $k_{1},\dots,k_{p}$, the set $I_{1}$ contains the marked points on the last component, moreover we impose the condition $2g_{1}-2+\left|I_{1}\right|+p=r_{1}$ on the genus and number of special points of the last component. The map $b:\overline{\mathcal{M}}_{g_{1},\left|I_{1}\right|+p}\times\overline{\mathcal{M}}_{g_{2},\left|I_{2}\right|+p+1}\rightarrow\overline{\mathcal{M}}_{g,n+1}$ is the glueing map. The DR cycle is a polynomial in $k_1\dots,k_p$ and $(a_i)_{i \in I_1}$ of degree $2g_1$. By induction, $s_{*}\left(t^{*}\left(D_{r_{2},\dots,r_{d+1}}\right)\cdot\overline{\mathcal{M}}_{g_{2}}\left(\mathbb{P}^{1},a_{I_{2}},k_{1},\dots,k_{p},-\suma\right)\right)$ is a polynomial in $k_{1},\dots,k_{p}$ and $\left(a_{i}\right)_{i\in I_{2}}$ of degree $2g+d-1$. Therefore, it follows from \Cref{lem:multi:Faulhaber} that the right hand side of \Cref{eq:induction-divisor} is a polynomial in $a_1,\dots,a_n$ of degree $2g+d$.

\smallskip

We now show that the polynomial
\begin{equation} \label{eq:PushforwardPol-DivisibleByA}
	\pi_*s_{*}\left(\tilde{\psi}_{0}^{d}\left[\overline{\mathcal{M}}_{g}^{\sim}\left(\mathbb{P}^{1},a_{1},\dots,a_{n},-\suma\right)\right]^{virt}\right)
\end{equation}
is divisible by $\suma$, where $\pi:\overline{\mathcal{M}}_{g,n+1}\rightarrow\overline{\mathcal{M}}_{g,n}$ is the forgetful map forgetting the last marked point. 
We proceed by induction on $d$. 
The base of the induction is clear: for $d=0$ we get a DR cycle and use  \Cref{prop:DR-pushforward-divisible-bb}.
For the step of induction, we use the computation close to the one in the proof of \Cref{prop:graph-dr-rubber}: we evaluate $\tilde \psi_0$ as in \eqref{eq:psi-tilde-evaluated} and then further split the $D_1$ term in this evaluation as $D_1 = (\suma-a_1)\Delta_{1}+D'_{1}$, where
\begin{align}
\Delta_1\coloneqq 
\left[\overline{\mathcal{M}}_{g}^{\sim}\left(\mathbb{P}^{1},a_{2},\dots,a_n,-\suma+a_1\right)\right]^{vir}\boxtimes\left[\overline{\mathcal{M}}_{0}^{\sim}\left(\mathbb{P}^{1},a_{1},\suma-a_1,-\suma\right)\right]^{vir}	
\end{align}
and $D'_{1}$ is the weighted sum of all other irreducible divisors entering $D_1$ according to~\eqref{eq:irred-divisor}. We note that $\pi_*s_*(\tilde \psi_0^{d-1} D'_{1})$ is given as 
\begin{align}
	\label{eq:D1'-term}
	& \pi_*s_*\bigg(\sum_{\substack{g_1,g_2\geq0, \,p\geq1 \\ g_{1}+g_{2}+p-1=g}}\sum_{\substack{I_{1}\sqcup I_{2}=\left\{ 2,\dots,n\right\} \\ (g_2,2+|I_2|+p)\\ \not=(0,3)}}\sum_{\substack{k_1,\dots,k_p\geq0 \\ k_{1}+\cdots+k_{p}=\left|a_{I_{1}}\right|}} \frac{\prod_{i=1}^p k_i}{p!}\, \tilde\psi_0^{d-1}\left[\overline{\mathcal{M}}_{g_{1}}^{\sim}\left(\mathbb{P}^{1},a_{I_1},-k_{1},\dots,-k_{p}\right)\right]^{vir}\boxtimes 
	\\ \notag & \qquad \qquad \qquad \qquad \qquad
	\left[\overline{\mathcal{M}}_{g_{2}}^{\sim}\left(\mathbb{P}^{1},a_{1},a_{I_2},k_{1},\dots,k_{p},-\suma\right)\right]^{vir}\bigg).
\end{align}
This expression is a polynomial in the $a_i$'s since it is a sum of polynomials in these variables (see the second equality in (\ref{eq:proving-divisibility}) below). Note that this does not directly follow from \Cref{lem:multi:Faulhaber} as it is not proved that $s_*\Big(\tilde\psi_0^{d-1}\left[\overline{\mathcal{M}}_{g_{1}}^{\sim}\left(\mathbb{P}^{1},a_{I},-k_{1},\dots,-k_{p}\right)\right]^{vir}\Big)$ is polynomial. After substituting $a_1:=\suma-\sum_{i \neq 1} a_i$ in (\ref{eq:D1'-term}), we get a polynomial in the variables $a_2,\dots,a_n,\suma$. Moreover, the only dependency in the variable $\suma$ is through the polynomial  $\pi_* s_*\left[\overline{\mathcal{M}}_{g_{2}}^{\sim}\left(\mathbb{P}^{1},a_{1},a_{J},k_{1},\dots,k_{p},-\suma\right)\right]^{vir}$, which by \Cref{prop:DR-pushforward-divisible-bb} only involves $\suma^2,\suma^3,\dots$. In particular, each coefficient of $a^k$ in (\ref{eq:D1'-term}), for $k\geq 2$, is a polynomial in $a_2,\dots,a_n$. Therefore, after the inverse substitution $\suma:=\sum_{i=1}^n a_i$, we get that (\ref{eq:D1'-term}) is a polynomial in $a_1,\dots,a_n$ in divisible by  $\suma=\sum_{i=1}^n a_i$.

%
%
%
%
%
%

Now note that on $\oM_{g,n+1}$ we have $\psi_1 = \pi^*\psi_1 + B_{1}$, where $B_1$ is the divisor of reducible curves with one component of genus $0$ with the points $1$ and $n+1$ on this component. Note that 
\begin{align}
B_1\cdot s_*\left[\overline{\mathcal{M}}_{g}^{\sim}\left(\mathbb{P}^{1},a_{1},\dots,a_{n},-\suma\right)\right]^{vir} = s_*\Delta_{1}	
\end{align}
Thus 
\begin{align}
\label{eq:proving-divisibility}
	& \pi_*s_* \bigg(\tilde \psi_0^d \left[\overline{\mathcal{M}}_{g}^{\sim}\left(\mathbb{P}^{1},a_{1},\dots,a_{n},-\suma\right)\right]^{vir}\bigg) 
	\\ \notag &
	= 	\pi_*s_* \bigg(\tilde \psi_0^{d-1} \Big( a_1s^*\psi_1 \left[\overline{\mathcal{M}}_{g}^{\sim}\left(\mathbb{P}^{1},a_{1},\dots,a_{n},-\suma\right)\right]^{vir}+ D_1\Big)\bigg) 
	\\ \notag &
	= 	\pi_*s_* \bigg(\tilde \psi_0^{d-1} \Big( a_1s^*\psi_1 \left[\overline{\mathcal{M}}_{g}^{\sim}\left(\mathbb{P}^{1},a_{1},\dots,a_{n},-\suma\right)\right]^{vir}+(\suma-a_1)\Delta_1 + D'_1\Big)\bigg) 
	\\ \notag
	& = \pi_*s_* \bigg(\tilde \psi_0^{d-1} \Big( a_1s^*\pi^*\psi_1 \left[\overline{\mathcal{M}}_{g}^{\sim}\left(\mathbb{P}^{1},a_{1},\dots,a_{n},-\suma\right)\right]^{vir}+a_1\Delta_1 + (\suma-a_1)\Delta_1 + D'_1\Big)\bigg) 
	\\ \notag
	& = a_1\psi_1 \pi_*s_* \bigg(\tilde \psi_0^{d-1}  \left[\overline{\mathcal{M}}_{g}^{\sim}\left(\mathbb{P}^{1},a_{1},\dots,a_{n},-\suma\right)\right]^{vir}\bigg)
	\\ \notag & \qquad +\suma \, s_* \bigg(\tilde \psi_0^{d-1}  \left[\overline{\mathcal{M}}_{g}^{\sim}\left(\mathbb{P}^{1},a_2\dots,a_{n},-\suma+a_1\right)\right]^{vir}\bigg) + \pi_*s_* \bigg(\tilde\psi_0^{d-1}D'_1\bigg).
\end{align}
Here all three terms are polynomials in $a_1,\dots,a_n$, and the first one is divisible by $\suma$ by the induction assumption, the second term has $\suma$ as a factor, and the last term is divisible by $\suma$ by \Cref{prop:DR-pushforward-divisible-bb}, as we discussed above. 

This completes the proof of  \Cref{thm:polynomiality-A}.

\medskip

\subsection{Compatibility of Equations (\ref{eq:constant-two point function})}\label{app: compatibility constant term two point functions}
\label{app:sec:compatibility}

In this section, we prove \Cref{lem:compatibility}, that is  the system of Equations~(\ref{eq:constant-two point function}) has a unique set of solutions $\Omega_{p,\alpha;q,\beta}\vert_{p_{*}^{*}=0}$, for $p,q\geq0$ and $1\leq\alpha,\beta\leq N$. First it is straightforward to show that if a solution exists, it must be unique: Equations~(\ref{eq:constant-two point function}) directly determine $\Omega_{0,\alpha;q,\beta}\vert_{p_{*}^{*}=0}$ for $q\geq0$ and $1\leq\alpha,\beta\leq N$, then, they inductively  fix $\Omega_{1,\alpha;q,\beta}\vert_{p_{*}^{*}=0}$ for $q\geq0$ and $1\leq\alpha,\beta\leq N$, and this process continues. The compatibility of Equations~(\ref{eq:constant-two point function}) is more subtle, it boils down to prove
\begin{equation}
\sum_{p+q=r}\left(-1\right)^{q}\frac{\partial\Omega_{p,\alpha;q,\beta}}{\partial p_{0}^{1}}\Bigg|_{p_{*}^{*}=0}=0\quad{\rm for}\;r\geq1\;{\rm and}\;1\leq\alpha,\beta\leq N.\label{eq: compatibility equation}
\end{equation}

\begin{lem}
\label{lem:derivative-two-point}
Let $p,q\geq0$ and $1\leq\alpha,\beta\leq N$. We have 
\begin{align}
\left.\frac{\partial\Omega_{p,\alpha;q,\beta}}{\partial p_{0}^{1}}\right|_{p_{*}^{*}=0} & ={\rm \Coeff}_{a}\Bigg(\sum_{g\geq0}\left(i\hbar\right)^{g}\sum_{\substack{g_{1},g_{2}\geq0 \,n\geq1\\
g_{1}+g_{2}+n-1=g\\
2g_{1}+n>0,\,2g_{2}+n>0
}
}\sum_{k_{1},\dots,k_{n}>0}\frac{k_{1}\cdots k_{n}}{n!}\prod_{i=1}^{n}\eta^{\mu_{i}\nu_{i}}\label{eq: derivative Omega wrt p}\\
 & \times\int_{{\rm DR}_{g_{1}}\left(0,k_{1},\dots,k_{n},-a\right)}\psi_{1}^{p}\Lambda\left(-\frac{\epsilon^{2}}{i\hbar}\right)c_{g_1,n+2}^{[\mu]}\left(v_{\alpha},v_{\mu_{1}},\dots,v_{\mu_{n}},v_{1}\right)\nonumber \\
 & \times\int_{{\rm DR}_{g_{2}}\left(0,-k_{1},\dots,-k_{n},a\right)}\psi_{1}^{q}\Lambda\left(-\frac{\epsilon^{2}}{i\hbar}\right)c_{g_2,n+2}^{[\mu]}\left(v_{\beta},v_{\nu_{1}},\dots,v_{\nu_{n}},v_{1}\right)\Bigg).\nonumber 
\end{align}
\end{lem}

\begin{proof}
For a differential polynomial $f$ such that $\partial_{x}f=g$, one can check that \begin{align}
	\frac{\partial f}{\partial p_{0}^{1}}\Big|_{p_{*}^{*}=0}={\rm \Coeff}_{a}\left({\rm \Coeff}_{p_{a}^{1}e^{iax}}-ig\right).
\end{align}
In particular, we have 
\begin{equation}
\left.\frac{\partial\Omega_{p,\alpha;q,\beta}}{\partial p_{0}^{1}}\right|_{p_{*}^{*}=0}={\rm \Coeff}_{a}\left({\rm \Coeff}_{p_{a}^{1}e^{iax}}\frac{1}{i\hbar}\left[H_{p-1,\alpha},\overline{H}_{q,\beta}\right]\right).
\end{equation}
Then, introduce $f\tilde{\star}g:=f\star g-fg$. Choosing $a>0$, we have 
\begin{equation}
{\rm \Coeff}_{p_{a}^{1}e^{iax}}\overline{H}_{q,\beta}\tilde{\star}H_{p-1,\alpha}=0.
\end{equation}
This vanishing happens for the following reason. The term $\overline{H}_{q,\beta}\tilde{\star}\left(\cdots\right)$ involves a product of derivative $\frac{\partial}{\partial p_{k_{1}}^{\alpha_{1}}}\cdots\frac{\partial}{\partial p_{k_{s}}^{\alpha_{s}}}$, with $k_{1},\dots,k_{s}>0$, from the star product which are acting on $\overline{H}_{q,\beta}$. Since we use the $\tilde{\star}$-product, we have $s\geq1$. Consequently, when extracting the coefficient of $p_{a}^{1}e^{aix}$ with the condition $a>0$, the sum of the parts of each DR-cycle from $\overline{H}_{q,\beta}$ is positive, contradicting the requirement that they sum to zero. Due to this vanishing, we get 
\begin{equation}
{\rm \Coeff}_{p_{a}^{1}e^{iax}}\frac{1}{i\hbar}\left[H_{p-1,\alpha},\overline{H}_{q,\beta}\right]={\rm \Coeff}_{p_{a}^{1}e^{iax}}\frac{1}{i\hbar}\left(H_{p-1,\alpha}\tilde{\star}\overline{H}_{q,\beta}\right).
\end{equation}
Using the same argument, we see that this last expression vanishes if the variable $p_{a}^{1}$ is taken from $H_{p-1,\alpha}$, and therefore it belongs to $\overline{H}_{q,\beta}$. Writing explicitly the definition of the hamiltonian density, the hamiltonian,  the star product, and extracting the coefficient of $p_{a}^{1}e^{iax}$, we get \Cref{eq: derivative Omega wrt p}. 
\end{proof}
Now using the the push forward property of the DR cycles (\Cref{cor: push forward DR numbers}) we obtain for $p,q\geq1$:
\begin{align}
\left.\frac{\partial\Omega_{p,\alpha;q,\beta}}{\partial p_{0}^{1}}\right|_{p_{*}^{*}=0} & ={\rm \Coeff}_{a}\Bigg(\sum_{g\geq0}i\left(i\hbar\right)^{g}\sum_{\substack{g_{1},g_{2}\geq0,\,n\geq1\\
g_{1}+g_{2}+n-1=g\\
2g_{1}-1+n>0,\,2g_{2}-1+n>0
}
}\sum_{k_{1},\dots,k_{n}>0}\frac{k_{1}\cdots k_{n}}{n!}\prod_{i=1}^{n}\eta^{\mu_{i}\nu_{i}}\nonumber \\
 & \times\int_{{\rm DR}_{g_{1}}\left(-a,k_{1},\dots,k_{n}\right)}\psi_{1}^{p-1}\Lambda\left(-\frac{\epsilon^{2}}{i\hbar}\right)c_{g_1,n+1}^{[\mu]}\left(v_{\alpha},v_{\mu_{1}},\dots,v_{\mu_{n}}\right)\\
 & \times\int_{{\rm DR}_{g_{2}}\left(-a,-k_{1},\dots,-k_{n}\right)}\psi_{1}^{q-1}\Lambda\left(-\frac{\epsilon^{2}}{i\hbar}\right)c_{g_2,n+1}^{[\mu]}\left(v_{\beta},v_{\nu_{1}},\dots,v_{\nu_{n}}\right)\Bigg),\nonumber 
\end{align}
and when $p=0$ and $q\geq1$, we get
\begin{align}
\left.\frac{\partial\Omega_{0,\alpha;q,\beta}}{\partial p_{0}^{1}}\right|_{p_{*}^{*}=0}=\left.\frac{\partial\Omega_{q,\beta;0,\alpha}}{\partial p_{0}^{1}}\right|_{p_{*}^{*}=0} & ={\rm \Coeff}_{a}\Bigg(\sum_{g\geq0}i\left(i\hbar\right)^{g}a\int_{{\rm DR}_{g}\left(-a,a\right)}\psi_{1}^{q-1}\Lambda\left(-\frac{\epsilon^{2}}{i\hbar}\right)c_{g,2}^{[\mu]}\left(v_{\alpha},v_{\beta}\right)\Bigg).
\end{align}
Using these above expressions, we obtain
\begin{align}
\sum_{p+q=r}\left(-1\right)^{q}\frac{\partial\Omega_{p,\alpha;q,\beta}}{\partial p_{0}^{1}}\Bigg|_{p_{*}^{*}=0} & ={\rm \Coeff}_{a}\left(\sum_{g\geq0}i\left(i\hbar\right)^{g}\Phi_{g}^{r}\right),
\end{align}
where
\begin{align}
\Phi_{g}^{r} & =a\int_{{\rm DR}_{g}\left(-a,a\right)}\psi_{1}^{r-1}\Lambda\left(-\frac{\epsilon^{2}}{i\hbar}\right)c_{g,2}^{[\mu]}\left(v_{\alpha},v_{\beta}\right)+\left(-1\right)^{r}a\int_{{\rm DR}_{g}\left(-a,a\right)}\psi_{1}^{r-1}\Lambda\left(-\frac{\epsilon^{2}}{i\hbar}\right)c_{g,2}^{[\mu]}\left(v_{\beta},v_{\alpha}\right)\nonumber \\
 & +\Bigg(\sum_{\substack{p+q=r-2\\
p,q\geq0
}
}\left(-1\right)^{q+1}\sum_{\substack{g_{1},g_{2},n\geq0\\
g_{1}+g_{2}+n-1=g\\
2g_{1}-1+n>0,\,2g_{2}-1+n>0
}
}\sum_{k_{1},\dots,k_{n}>0}\frac{k_{1}\cdots k_{n}}{n!}\prod_{i=1}^{n}\eta^{\mu_{i}\nu_{i}}\\
 & \times\int_{{\rm DR}_{g_{1}}\left(-a,k_{1},\dots,k_{n}\right)}\psi_{1}^{p}\Lambda\left(-\frac{\epsilon^{2}}{i\hbar}\right)c_{g_1,n+1}^{[\mu]}\left(v_{\alpha},v_{\mu_{1}},\dots,v_{\mu_{n}}\right)\nonumber \\
 & \times\int_{{\rm DR}_{g_{2}}\left(-a,-k_{1},\dots,-k_{n}\right)}\psi_{1}^{q}\Lambda\left(-\frac{\epsilon^{2}}{i\hbar}\right)c_{g_2,n+1}^{[\mu]}\left(v_{\beta},v_{\nu_{1}},\dots,v_{\nu_{n}}\right)\Bigg).\nonumber 
\end{align}
We now prove that $\Phi_{g}^{r}$ vanishes. This comes from the following relation in the Losev-Manin space. 
\begin{lem}
Let $r\geq0$, we have the relation
\begin{equation}
\left(-1\right)^{r+1}\psi_{0}^{r}+\sum_{p+q=r-1}\left(-1\right)^{p-1}\psi_{0}^{p}\Delta\psi_{\infty}^{q}+\psi_{\infty}^{r}=0\label{eq: Losev Manin compatibility}
\end{equation}
in the Losev-Manin space $LM_{m}$, for $m\geq0$, where $\Delta$ represents the sum of boundary divisors corresponding to curves with one separating node, such that the point $0$ lies on one component and the point $\infty$ lies on the other.
\end{lem}

\begin{proof}
The relation is obvious in degree $0$. In degree $1$ it follows from expressing the $\psi$-class in genus $0$ as in \cite[Eq.~(2)~Section~2.2.5]{BSSZ}. The relation in higher degree follows from a simple induction. 
\end{proof}
Denote by $\overline{\mathcal{M}}^{\sim}\left(\mathbb{P}^{1},a,-a\right)$ the moduli space of rubber stable maps to $\left(\mathbb{P}^{1},0,\infty\right)$ with profile over $0$ and $\infty$ given by $\left(a\right)$. Its projection onto the source curve is denoted by 
\begin{align}
s:\overline{\mathcal{M}}_{g}^{\sim}\left(\mathbb{P}^{1},a,-a\right)\rightarrow\overline{\mathcal{M}}_{g,2},
\end{align}
and its projection onto the target curve is denoted by 
\begin{align}
t:\overline{\mathcal{M}}_{g}^{\sim}\left(\mathbb{P}^{1},a,-a\right)\rightarrow LM_{2g}.
\end{align}
It follows from \cite[Proposition~2.5~and~Lemma~2.6]{BSSZ} that 
\[
as^{*}\left(\psi_{1}\right)=a\Psi_{1}=t^{*}\left(\psi_{0}\right),\quad{\rm and}\quad as^{*}\left(\psi_{2}\right)=a\Psi_{2}=t^{*}\left(\psi_{\infty}\right),
\]
where $\Psi_{i}$ is the psi class on $\overline{\mathcal{M}}_{g}^{\sim}\left(\mathbb{P}^{1},a,-a\right)$ at the marked point $i$. Thus, by first pulling back by $t$ \Cref{eq: Losev Manin compatibility}, and then pushing forward by $s$, we obtain a relation in $\overline{\mathcal{M}}_{g,2}$. By intersecting this relation with $\Lambda\left(-\frac{\epsilon^{2}}{i\hbar}\right)$ and $c_{g,2}\left(v_{\alpha},v_{\beta}\right)$, and then extracting the top cohomological degree, we get precisely $\Phi_{g}^{r}=0$. This proves the compatibility condition (\Cref{eq: compatibility equation}).

\appendix

\end{document}